\documentclass[12pt]{amsart}

\pdfoutput=1

\usepackage{amsthm}
\usepackage{amsmath}
\usepackage{amstext}
\usepackage{amssymb}

\usepackage{caption}
\usepackage{verbatim}
\usepackage{setspace}
\usepackage[dvips,letterpaper,margin= 1.3 in]{geometry}

\usepackage{graphicx}
\usepackage{hyperref}
\usepackage{pdflscape}

\usepackage{float}
\usepackage[utf8]{inputenc}
\usepackage[final]{pdfpages}

\newtheorem{theorem}{Theorem}[section]

\newtheorem{remark}[theorem]{Remark}
\newtheorem{proposition}[theorem]{Proposition}

\newtheorem{conjecture}[theorem]{Conjecture}
\newtheorem{definition}[theorem]{Definition}

\newtheorem{problem}[theorem]{Problem}

\title[Dungeons and Dragons]{Dungeons and Dragons: \\ Combinatorics for the $dP_3$ Quiver}

\begin{document}

\author{Tri Lai}
\address{Department of Mathematics, University of Nebraska, Lincoln, Nebraska 68588}
\email{tlai3@unl.edu}

\author{Gregg Musiker}
\address{School of Mathematics, University of Minnesota, Minneapolis, Minnesota 55455}
\email{musiker@math.umn.edu}

\thanks{The second author was supported by NSF Grant DMS-\#13692980.}


\date{April 25, 2019}


\maketitle

\begin{abstract}
In this paper, we utilize the machinery of cluster algebras, quiver mutations, and brane tilings to study a variety of historical enumerative combinatorics questions all under one roof.  Previous work \cite{zhang, LMNT}, which arose during the second author's monitorship of undergraduates, and more recently of both authors \cite{LaiMus}, analyzed the cluster algebra associated to the cone over $\mathbf{dP_3}$, the del Pezzo surface of degree $6$ ($\mathbb{CP}^2$ blown up at three points).  By investigating sequences of toric mutations, those occurring only at vertices with two incoming and two outgoing arrows, in this cluster algebra, we obtained a family of cluster variables that could be parameterized by $\mathbb{Z}^3$ and whose Laurent expansions had elegant combinatorial interpretations in terms of dimer partition functions (in most cases).  While the earlier work \cite{zhang, LMNT, LaiMus} focused exclusively on one possible initial seed for this cluster algebra, there are in total four relevant initial seeds (up to graph isomorphism).  In the current work, we explore the combinatorics of the Laurent expansions from these other initial seeds and how this allows us to relate enumerations of perfect matchings on Dungeons to Dragons.

\end{abstract}

\tableofcontents

\section{Introduction}

Cluster algebras were introduced by Fomin and Zelevinsky in 2001 motivated by their study of total positivity and canonical bases \cite{FZ}.  Since their introduction, deep connections to a variety of topics in mathematics and physics have been found and explored.  Among others, these include quiver representations, hyperbolic geometry, discrete dynamical systems, and string theory.  In this paper, we highlight the relationship between cluster algebras and string theory utilizing the concept of brane tilings \cite{brane_dimer}.  This correspondence yields discrete dynamical systems related to historical problems in enumerative combinatorics dating from the turn of the millenium \cite{Propp}, and also yields a new perspective on solutions to the hexahedron recurrence studied by Kenyon and Pemantle \cite{KP}.

To lay the foundation for describing our results, we start with the definition of quivers and cluster algebras.
A \textbf{quiver} $Q$ is a directed finite graph with a set of vertices $V$ and a set of edges $E$ connecting them whose direction is denoted by an arrow. For our purposes, $Q$ may have multiple edges connecting two vertices but may not contain any loops or $2-$cycles.

\begin{definition}[\textbf{Quiver Mutation}] Mutating at a vertex $i$ in $Q$ is denoted by $\mu_{i}$ and corresponds to the following actions on the quiver:
\begin{itemize}
\item For every 2-path through $i$ (e.g. $j \rightarrow i \rightarrow k$), add an edge from $j$ to $k$.
\item Reverse the directions of the arrows incident to $i$.
\item Delete any 2-cycles created from the previous two steps.
\end{itemize}
\end{definition}
\noindent
We define a {\bf cluster algebra} with initial seed $\{x_{1},x_{2},\ldots,x_{n}\}$ from the quiver $Q$ by associating a cluster variable $x_{i}$ to every vertex labeled $i$ in $Q$ where $|V| = n$.
When we mutate at a vertex $i$, the cluster variable at this vertex is updated and all other cluster variables remain unchanged \cite{FZ}. The action of $\mu_{i}$ leads to the following binomial exchange relation
\begin{equation*}
\label{eq: exchange relation}
x'_{i}x_{i} = \prod_{i \rightarrow j \; \mathrm{in} \; Q}x_{j}^{a_{i \rightarrow j}} + \prod_{j \rightarrow i \; \mathrm{in} \; Q}x_{j}^{b_{j \rightarrow i}}
\end{equation*}
where $x_i'$ is the new cluster variable at vertex $i$, $a_{i \rightarrow j}$ denotes the number of edges from $i$ to $j$, and $b_{j \rightarrow i}$ denotes the number of edges from $j$ to $i$.

The cluster algebra associated to $Q$ is generated by the union of all the cluster variables at each vertex allowing iterations of all finite sequences of mutations at every vertex.  One of the first results in the theory of cluster algebras was the {\bf Laurent Phenomenon} stating that every cluster variable is a Laurent polynomial, i.e. a rational function with a single monomial as a denominator \cite{FZ, laurent}.  In particular, if one starts with an initial cluster of $\{x_1,x_2,\dots, x_n\} = \{1,1,\dots, 1\}$ then all resulting cluster variables are integers (rather than rational numbers) despite the iterated division coming from the definition of cluster variable mutation.

In our previous work \cite{LaiMus}, we discussed a cluster algebra associated to the cone over $\mathbf{dP_3}$, the del Pezzo surface\footnote{All four toric phases of the $\mathbf{dP_3}$ quiver appeared for the first time in \cite[Section 6.3]{feng}, but this required brute force since it preceded the more efficient dimer model technology.}
of degree $6$ ($\mathbb{CP}^2$ blown up at three points) that was one of a number of cluster algebras arising from brane tilings (see \cite{brane_dimer} and the related work \cite{Hanany-Kennaway} and \cite{Quivering} on dimer models).  In particular, we investigated {\bf toric mutations} in such a cluster algebra, i.e. sequences of mutations exclusively at vertices with two incoming and two outgoing arrows.

The Laurent expansions of the corresponding toric cluster variables (i.e. those generators reachable via toric mutations) were given a combinatorial interpretation therein \cite[Theorem 5.9]{LaiMus} assuming that the initial seed was the quiver $Q_1$ as in Figure \ref{fig:quiv_brane} (Middle), see \cite[Figure 1]{brane_dimer} or \cite[Figure 22]{hanany_polygons}.  However, there are three other non-isomorphic seeds that are mutation-equivalent to $Q_1$ via toric mutations, see Figure \ref{fig:dP3QuiverModels}.  These four models are adjacent to each other as illustrated in Figure \ref{fig:ModelConnections} from \cite[Figure 27]{franco_eager}.

\begin{figure}\centering
\includegraphics[width=12cm]{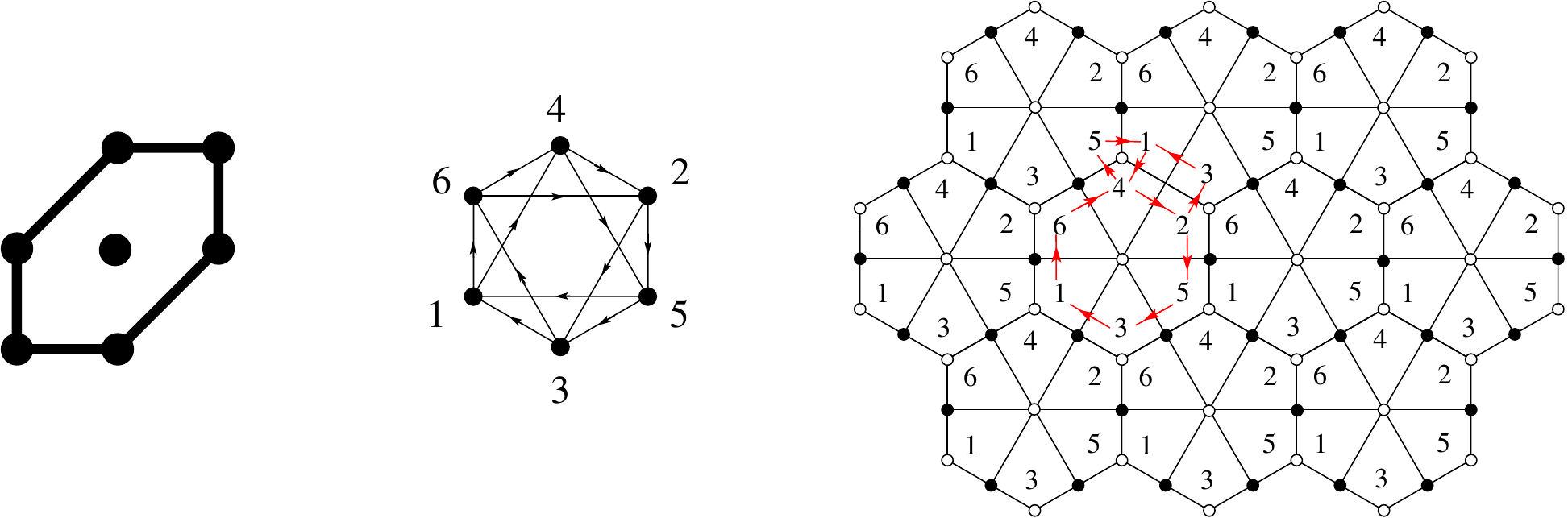}
 \caption{\small The $dP_3$ toric diagram, quiver $Q_1$, and its associated brane tiling $\mathcal{T}_1$. (Figure 1 of \cite{LaiMus}.)}
    \label{fig:quiv_brane}
\end{figure}

\begin{figure}
\includegraphics[width=5.8in]{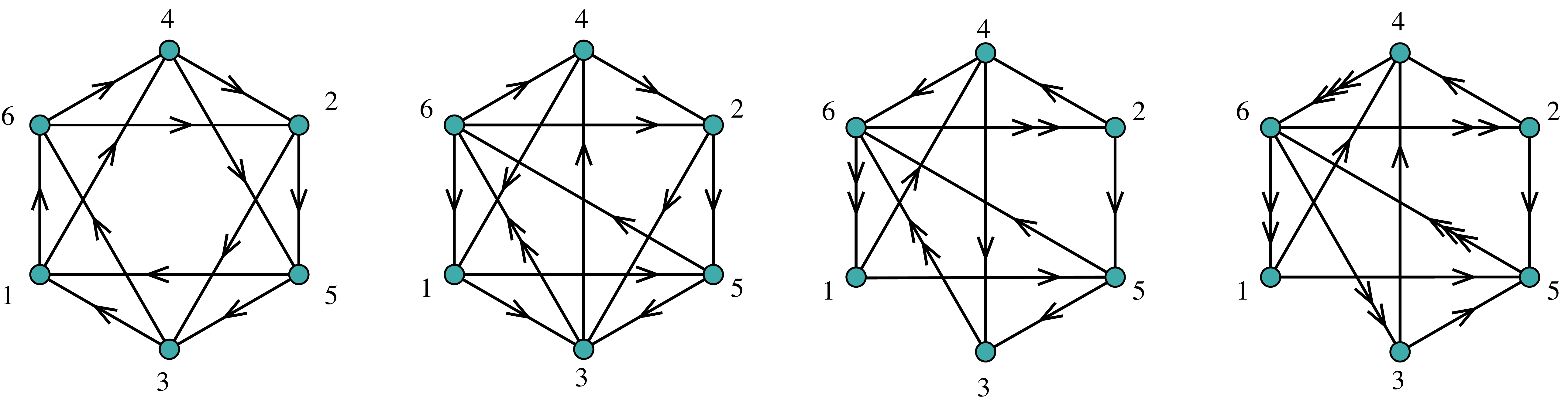}
\caption{Models 1, 2, 3, and 4 of the $dP_3$, i.e. quivers $Q_1$, $Q_2$, $Q_3$, and $Q_4$.  Two quivers are considered to be the same model if they are equivalent under (i) graph isomorphism and (ii) reversal of all edges.}
\label{fig:dP3QuiverModels}
\end{figure}

\begin{figure}\centering
\includegraphics[width=12cm]{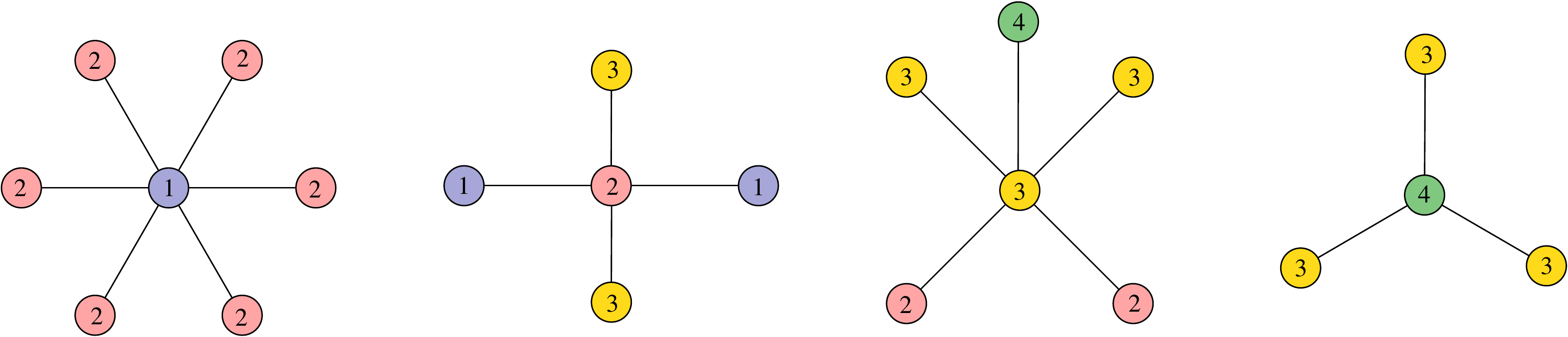}
\caption{Adjacencies between the different models. (Figure 27 of \cite{franco_eager}.)}
\label{fig:ModelConnections}
\end{figure}

In the current paper, our goal is to provide analogous combinatorial interpretations for Laurent expansions of toric cluster variables for three other possible initial seeds.  We refer to these three other initial seeds as Models 2, 3, and 4, and together with Model 1, they comprise the set of all quivers (up to graph isomorphism or reversal of all arrows) that are reachable from $Q_1$, i.e. Model 1, by toric mutations.   All four of these quivers are also associated to the cone over  $\mathbf{dP_3}$.

In addition to the data of a quiver, one can also associate a {\bf potential}, which is a linear combination of cycles of the quiver \cite{DWZ}.  Taken together, the data of a quiver and a potential yields a {\bf brane tiling}, i.e. a tesselation of a torus, for each of these four models.  For example, in \cite{LaiMus}, we studied the quiver $Q_1$ with the potential
\begin{eqnarray} \label{eq:pot} W_1 &=& A_{16}A_{64}A_{42}A_{25}A_{53}A_{31}
~+~ A_{14}A_{45}A_{51} ~+~ A_{23}A_{36}A_{62} \\ \nonumber &-&  A_{16}A_{62}A_{25}A_{51}
~-~ A_{36}A_{64}A_{45}A_{53} ~-~ A_{14}A_{42}A_{23}A_{31}.\end{eqnarray}
The pair $(Q_1,W_1)$ yields the brane tiling $\mathcal{T}_1$ as illustrated in Figure \ref{fig:quiv_brane} (Right) by unfolding the quiver in a periodic way that preserves the cycles arising in the potential and then taking the dual graph.  Following the rules laid out in \cite{brane_dimer} (which is a special case of the construction in \cite{DWZ}), the quiver with potential $(Q_1,W_1)$ may be mutated to yield not only mutation-equivalent quivers, but accompanying potentials.  In particular, Models 2, 3, and 4 correspond to the quivers with potentials $(Q_2,W_2)$, $(Q_3,W_3)$ and $(Q_4,W_4)$ by mutating at vertices $1$, $4$, then $3$ respectively, where
\begin{eqnarray} \label{eq:pot} W_2 &=& A_{36}^{(2)}A_{64}A_{42}A_{25}A_{53}
~+~ A_{23}A_{36}^{(1)}A_{62} + A_{34}A_{41}A_{13} + A_{56}A_{61}A_{15} \\ \nonumber &-&  A_{56}A_{62}A_{25}
~-~ A_{53}A_{36}^{(1)}A_{64}A_{41}A_{15} ~-~ A_{34}A_{42}A_{23}
 - A_{36}^{(2)}A_{61}A_{13}
.\end{eqnarray}

\begin{eqnarray} \label{eq:pot} W_3 &=& A_{36}^{(2)}A_{62}^{(2)}A_{25}A_{53}
~+~  A_{56}A_{61}^{(1)}A_{15} +
A_{24}A_{43}A_{36}^{(1)}A_{62}^{(1)} + A_{61}^{(2)}A_{14}A_{46}  \\ \nonumber &-&  A_{56}A_{62}^{(1)}A_{25}
~-~ A_{53}A_{36}^{(1)}A_{61}^{(2)}A_{15}
 - A_{36}^{(2)}A_{61}^{(1)}A_{14}A_{43} - A_{62}^{(2)}A_{24}A_{46}
  .\end{eqnarray}

{\scriptsize \begin{eqnarray} \label{eq:pot} W_4 &=& A_{56}^{(2)} A_{62}^{(2)}A_{25} + A_{46}^{(1)} A_{62}^{(1)} A_{24}
+ A_{56}^{(3)} A_{61}^{(1)} A_{15} + A_{61}^{(2)} A_{14}A_{46}^{(3)} + A_{46}^{(2)}A_{63}^{(2)} A_{34} +
A_{56}^{(1)} A_{63}^{(1)} A_{35}  \\ \nonumber &-&
A_{14}A_{46}^{(2)}A_{61}^{(1)} - A_{56}^{(3)}A_{62}^{(1)}A_{25} - A_{56}^{(1)}A_{61}^{(2)}A_{15} - A_{62}^{(2)}A_{24}A_{46}^{(3)}
- A_{46}^{(1)}A_{63}^{(1)} A_{34} - A_{56}^{(2)}A_{63}^{(2)}A_{35}
  .\end{eqnarray}}

Just as in the Model 1 case, each of these quivers with potential correspond to a brane tiling, see for example \cite[Figures 18, 23, and 24]{brane_dimer} or \cite[Figures 25, 26, 27]{hanany_polygons}.
Using Kasteleyn theory (see \cite{GK}, \cite{Hanany-Vegh}, or \cite{dimermodel}), one can associate a polygon (toric diagram) to each of these brane tilings.  For each of these four models, the associated toric diagram is the hexagon (up to $SL_2(\mathbb{Z})$-transformations) as in Figure \ref{fig:quiv_brane} (Left).  We illustrate these brane tilings alongside our combinatorial formulas in Section \ref{sec:comb}.

\vspace{1em}

Our main results are analogous to the main results of \cite{LaiMus} but starting from one of the other three models as an initial seed.  In particular, we obtain
(1) a similar parametrization of toric cluster variables by points in $\mathbb{Z}^3$, (2) a compact algebraic formula for all such toric cluster variables as Laurent polynomials, (3) a construction of a subgraph of the appropriate brane tiling for \emph{most}\footnote{Our construction is defined for those cases where the contour defined by lifting the point in $\mathbb{Z}^3$ to $\mathbb{Z}^6$ has no self-intersections}. toric cluster variables, (4) a proof that the partition function for perfect matchings of such subgraphs, with appropriate weightings, yields a combinatorial formula for the Laurent expansions of \emph{most} toric cluster variables, (5) a combinatorial interpretation using double-dimers for toric cluster variables associated to solutions of the hexahedron recurrence of \cite{KP}, and (6) a conjectured combinatorial interpretation using a mix of dimers and double-dimers for the remaining toric cluster variables not yet covered.

This paper is organized as follows.  In Section \ref{sec:param}, the parameterization and algebraic formula (for toric cluster variables) from \cite{LaiMus} is reviewed and extended to Models 2, 3, and 4.   We give our main results, the associated combinatorial formula for Models 2, 3, and 4 for (most) toric cluster variables in Section \ref{sec:comb}.  In Section \ref{sec:history}, we review the history of related combinatorial questions and how mutations of the $dP_3$ quiver allow us to package these disparate results all under one roof.  Section \ref{sec:proofs} provides the proofs while Sections  \ref{sec:Mod3} and \ref{sec:Mod4} provide the associated changes of coordinates 
relating our formulas to previous combinatorial results.  We revisit the hexahedron recurrence of Kenyon and Pemantle and provide our conjectured combinatorial interpretation in terms of dimers and double-dimers in Section \ref{sec:super}.  Finally, we discuss additional directions and open questions in 
Section \ref{sec:open}. 

\section{Parameterization by $\mathbb{Z}^3$ and Compact Algebraic Formulae}

\label{sec:param}

We begin by reviewing some results from Section 2 of \cite{LaiMus}.  In that paper, our starting quiver was $Q_1$, i.e. Model 1, and we parameterized the initial cluster $\{x_1,x_2,x_3,x_4,x_5,x_6\}$ by the prism
$$[(0,-1,1), (0,-1,0), (-1,0,0), (-1,0,1), (0,0,1),(0,0,0)] \in \mathbb{Z}^3.$$
Mutating by $\mu_1$, $\mu_4$, and $\mu_3$ (in that order) leads to quivers of Model 2, 3, and 4, respectively.  By comparing symmetries of $\mathbb{Z}^3$ to the relations induced by the action of mutations on clusters, e.g. $(\mu_1\mu_2)^2 = (\mu_1\mu_2\mu_3\mu_4)^3 = 1$, we were able to translate toric mutations into geometric transformations.
Consequently, sequences of toric mutations lead to new cluster seeds parameterized by different $6$-tuples in $\mathbb{Z}^3$.  The shape of each of those $6$-tuples corresponds to whether it associates to a quiver of Model 1, 2, 3, or 4.  See Figures \ref{fig:ModelII}, \ref{fig:ModelIII}, and \ref{fig:ModelIV}.  Accordingly, starting from either of these four models (up to graph isomorphism and reversal of all arrows, we assume our starting point is the quiver $Q_1$, $Q_2$, $Q_3$, or $Q_4$), all cluster variables that are reachable by a sequence of toric mutations are thus parameterized by $(i,j,k)\in \mathbb{Z}^3$.

\begin{figure}
\includegraphics[width=5.5in]{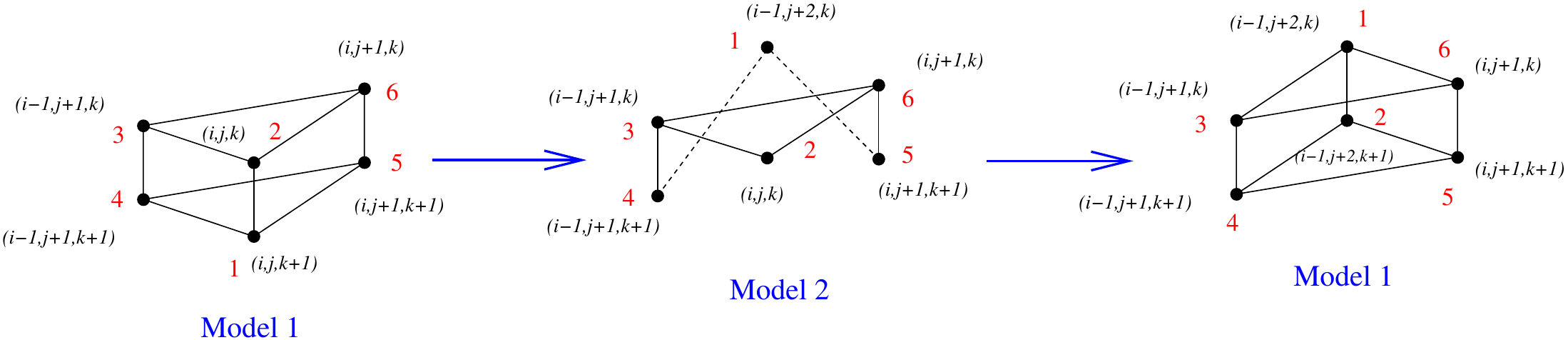}
\caption{Illustrating the transformations induced by $\mu_1$ and $\mu_2$. (Figure 8 of \cite{LaiMus}.)}
\label{fig:ModelII}
\end{figure}

\begin{figure}
\includegraphics[width=6in]{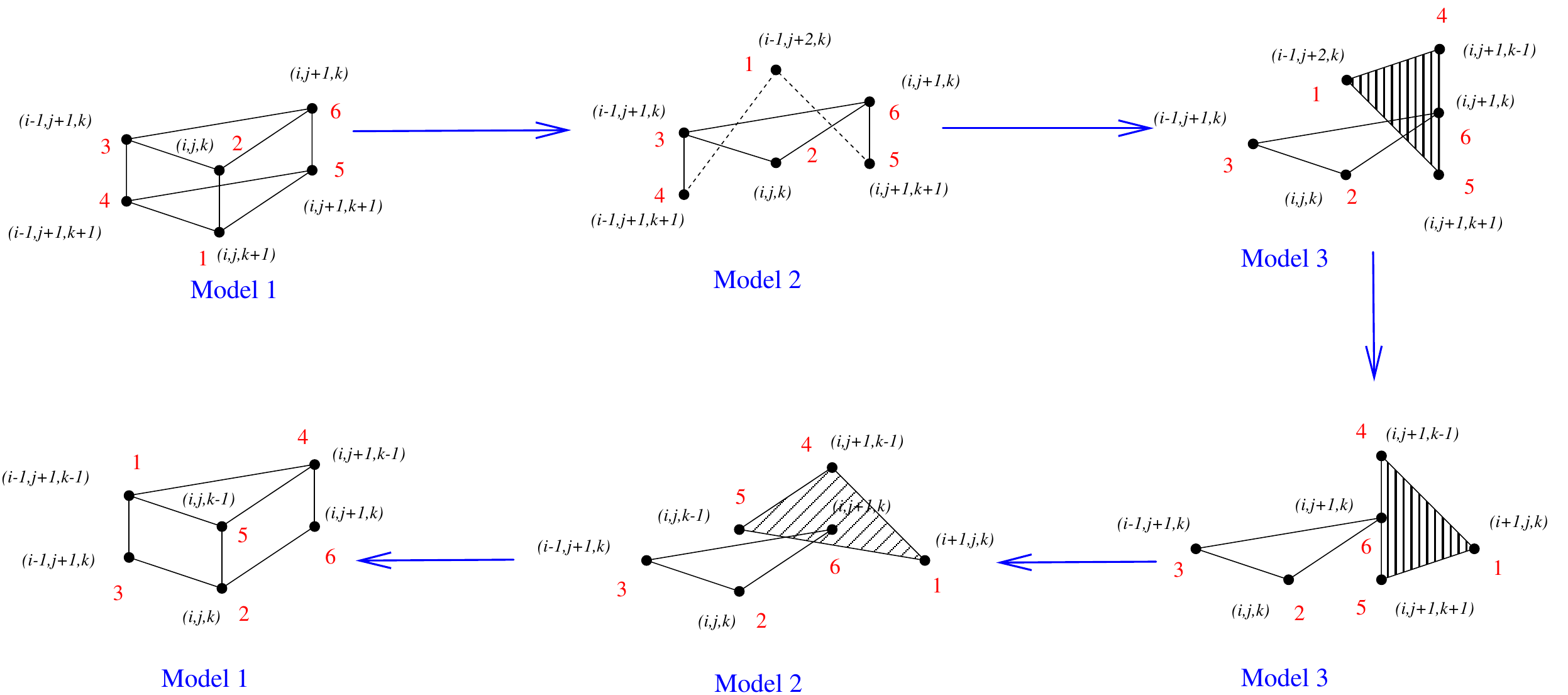}
\caption{Illustrating the $\mathbb{Z}^3$-transformations induced by $\mu_1\mu_4\mu_1\mu_5\mu_1$. (Based on Figure 9 of \cite{LaiMus}.)}
\label{fig:ModelIII}
\end{figure}

\begin{figure}
\includegraphics[width=4.5in]{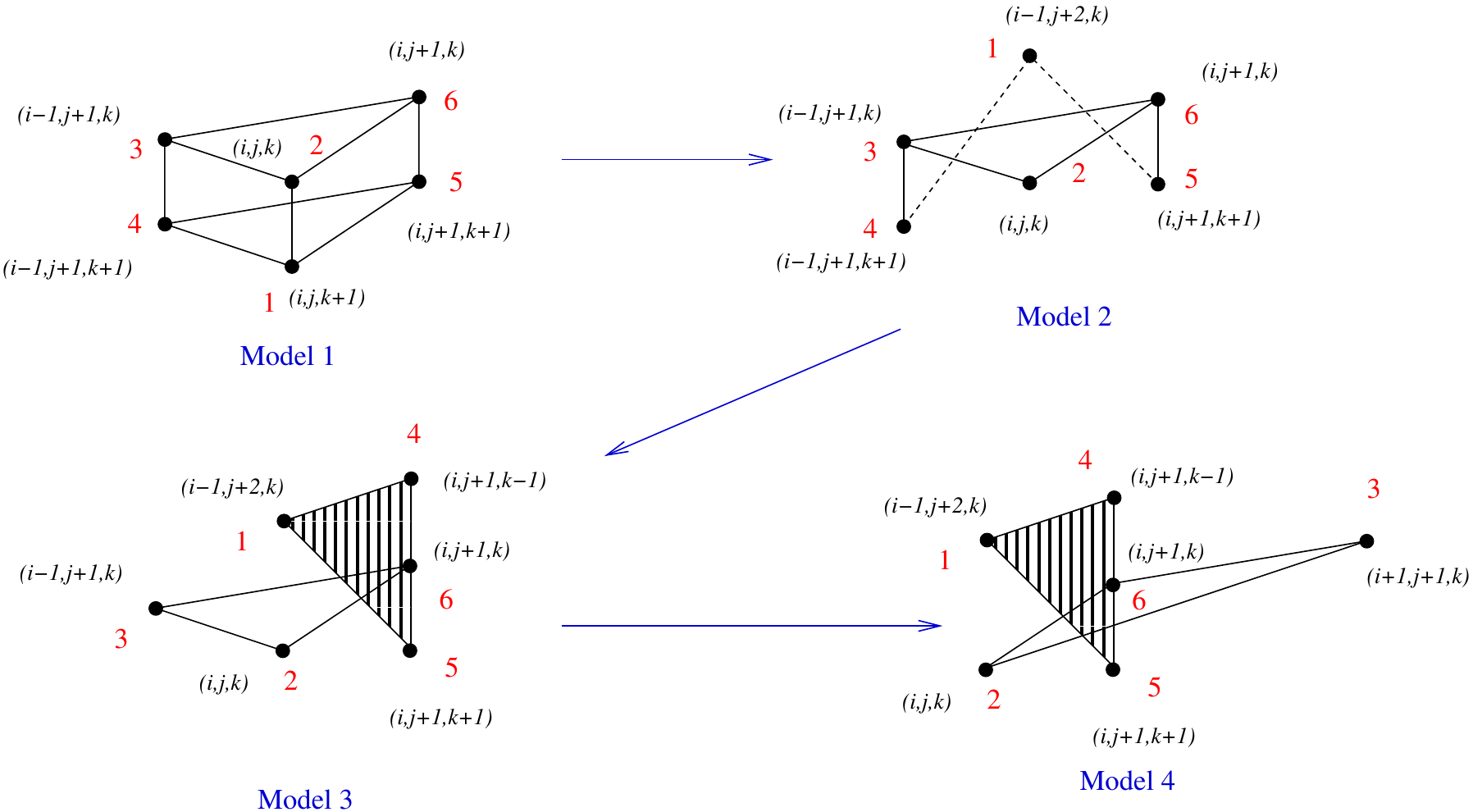}
\caption{Illustrating the toric mutation sequence $\mu_1\mu_4\mu_3$. (Figure 10 of \cite{LaiMus}.)}
\label{fig:ModelIV}
\end{figure}

Under this parametrization, we let the initial cluster correspond to $$\Delta_1 = [(0,-1,1), (0,-1,0), (-1,0,0), (-1,0,1), (0,0,1),(0,0,0)] \in \mathbb{Z}^3$$ when starting with the quiver $Q_1$.  Mutating the quiver $Q_1$ by $\mu_1$ leads to $Q_2$ and the seed corresponding to
$$\Delta_2 = [(-1,1,0), (0,-1,0), (-1,0,0), (-1,0,1), (0,0,1),(0,0,0)].$$
Following up from that, mutating by $\mu_4$ yields $Q_3$ and the seed corresponding to
$$\Delta_3 = [(-1,1,0), (0,-1,0), (-1,0,0), (0,0,-1), (0,0,1),(0,0,0)].$$
Lastly, mutating by $\mu_3$ yields $Q_4$ and the seed corresponding to
$$\Delta_4 = [(-1,1,0), (0,-1,0), (1,0,0), (0,0,-1), (0,0,1),(0,0,0)].$$
We thus get four different families of Laurent expansions depending on whether we associate the initial cluster $\{x_1,x_2,x_3,x_4,x_5,x_6\}$ to
$\Delta_1, \Delta_2, \Delta_3$, or $\Delta_4$.

\subsection{Formula in the Model 1 Case}

Letting $(i,j,k) \in \mathbb{Z}^3$, we let $z_{i,j,k}^{(1)}$ denote the cluster variable (reachable by
a toric mutation sequence) starting from the initial seed with the quiver $Q_1$.

\begin{theorem} [Theorem 3.1 of \cite{LaiMus}] \label{thm:z1}
Let $A = \frac{x_3x_5 + x_4x_6}{x_1x_2}$,
$B = \frac{x_1x_6 + x_2x_5}{x_3x_4}$, \\
$C = \frac{x_1x_3 + x_2x_4}{x_5x_6}$,
$D = \frac{x_1x_3x_6 + x_2x_3x_5 + x_2x_4x_6}{x_1x_4x_5}$,
$E = \frac{x_2x_4x_5 + x_1x_3x_5 + x_1x_4x_6}{x_2x_3x_6}$.

$$z_{i,j,k}^{(1)} = x_r A^{\lfloor \frac{(i^2+ij+j^2+1)+i+2j}{3}\rfloor}
B^{\lfloor \frac{(i^2+ij+j^2+1)+2i+j}{3}\rfloor}
C^{\lfloor \frac{i^2+ij+j^2+1}{3}\rfloor}
D^{\lfloor \frac{(k-1)^2}{4}\rfloor}
E^{\lfloor \frac{k^2}{4}\rfloor}$$ where
$r = 1$ if $2(i-j) + 3k \equiv 5$, ~ $r = 2$ if $2(i-j) + 3k \equiv 2$, ~ $r = 3$ if $2(i-j) + 3k \equiv 4$,

$r = 4$ if $2(i-j) + 3k \equiv 1$, ~ $r = 5$ if $2(i-j) + 3k \equiv 3$, ~ $r = 6$ if $2(i-j) + 3k \equiv 0$ working modulo $6$.  In particular, the variable $x_r$ is uniquely determined by the values of $(i-j)$ modulo $3$ and $k$ modulo $2$.
\end{theorem}

\begin{remark} \label{rem:z1}
This nontrivial correspondence between the values of $r$ and $2(i-j)+3k \mod 6$ comes from our cyclic ordering
of $Q_1$ in counter-clockwise order given in Figure \ref{fig:quiv_brane}.  In particular, as we rotate from vertex $r$ to $r'$ in clockwise order, the corresponding value of $2(i-j)+3k \mod 6$ increases by $1$ (circularly).
\end{remark}

\subsection{Formulae for Models 2, 3, and 4}

By mutating quiver $Q_1$ by vertices $1$, $4$, and $3$ in order, we obtain variants of Theorem 3.1 of \cite{LaiMus}.  We let $z_{i,j,k}^{(2)}$, $z_{i,j,k}^{(3)}$, and $z_{i,j,k}^{(4)}$ denote the cluster variable parameterized by $(i,j,k)$ starting from the quivers $Q_2$, $Q_3$, or $Q_4$ respectively.  In the following, we use $Y_r$, $Y_r'$ and $Y_r''$, respectively in place of $x_r$ but with the same nontrivial cyclic ordering as indicated in Theorem \ref{thm:z1} and Remark \ref{rem:z1}.

\begin{theorem} \label{thm:z2}
Let $A = \frac{x_1}{x_2}$,
$B = \frac{x_4x_6^2 + x_1x_2x_5 + x_3x_4x_5}{x_1x_3x_4}$,
$C = \frac{x_1x_2x_4 + x_3x_4x_6 + x_3^2x_5}{x_1x_5x_6}$, \\
$D = \frac{x_1x_2 + x_3x_6}{x_4x_5}$,
$E = \frac{x_4^2x_6^2 + x_1x_2x_4x_5 + 2x_3x_4x_5x_6  + x_3^2x_5^2}{x_1x_2x_3x_6}$.

$$z_{i,j,k}^{(2)} = Y_r A^{\lfloor \frac{(i^2+ij+j^2+1)+i+2j}{3}\rfloor}
B^{\lfloor \frac{(i^2+ij+j^2+1)+2i+j}{3}\rfloor}
C^{\lfloor \frac{i^2+ij+j^2+1}{3}\rfloor}
D^{\lfloor \frac{(k-1)^2}{4}\rfloor}
E^{\lfloor \frac{k^2}{4}\rfloor}$$ where
$Y_1 = \frac{x_4x_6+x_3x_5}{x_1}$, $Y_j = x_j$ for $2 \leq j \leq 6$.
\end{theorem}

\begin{theorem} \label{thm:z3} Let $A = \frac{x_1}{x_2}$,
$B = \frac{x_4x_5 + x_6^2}{x_1x_3}$,
$C = \frac{x_1^2x_2^2  + 2 x_1x_2x_3x_6 + x_3^2x_6^2 + x_3^2x_4x_5}{x_1x_4x_5x_6}$,
$D = \frac{x_4}{x_5}$,
$E = \frac{x_1^2x_2^2x_6^2 + 2x_1x_2x_3x_6^3 + x_3^2x_6^4 + x_1^2x_2^2x_4x_5 + 3x_1x_2x_3x_4x_5x_6 + 2x_3^2x_4x_5x_6^2 + x_3^2x_4^2x_5^2}{x_1x_2x_3x_4^2x_6}$.

$$z_{i,j,k}^{(3)} = Y_r' A^{\lfloor \frac{(i^2+ij+j^2+1)+i+2j}{3}\rfloor}
B^{\lfloor \frac{(i^2+ij+j^2+1)+2i+j}{3}\rfloor}
C^{\lfloor \frac{i^2+ij+j^2+1}{3}\rfloor}
D^{\lfloor \frac{(k-1)^2}{4}\rfloor}
E^{\lfloor \frac{k^2}{4}\rfloor}$$ where
$Y_1' = \frac{x_1x_2x_6+x_3x_6^2 + x_3x_4x_5}{x_1x_4}$, $Y_4' = \frac{x_1x_2+x_3x_6}{x_4}$,   $Y_j' = x_j$ for $j \in \{2,3,5,6\}$.
\end{theorem}

\begin{theorem} \label{thm:z4} Let $A = \frac{x_1}{x_2}$,
$B = \frac{x_3}{x_1}$,
$C = \frac{x_6^6 + 2x_1x_2x_3x_6^3 + x_1^2x_2^2x_3^2 + 3x_4x_5x_6^4 + 2x_1x_2x_3x_4x_5x_6 + 3x_4^2x_5^2x_6^2 + x_4^3x_5^3}{x_1x_3^2x_4x_5x_6}$,
$D = \frac{x_4}{x_5}$,
$E = \frac{x_6^6 + 2x_1x_2x_3x_6^3 + x_1^2x_2^2x_3^2 + 3x_4x_5x_6^4 + 3x_1x_2x_3x_4x_5x_6 + 3x_4^2x_5^2x_6^2 + x_4^3x_5^3}{x_1x_2x_3x_4^2x_6}$.

$$z_{i,j,k}^{(4)} = Y_r'' A^{\lfloor \frac{(i^2+ij+j^2+1)+i+2j}{3}\rfloor}
B^{\lfloor \frac{(i^2+ij+j^2+1)+2i+j}{3}\rfloor}
C^{\lfloor \frac{i^2+ij+j^2+1}{3}\rfloor}
D^{\lfloor \frac{(k-1)^2}{4}\rfloor}
E^{\lfloor \frac{k^2}{4}\rfloor}$$ where
$Y_1'' = \frac{x_6^4 + x_1x_2x_3x_6 + 2x_4x_5x_6^2 + x_4^2x_5^2}{x_1x_3x_4}$, $Y_3'' = \frac{x_4x_5 + x_6^2}{x_3}$,
$Y_4'' = \frac{x_1x_2x_3 + x_4x_5x_6 + x_6^3}{x_3x_4}$, $Y_j'' = x_j$ for $j \in \{2,5,6\}$.
\end{theorem}

The proof of these three theorems follow from Theorem 3.1 of \cite{LaiMus} followed by algebraic transformations corresponding to cluster mutations.

\begin{proof}
Beginning with the quiver $Q_2$ and the initial cluster $\{x_1,x_2,x_3,x_4,x_5,x_6\}$, if we mutate at vertex $1$, then we get $Q_1$ with the new cluster $\{Y_1,x_2,x_3,x_4,x_5,x_6\}$ where $Y_1 = \frac{x_4x_6+x_3x_5}{x_1}$.  Hence, we obtain the formula for $z_{i,j,k}^{(2)}$ by taking the formula for $z_{i,j,k}^{(1)}$ and substituting in $x_1$ by $Y_1$.  Such a substitution also alters the values of $A$, $B$, $C$, $D$, and $E$ accordingly.  This proves Theorem \ref{thm:z2}.
Similarly, if we start with Model 3 quiver $Q_3$ as well as the initial cluster $\{x_1,x_2,x_3,x_4,x_5,x_6\}$, then we mutate by $\mu_4$ then $\mu_1$ to get to Model $2$ and then to Model $1$.  That yields quiver $Q_1$ with cluster $\{Y_1',x_2,x_3,Y_4',x_5,x_6\}$ where
$Y_1' = \frac{x_1x_2x_6+x_3x_6^2 + x_3x_4x_5}{x_1x_4}$ and $Y_4' = \frac{x_1x_2+x_3x_6}{x_4}$.  Again, we substitute in $Y_1'$ for $x_1$ and $Y_4'$ for $x_4$ hence obtaining the formula of Theorem \ref{thm:z3}.
Finally starting with the Model 4 quiver $Q_4$ and the initial cluster $\{x_1,x_2,x_3,x_4,x_5,x_6\}$, we mutate by $\mu_3$, $\mu_4$, and $\mu_1$ in that order.  This mutation sequence yields quiver $Q_1$ with cluster
 $\{Y_1'', x_2, Y_3'', Y_4'', x_5,x_6\}$ where $Y_1'' = \frac{x_6^4 + x_1x_2x_3x_6 + 2x_4x_5x_6^2 + x_4^2x_5^2}{x_1x_3x_4}$, $Y_3'' = \frac{x_4x_5 + x_6^2}{x_3}$,
$Y_4'' = \frac{x_1x_2x_3 + x_4x_5x_6 + x_6^3}{x_3x_4}$, yielding Theorem \ref{thm:z4}.
\end{proof}

\section{Constructing Subgraphs of Contours}

\label{sec:comb}

We now proceed to describe the construction of subgraphs for the Model 2, 3 and 4 quivers $Q_2$, $Q_3$, and $Q_4$, respectively.  Our procedure involves contours and is analogous to the construction
from Section 4 of \cite{LaiMus}.  We first lift our $\mathbb{Z}^3$-parametrization to a $\mathbb{Z}^6$-parametrization:
$$\phi: (i,j,k) \longrightarrow (a,b,c,d,e,f) = (j+k, -i-j-k, i+k, j+1-k, -i-j-1+k, i+1-k).$$
We use the resulting $6$-tuples to build subgraphs of the brane tilings $\mathcal{T}_2$ (Figure \ref{fig:Model2-contour}), $\mathcal{T}_3$ (Figure \ref{fig:Model3-contour}), and $\mathcal{T}_4$ (Figure \ref{fig:Model4-contour}), respectively.

\subsection{Model 1}

We begin by reviewing the situation for Model 1 as described in \cite{LaiMus}.
Given a $6$-tuple $(a,b,c,d,e,f) \in \mathbb{Z}^6$, we consider a {\bf contour} $\mathcal{C}_1(a,b,c,d,e,f)$ whose side-lengths are $a,b, \dots , f$ in counter-clockwise order (starting with a side in the downward and rightward direction when the entry $a$ is positive).  In the case of a negative entry, we draw the contour in the opposite direction for the associated side.  Several qualitatively different contours are illustrated in Figure \ref{fig:contours} with their corresponding subgraphs of the brane tiling $\mathcal{T}_1$ shown in Figure \ref{fig:ex0}.

\begin{figure}
    \centering
    \scalebox{0.9}{\includegraphics[keepaspectratio=true, width=150 mm]{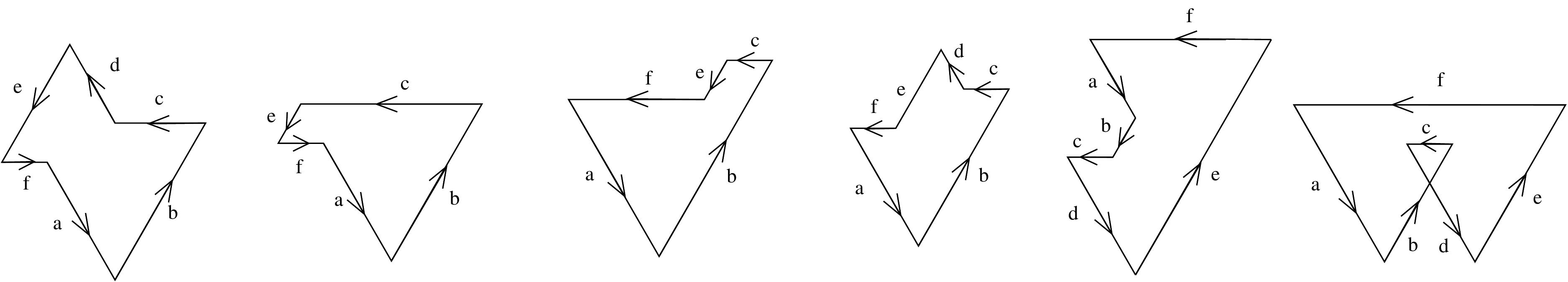}}
\caption{\small From left to right, the cases where $(a,b,c,d,e,f) = $ (1)  $(+, -, + , +, -, +)$, (2) $(+, -, +, 0, -, +)$, (3) $(+, -, +, 0, -, -)$,
(4)$(+, -, + , +, -, -)$,  (5) $(+, +, +, -, +, -)$, (6) $(+, -, +, -, +, -)$. (Figure 12 of \cite{LaiMus}.)}
    \label{fig:contours}
\end{figure}

\begin{figure}
\includegraphics[width=4.9in]{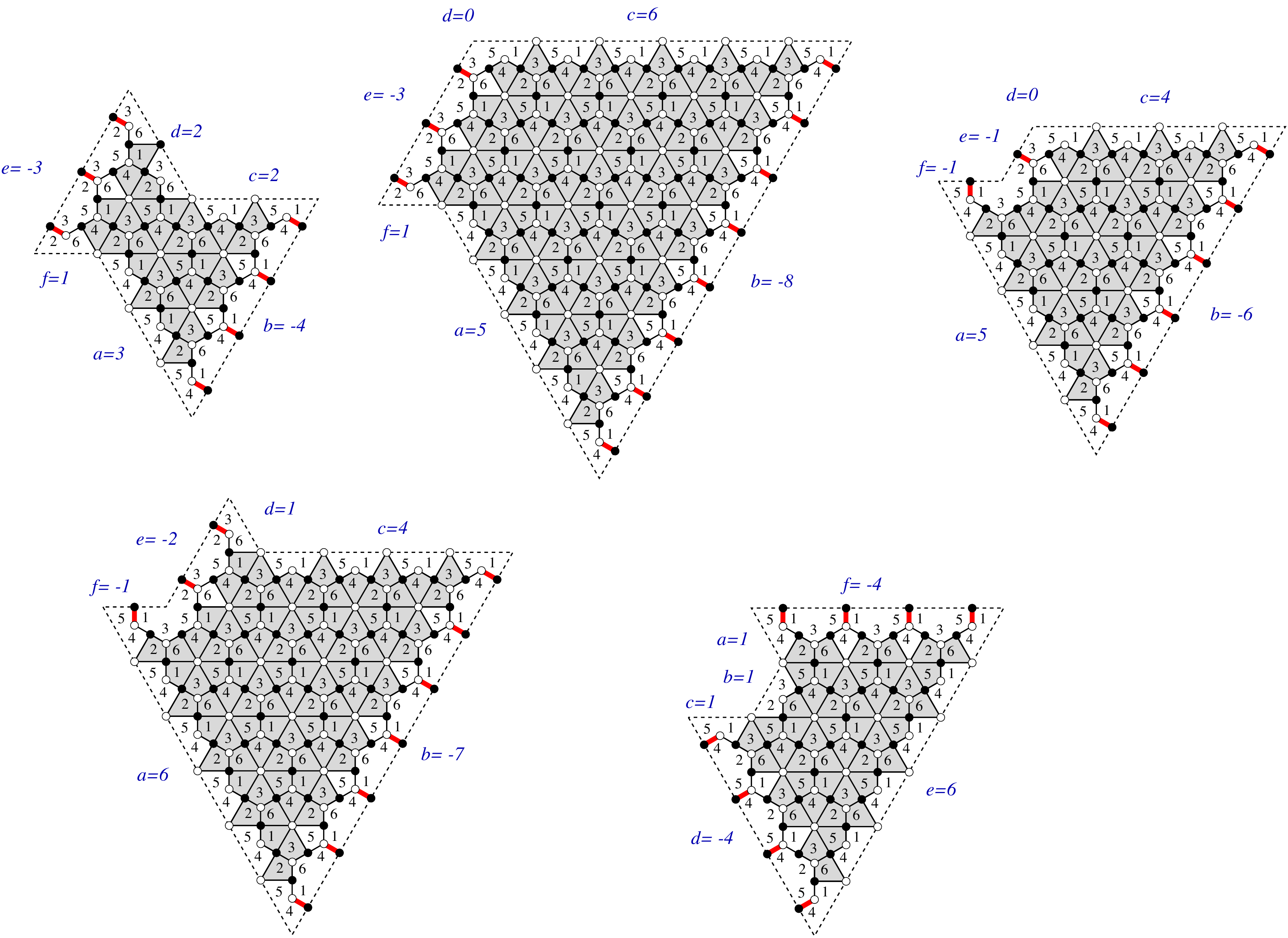}
\caption{
(1) $\mathcal{G}_1(3, -4, 2, 2, -3, 1)$, (2) $\mathcal{G}_1(5,-8,6,0,-3,1)$, (3) $\mathcal{G}_1(5,-6,4,0,-1,-1)$,
(4) $\mathcal{G}_1(6,-7,4,1,-2,-1)$, and (5) $\mathcal{G}_1(1, 1, 1, -4, 6, -4)$ respectively. (Figure 13 of \cite{LaiMus}.)}
\label{fig:ex0}
\end{figure}

\begin{definition} [Definition 4.1 of \cite{LaiMus}]
\label{def:subgraphs} Suppose that the contour $\mathcal{C}_1(a,b,c,d,e,f)$ does not intersect itself (the rightmost picture in Figure \ref{fig:contours} shows an example of a contour with self-intersections). We define $\mathcal{G}_1(a,b,c,d,e,f)$ and $\widetilde{\mathcal{G}}_1(a,b,c,d,e,f)$ by the following rules:

Step 1: The brane tiling $\mathcal{T}_1$ consists of a subdivided triangular lattice.  We superimpose the contour $\mathcal{C}_1(a,b,c,d,e,f)$ on top of $\mathcal{T}_1$ so that its sides follow the lines of the triangular lattice, beginning at a white vertex of degree $6$.  In particular, sides $a$ and $d$ are tangent to the faces $1$ and $2$, sides $b$ and $e$ are tangent to the faces $5$ and $6$, and sides $c$ and $f$ are tangent to the faces $3$ and $4$.  We scale the contour so that a side of length $\pm 1$ transverses two edges of the brane tiling $\mathcal{T}_1$, and thus starts and ends at a white vertex of degree $6$ with no such white vertices in-between.

Step 2: For any side of positive (resp. negative) length, we remove all black (resp. white) vertices along that side.

Step 3: A side of length zero corresponds to a single white vertex.  If one of the adjacent sides is of negative length or also of length zero, then that white vertex is removed during step 2.  On the other hand, if the side of length zero is adjacent to two sides of positive length, we keep the white vertex.  

Step 4: We define $\widetilde{\mathcal{G}}_1(a,b,c,d,e,f)$ to be the resulting subgraph, which will contain a number of black vertices of valence one.  After matching these up with the appropriate white vertices and continuing this process until no vertices of valence one are left, we obtain a simply-connected graph $\mathcal{G}_1(a,b,c,d,e,f)$ which we call the {\bf Core Subgraph}, following the notation of \cite{BMPW}.
\end{definition}

We now describe the weighting scheme that yields Laurent polynomials from the subgraphs defined above, following Section 5 of \cite{LaiMus} and based on earlier work in \cite{speyer}, \cite{GK}, \cite{zhang}, and \cite{LMNT}.  We associate the \textbf{weight} $\frac{1}{x_i x_j}$ to each edge bordering the faces labeled $i$ and $j$ in the brane tiling and let $\mathcal{M}(G_1)$ denote the set of perfect matchings of a subgraph $G_1$ of the brane tiling. Define the weight $w(M)$ of a perfect matching $M$ in the usual manner as the product of the weights of the edges included in the matching under the weighting scheme, and then we define the weight of $G_1$ as
\[
w(G) = \sum_{M \in \mathcal{M}(G_1)}w(M).
\]
We define the \textbf{covering monomial for Model 1}, $m(\widetilde{G}_1)$ of  the graph $\widetilde{G}_1=\widetilde{\mathcal{G}}_1(a,b,c,d,e,f)$ as the product $x_1^{a_1}x_2^{a_2}x_3^{a_3}x_4^{a_4}x_5^{a_5}x_6^{a_6}$, where $a_j$ is the number of faces labeled $j$ restricted inside the contour $\mathcal{C}_1(a,b,c,d,e,f)$.  This definition of the covering monomial is based on earlier work of \cite{jeong} and \cite{JMZ}, as motivated by conductances coordinates of \cite{GK}.

\begin{theorem} [Theorem 5.9 of \cite{LaiMus}] \label{thm:formula1} Let $(a,b,c,d,e,f) \in \mathbb{Z}^6$ be the image of $\phi(i,j,k)$ for $(i,j,k) \in \mathbb{Z}^3$, where $\phi$ is defined at the beginning of Section \ref{sec:comb}.  Then as long as the contour $\mathcal{C}_1(a,b,c,d,e,f)$ has no self-intersections, then
$$z_{i,j,k}^{(1)} = m(\widetilde{\mathcal{G}}_1(a,b,c,d,e,f)) \cdot w(\mathcal{G}_1(a,b,c,d,e,f)).$$
\end{theorem}

\subsection{Model 2}
For a given $(a,b,c,d,e,f) \in \mathbb{Z}^6$ in the image of this map $\phi$, we begin at one of the $5$-valent white vertices in the brane tiling $\mathcal{T}_2$ associated to $(Q_2,W_2)$.  This vertex is unique in the fundamental domain of the brane tiling since there is a unique negative term in the potential $W_2$ of degree $5$.  We then follow the lines as illustrated in Figure \ref{fig:Model2-contour}.  Here we have shown the orientations if the entry is positive.  For a negative entry, we instead traverse the indicated trajectory in the opposite direction.  Like the Model 1 case, the absolute value of an entry of this six-tuple indicates the length of the contour in that direction, where a segment of length one ends at the next five-valent white vertex reached (going through one black vertex in the process).  We abbreviate this contour as $\mathcal{C}_2(a,b,c,d,e,f)$.  We next describe how to translate these contours into subgraphs.  We follow the same prescription as in \cite[Definition 4.1]{LaiMus}.

\begin{figure}
\includegraphics[width=3in]{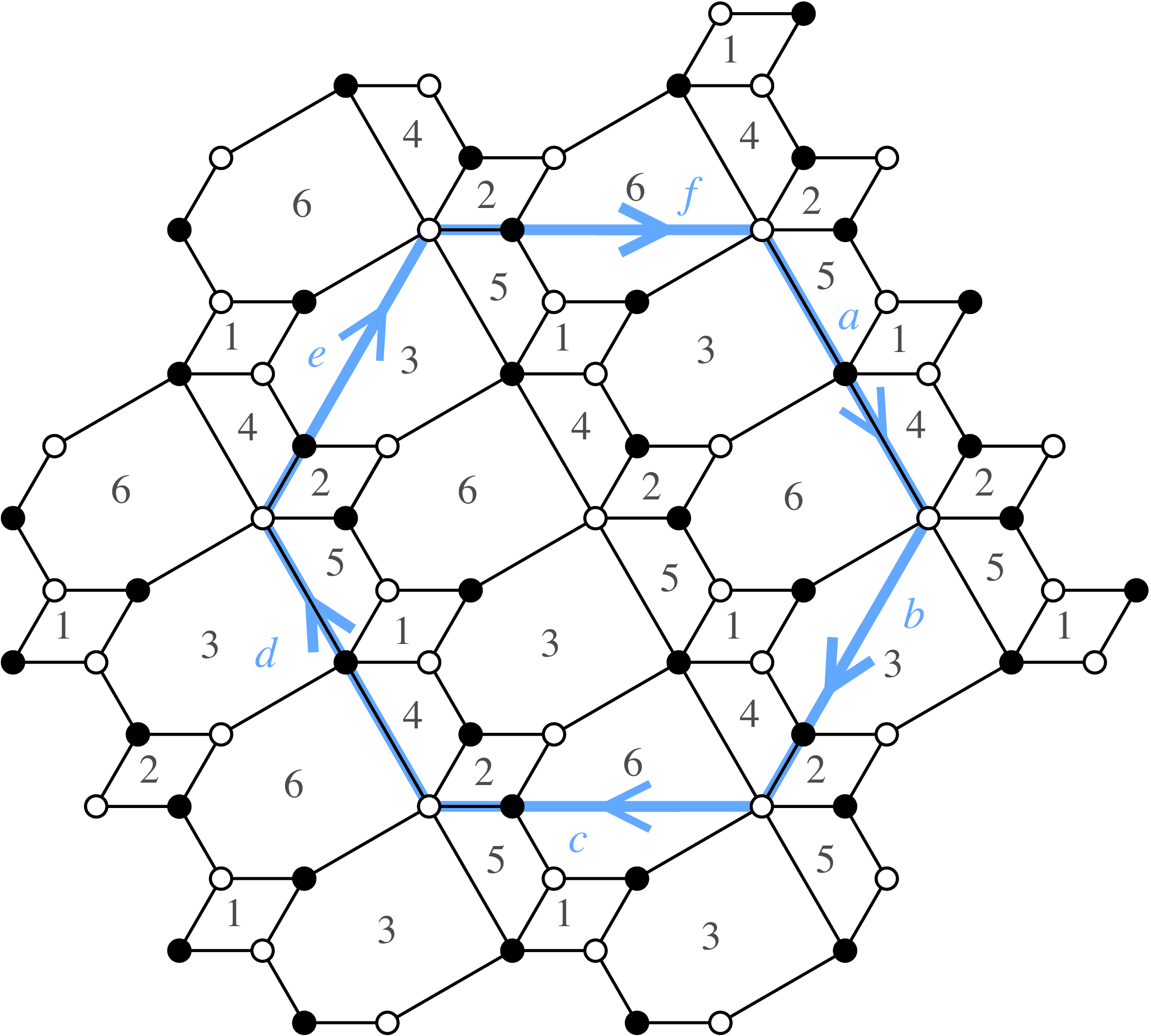}
\caption{Contour for $(Q_2,W_2)$, on top of the brane tiling $\mathcal{T}_2$, in the positive directions with segments of length one indicated. (Note that the sides $a,b,\dots, f$ appear to be in clockwise order.  However, as in Figures \ref{fig:Model2initial} and \ref{fig:Model2egs}, for the $6$-tuples lifted from $\mathbb{Z}^3$ via $\phi$ as in this paper, the sides instead will appear in counter-clockwise order.) \\
}
\label{fig:Model2-contour}
\end{figure}

\begin{definition}
Suppose that the contour $\mathcal{C}_2(a,b,c,d,e,f)$ does not intersect itself.  Under this assumption, we use the
contour $\mathcal{C}_2(a,b,c,d,e,f)$ to define two subgraphs, $\widetilde{\mathcal{G}}_2(a,b,c,d,e,f)$ and $\mathcal{G}_2(a,b,c,d,e,f)$, of the Model 2 brane tiling $\mathcal{T}_2$ by the following:

Step 1: Superimpose the contour $\mathcal{C}_2(a,b,c,d,e,f)$ onto $\mathcal{T}_2$ starting from a $5$-valent white vertex as above.

Step 2: For any side of positive (resp. negative) length, we remove all black (resp. white) vertices along that side.

Step 3: A side of length zero corresponds to a single white vertex.  If one of the adjacent sides is of negative length or also of length zero, then that white vertex is removed during Step 2.
On the other hand, if the side of length zero is adjacent to two sides of positive length, we keep the white vertex.  

Step 4: We define $\widetilde{\mathcal{G}}_2(a,b,c,d,e,f)$ to be the resulting subgraph.  However, this subgraph will often include black vertices of valence one.  After matching these up with the appropriate white vertices and continuing this process until there are no vertices of valence one left, we are left with the simply-connected graph $\mathcal{G}_2(a,b,c,d,e,f)$.
\end{definition}

To obtain Laurent polynomial expressions from the graphs $\mathcal{G}_2(a,b,c,d,e,f)$'s, we need to define covering monomials for the Model 2 case.  Unlike the work of \cite[Definition 5.6]{LaiMus}, we have (i) both hexagonal and quadrilateral faces, and (ii) the contours lines can cut through particular faces (e.g. hexagonal faces 3 and 6).  We update the definition of covering monomial accordingly.

\begin{definition}
Given a contour $\mathcal{C}_2(a,b,c,d,e,f)$ which defines the extended subgraph $\widetilde{\mathcal{G}}_2 = \widetilde{\mathcal{G}}_2(a,b,c,d,e,f)$, we define the {\bf covering monomial for Model 2}
$m(\widetilde{\mathcal{G}}_2)$ as the product $x_1^{a_1}x_2^{a_2}x_3^{2a_3+b_3}x_4^{a_4}x_5^{a_5}x_6^{2a_6+b_6}$ where $a_j$ is the number of faces labeled $j$ which lie fully inside the contour $\mathcal{C}_2(a,b,c,d,e,f)$ and $b_3$ (resp. $b_6$) equals the number of hexagonal faces labeled $3$ (resp. $6$) which partially live inside the contour.
\end{definition}

We then compute the weight of the subgraph $\mathcal{G}_2 = \mathcal{G}_2(a,b,c,d,e,f)$ as
$$w(\mathcal{G}_2) = \sum_{M \in \mathcal{M}(\mathcal{G}_2)} w(M)$$
where $\mathcal{M}(\mathcal{G}_2)$ is the set of perfect matchings $M$ of subgraph $\mathcal{G}_2$, and the weight of a matching $M$ is the product
$w(M) = \prod_{e_{ij} \in M} \frac{1}{x_i x_j}$.
Here $e_{ij}$ is an edge in the perfect matching $M$ which borders the faces labeled $i$ and $j$.

\begin{remark}
\label{rem:CM-variant}
As another way of calculating the covering monomial, we can break all hexagons labeled $3$ or $6$ in the brane tiling into quadrilaterals (both labeled $3$ or $6$ respectively) as in Figure \ref{fig:CM-quads2}.  With this adjustment made, all six sides of the contour (even sides $b,c,e,$ and $f$) travel only along segments of the adjusted brane tiling.  Hence, for the purposes of defining the covering monomial, it is the usual definition of taking the product of all $x_j$ where $j$ is the label of a quadrilateral face (fully) contained inside the contour.
\end{remark}

\begin{figure}\centering
\includegraphics[width=3in]{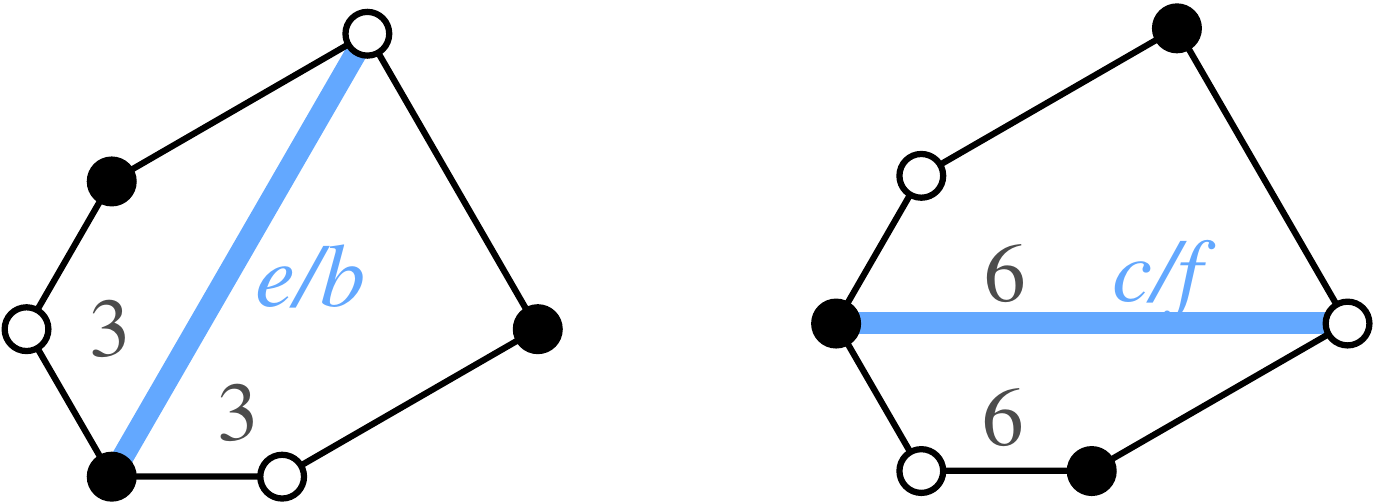}
\caption{Cutting Hexagonal Faces $3$ and $6$ into Quadrilaterals in the Model 2 case.}
\label{fig:CM-quads2}
\end{figure}

\begin{theorem} \label{thm:formula2} Let $(a,b,c,d,e,f) \in \mathbb{Z}^6$ be the image $\phi(i,j,k)$ for $(i,j,k) \in \mathbb{Z}^3$.  Then as long as the contour $\mathcal{C}_2(a,b,c,d,e,f)$ has no self-intersections, then
$$z_{i,j,k}^{(2)} = m(\widetilde{\mathcal{G}}_2(a,b,c,d,e,f)) \cdot w(\mathcal{G}_2(a,b,c,d,e,f)).$$
\end{theorem}

For the $6$-tuples corresponding to the image of $\phi(\Delta_1)$, i.e.
\begin{eqnarray*}
\hspace{2em} C_1^{(2)} = \mathcal{C}_2(0,0,1,-1,1,0), C_2^{(2)} = \mathcal{C}_2(-1,1,0,0,0,1), C_3^{(2)} = \mathcal{C}_2(0,1,-1,1,0,0), \\
\hspace{2em} C_4^{(2)} = \mathcal{C}_2(1,0,0,0,1,-1), C_5^{(2)} = \mathcal{C}_2(1,-1,1,0,0,0), C_6^{(2)} = \mathcal{C}_2(0,0,0,1,-1,1),
\end{eqnarray*}
we draw the contours and the associated subgraphs in Figure \ref{fig:Model2initial}.  In particular, the initial contour $C_1^{(2)}$ yields a subgraph $\mathcal{G}_2(0,0,1,-1,1,0)$ consisting solely of a quadrilateral labeled $1$.  The weighted enumeration of perfect matchings of this subgraph gives $Y_1 = \frac{x_4x_6+x_3x_5}{x_1}$ exactly as desired.  The remaining initial contours $C_2^{(2)},C_3^{(2)},\dots, C_6^{(2)}$ yield empty subgraphs $\mathcal{G}_2$ although their extended subgraphs $\widetilde{\mathcal{G}}_2$ contain some edges incident to $1$-valent black vertices.  Hence, $C_2^{(2)},C_3^{(2)},\dots, C_6^{(2)}$ yield $x_2,x_3,\dots,x_6$ as expected.  Some further examples of subgraphs associated to contours in $\mathcal{T}_2$ appear in Figure \ref{fig:Model2egs}.

\begin{figure}\centering
\includegraphics[width=6in]{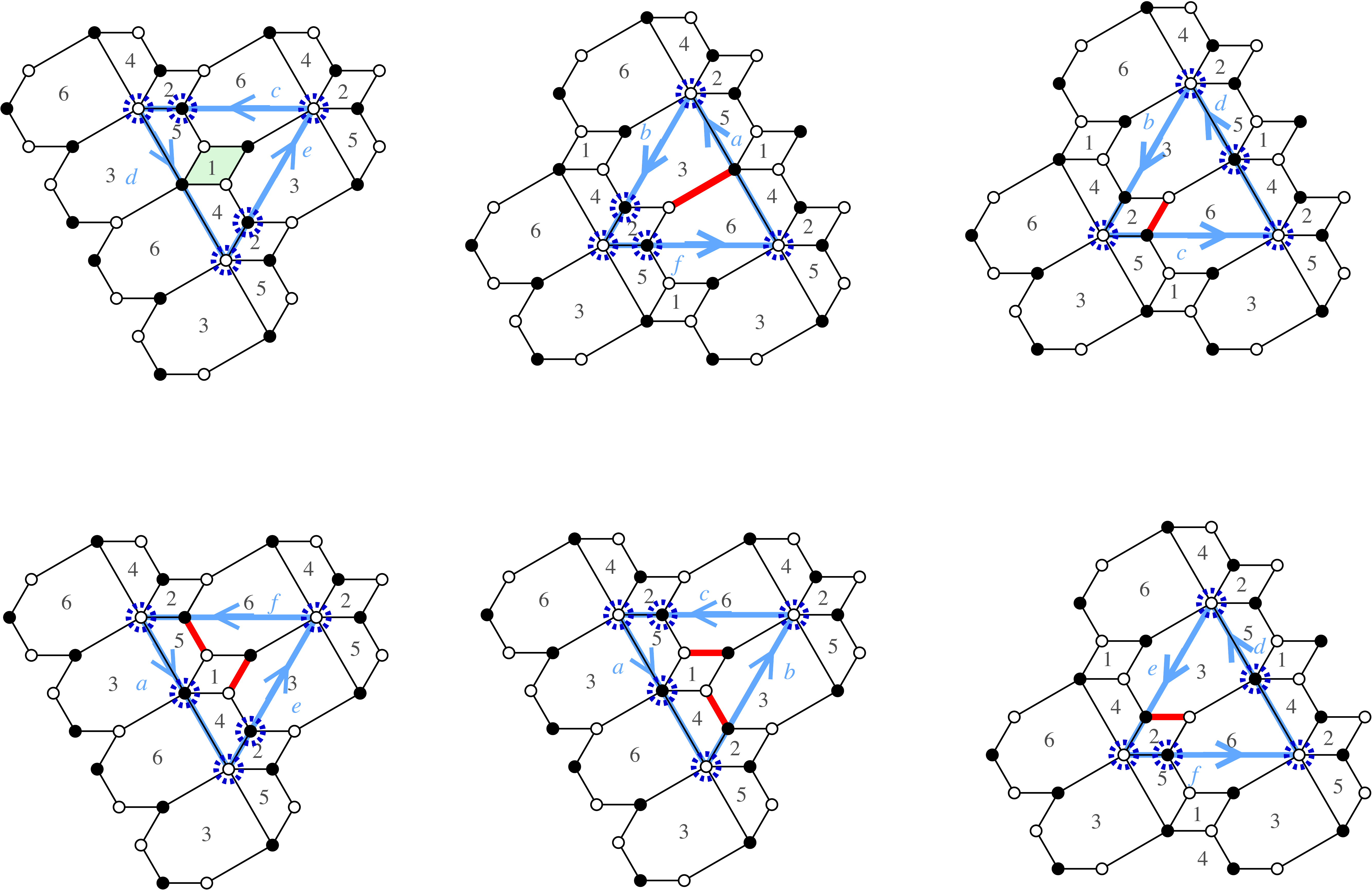}
\caption{Contours $C_1^{(2)},C_2^{(2)},\dots, C_6^{(2)}$, respectively for Model 2. Blue sundials indicate vertices along the contours which should be removed.
}
\label{fig:Model2initial}
\end{figure}

\begin{figure}\centering
\includegraphics[width=2in]{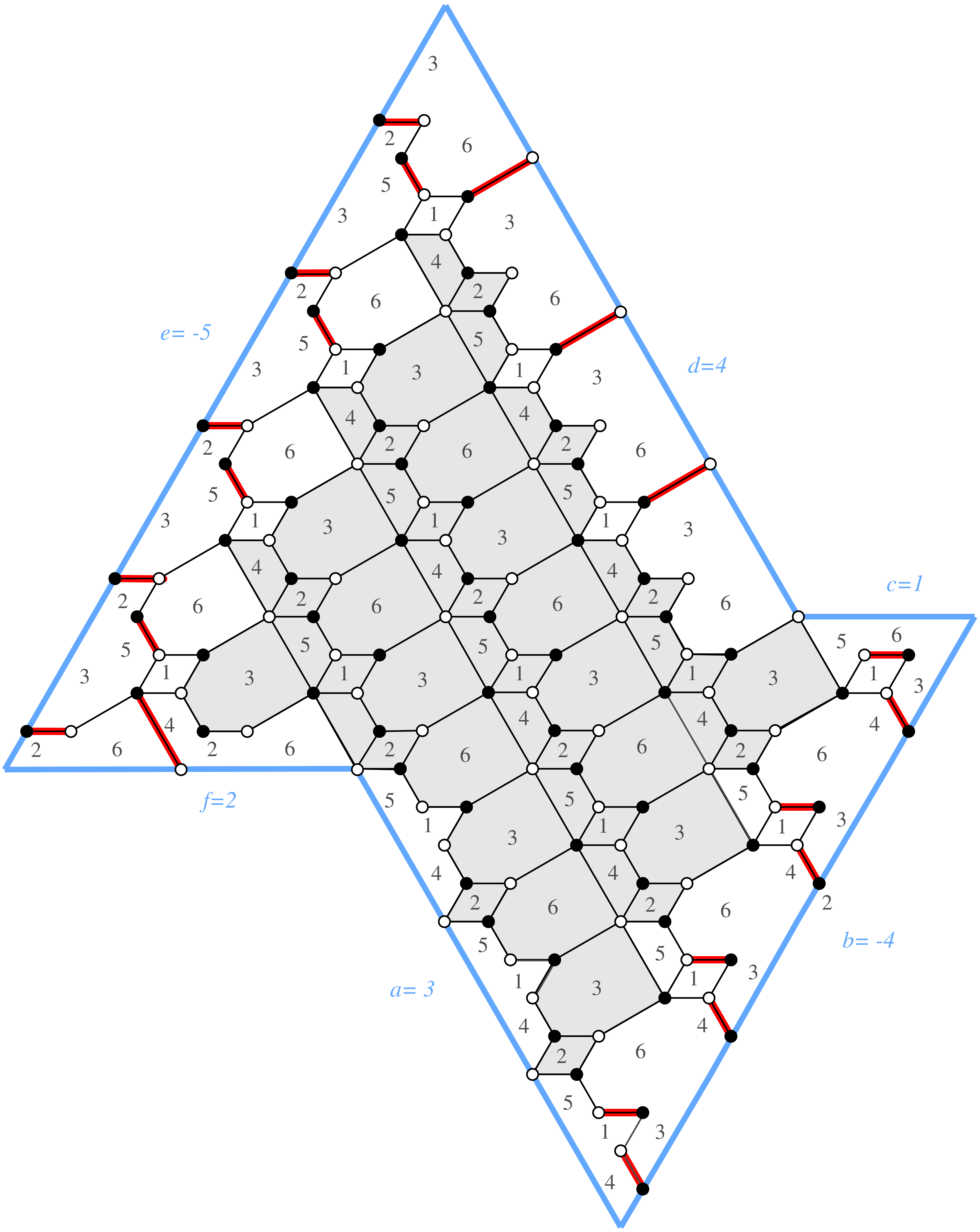} ~ \includegraphics[width=2.5in]{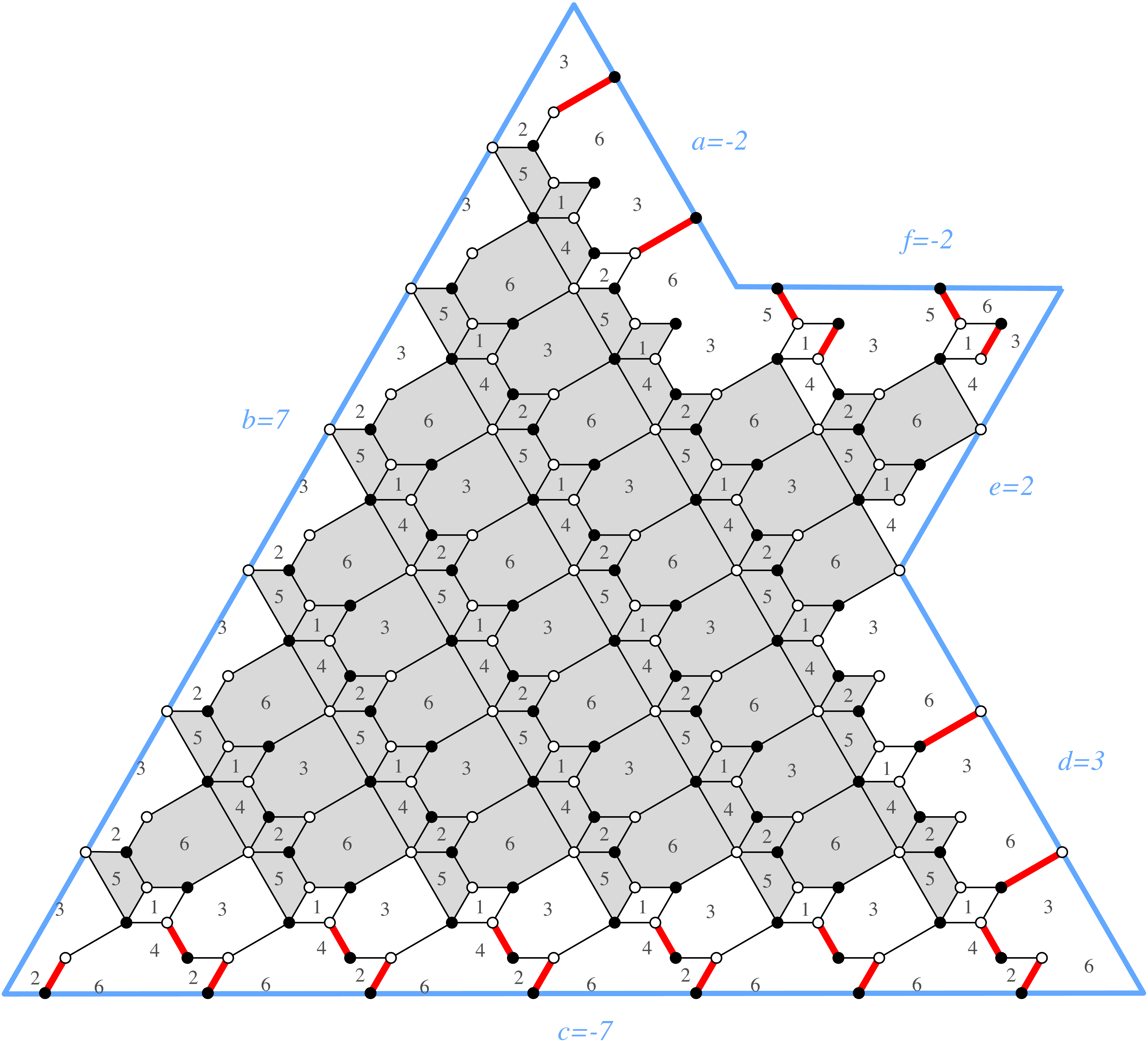} ~ \includegraphics[width=2in]{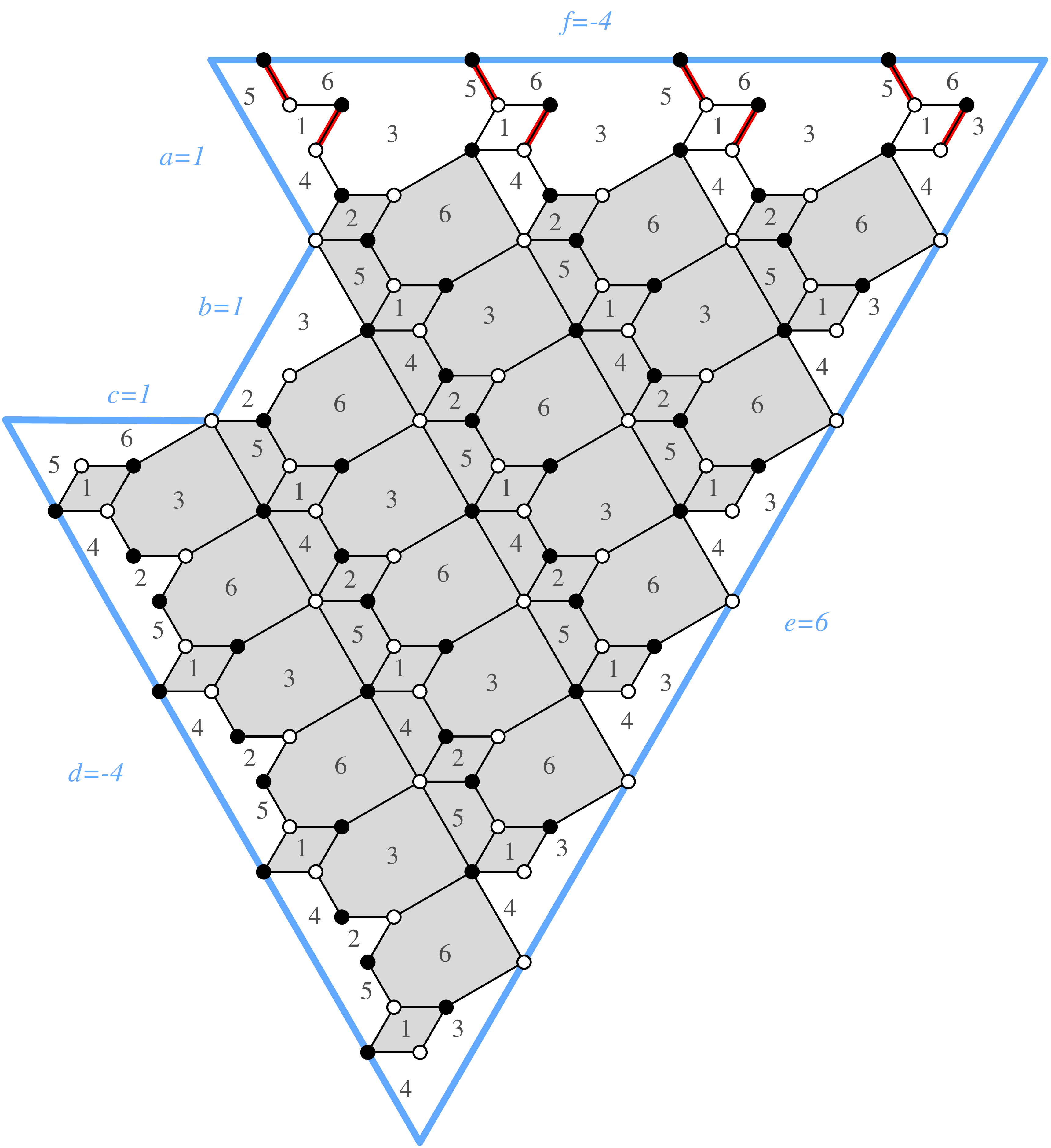}
\caption{Examples of larger contours and the corresponding subgraphs for Model 2.}
\label{fig:Model2egs}
\end{figure}

\subsection{Model 3}

For a given $(a,b,c,d,e,f) \in \mathbb{Z}^6$ in the image of this map $\phi$, we begin at one of the white vertices of degree $4$ in the brane tiling associated to $(Q_3,W_3)$ that borders the faces $2, 3, 5,$ and $6$.  Analogous to the Model 1 and 2 cases, this is a choice of vertex that is unique in the fundamental domain of this brane tiling.  We then follow the lines as illustrated in Figure \ref{fig:Model3-contour}.  Again, we have shown the orientations if the entry is positive.  For a negative entry, we instead traverse the indicated trajectory in the opposite direction.  The absolute value of an entry of this six-tuple indicates the length of the contour in that direction, where a segment of length one ends at the next translate of the initial vertex.  Unlike the Model 2 case, more than one type of white vertex (and more than one type of black vertex) can appear along a segment of length one.  In particular, along the sides labeled $b$ and $e$, a segment of length one contains an additional white vertex in its interior.  Despite this change along sides $b$ and $e$,  the other four sides behave just as in the Model 2 case.  We abbreviate this contour as $\mathcal{C}_3(a,b,c,d,e,f)$.  We next describe how to translate these contours into subgraphs.  We follow the same prescription as above.

\begin{figure}
\includegraphics[width=3.5in]{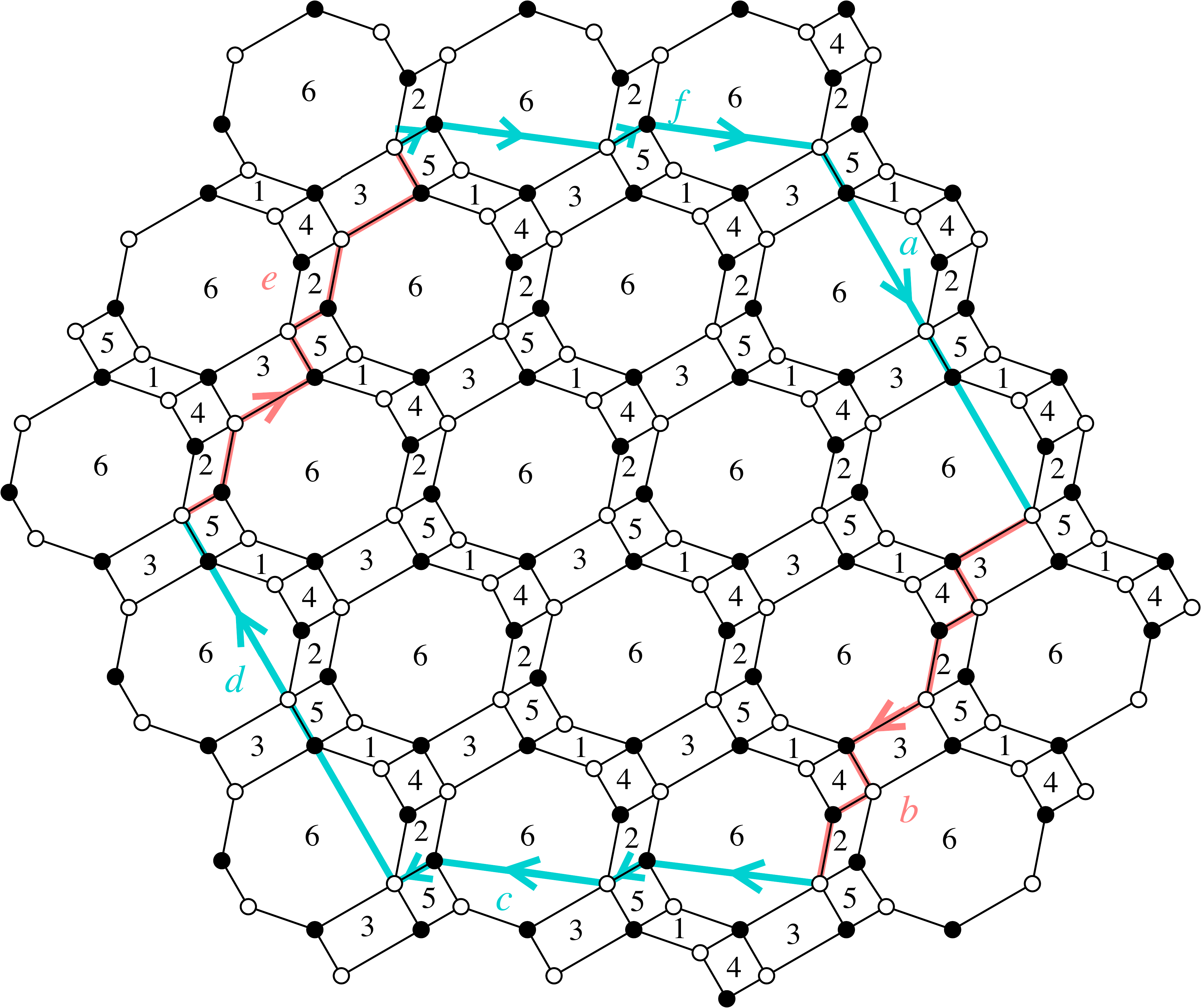}
\caption{Contour for $(Q_3,W_3)$, on top of the brane tiling $\mathcal{T}_3$, in the positive directions.  Segments of length two indicated.}
\label{fig:Model3-contour}
\end{figure}

\begin{definition} \label{def:Model3subgraph}
Suppose that the contour $\mathcal{C}_3(a,b,c,d,e,f)$ does not intersect itself (although back-tracking is allowed).  Under this assumption, we use the
contour $\mathcal{C}_3(a,b,c,d,e,f)$ to define two subgraphs, $\widetilde{\mathcal{G}}_3(a,b,c,d,e,f)$ and $\mathcal{G}_3(a,b,c,d,e,f)$, of the Model 3 brane tiling $\mathcal{T}_3$ by the following:

Step 1: Superimpose the contour $\mathcal{C}_3(a,b,c,d,e,f)$ onto $\mathcal{T}_3$ starting from a $4$-valent white vertex bordering the faces $2$, $3$, $5$, and $6$ as above.

Step 2: For any side of positive (resp. negative) length, we remove all black (resp. white) vertices along that side.

Step 3: A side of length zero corresponds to a single white vertex.  If one of the adjacent sides is of negative length or also of length zero, then that white vertex is removed during Step 2.  On the other hand, if the side of length zero is adjacent to two sides of positive length, we keep the white vertex. 

Step 4: We define $\widetilde{\mathcal{G}}_3(a,b,c,d,e,f)$ to be the resulting subgraph.  However, this subgraph will often include black vertices of valence one.  After matching these up with the appropriate white vertices and continuing this process until there are no vertices of valence one left, we are left with the simply-connected graph $\mathcal{G}_3(a,b,c,d,e,f)$.
\end{definition}

\begin{figure}\centering
\includegraphics[width=6in]{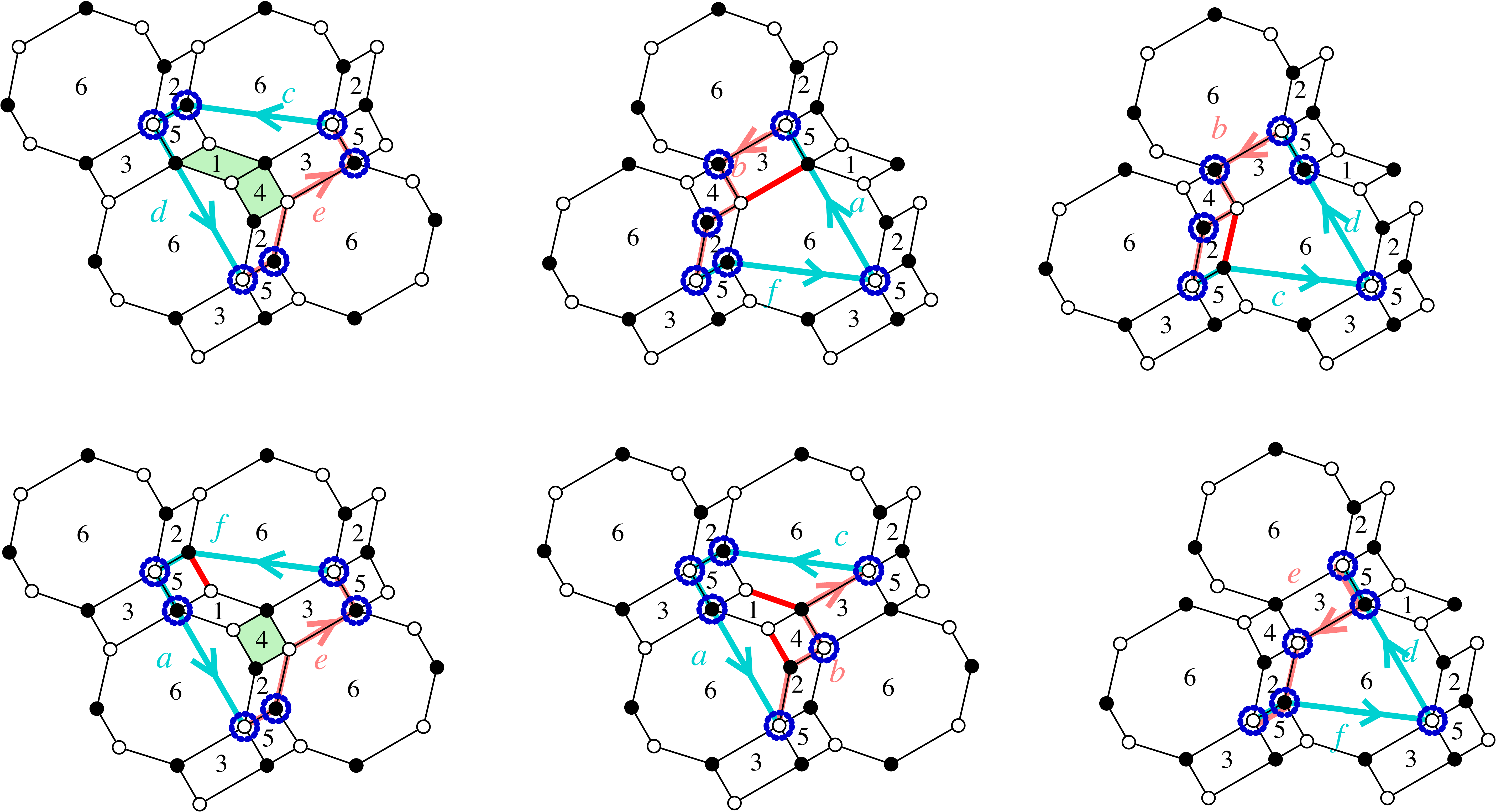}
\caption{Contours $C_1^{(3)},C_2^{(3)},\dots, C_6^{(3)}$, respectively for Model 3.  Blue sundials indicate vertices along the contours which should be removed.}
\label{fig:Model3initial}
\end{figure}

\begin{figure}\centering
\includegraphics[width=2in]{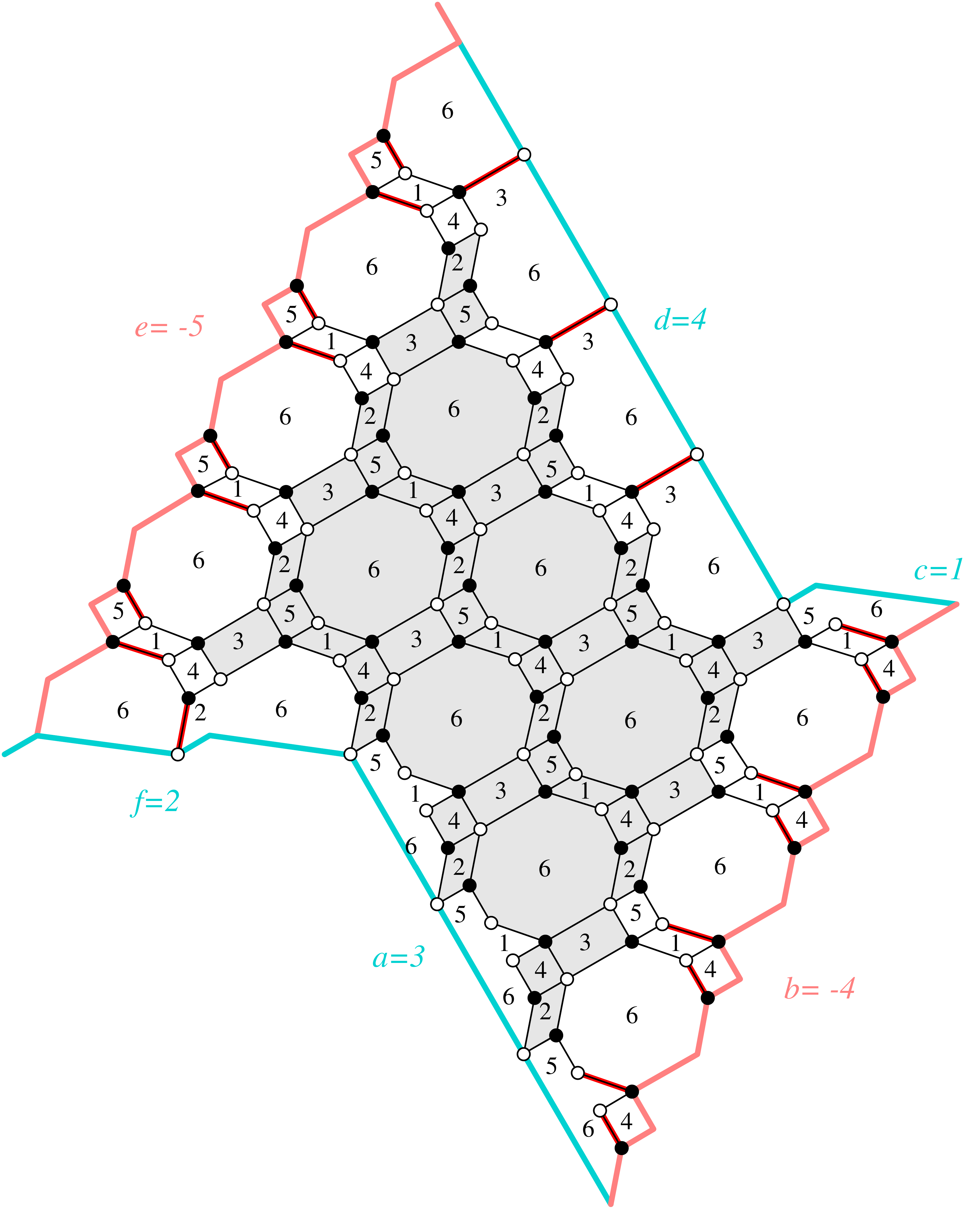} ~ \includegraphics[width=2.5in]{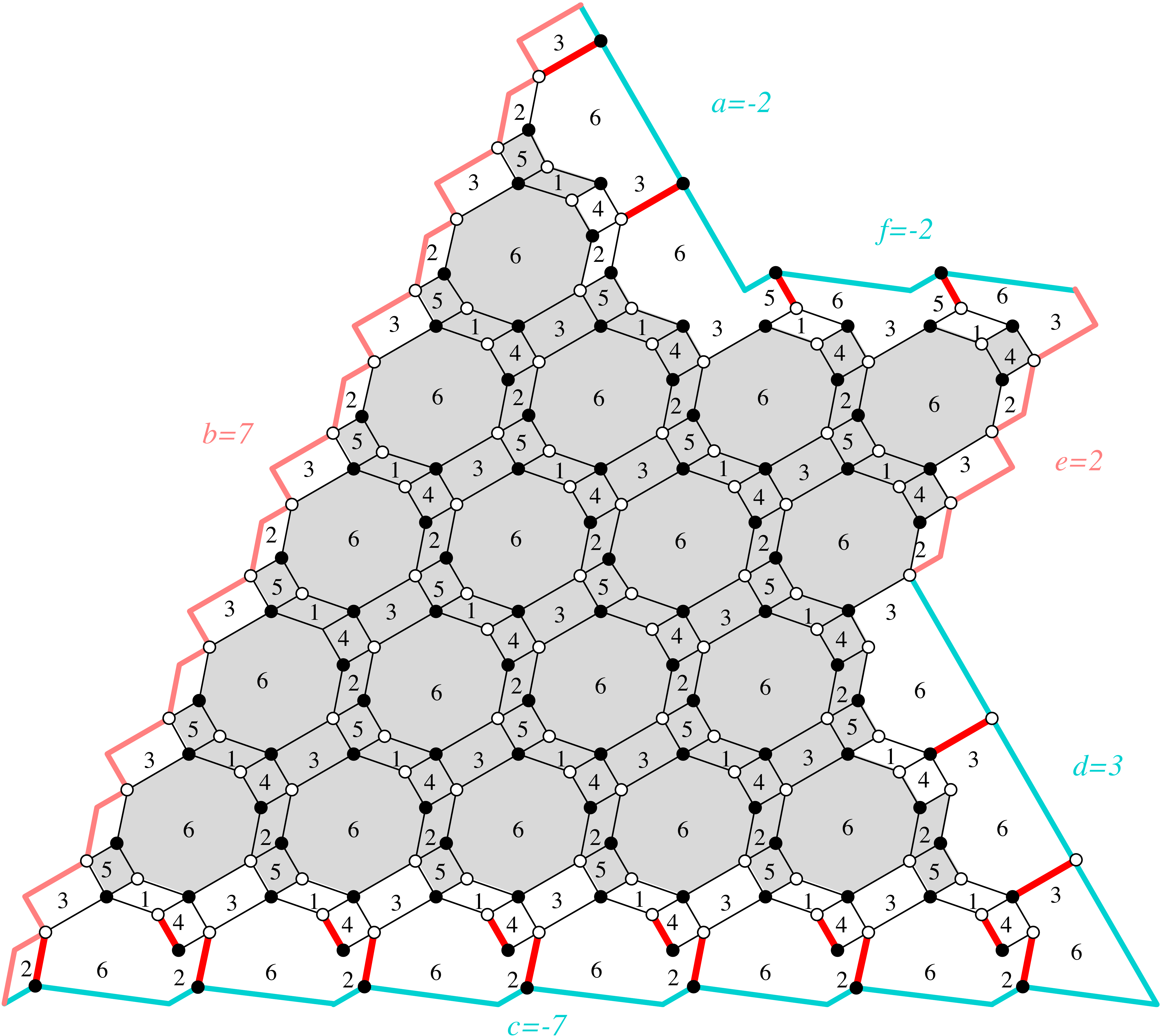} ~ \includegraphics[width=2in]{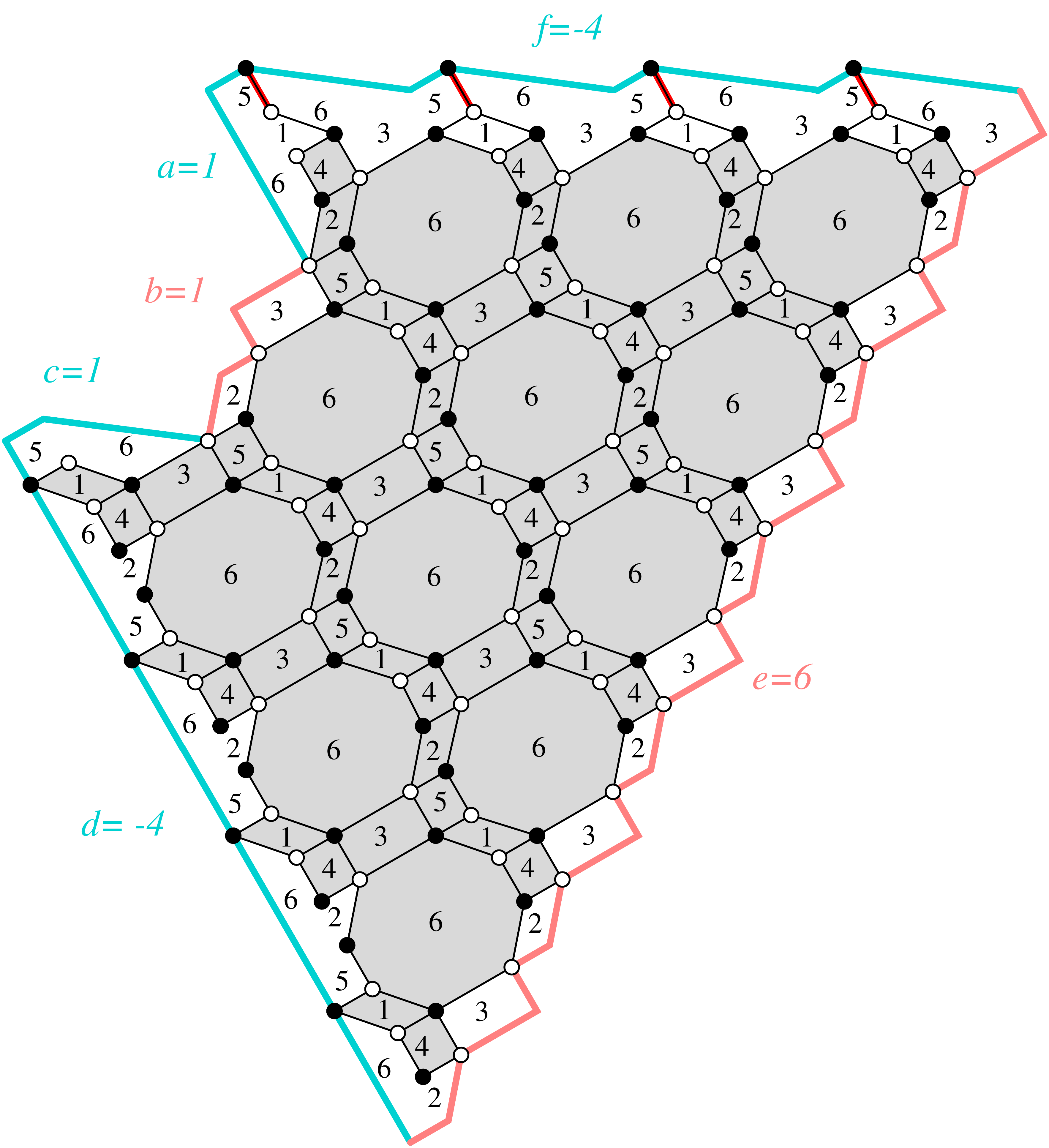}
\caption{Examples of larger contours and the corresponding subgraphs for Model 3.}
\label{fig:Model3egs}
\end{figure}

For the $6$-tuples corresponding to the image of $\phi(\Delta_1)$, we draw the initial contours $C_1^{(3)},C_2^{(3)},\dots, C_6^{(3)}$ and the associated subgraphs in Figure \ref{fig:Model3initial}.  In particular, the initial contour $C_1^{(3)}$ yields a subgraph $\mathcal{G}_3(0,0,1,-1,1,0)$ consisting of two connected quadrilaterals labeled $1$ and $4$, while the initial contour $C_4^{(3)}$ yields a subgraph $\mathcal{G}_3(1,0,0,0,1,-1)$ consisting solely of a quadrilateral labeled $4$.  The remaining initial contours yield empty subgraphs $\mathcal{G}_3$ although their extended subgraphs $\widetilde{\mathcal{G}}_3$ contain some edges incident to $1$-valent black vertices.  Notice that in the case of $C_6^{(3)}$, the white and black vertices are deleted along the sides of the contours just as described by Steps 2 and 3, despite the fact that this contour involves some back-tracking between edges $d$ and $e$ as well as between edges $e$ and $f$.  Some further examples appear in Figure \ref{fig:Model3egs}.

We obtain Laurent polynomial expressions from the graphs $\mathcal{G}_3(a,b,c,d,e,f)$'s by the same process as above.  For this, we must extend our definition of covering monomials to the Model 3 case.  We now have octagonal faces, which the contours lines can cut through.  However, we again see that the remaining faces are quadrilaterals and contours do not cut through such faces.

\begin{definition}
Given a contour $\mathcal{C}_3(a,b,c,d,e,f)$ which defines the extended subgraph $\widetilde{\mathcal{G}}_3 = \widetilde{\mathcal{G}}_3(a,b,c,d,e,f)$, we define the {\bf covering monomial for Model 3}
$m(\widetilde{\mathcal{G}}_3)$ as the product $x_1^{a_1}x_2^{a_2}x_3^{a_3}x_4^{a_4}x_5^{a_5}x_6^{3a_6+2b_6 + c_6}$ where $a_j$ is the number of faces labeled $j$ which lie fully inside the contour $\mathcal{C}_3(a,b,c,d,e,f)$, $b_6$ equals the number of octagonal faces labeled $6$ which partially live inside the contour as a hexagon, and $c_6$ equals the number of octagonal faces labeled $6$ which partially live inside the contour as a quadrilateral.
\end{definition}

We then compute the weight of the subgraph $\mathcal{G}_3 = \mathcal{G}_3(a,b,c,d,e,f)$ as
$$w(\mathcal{G}_3) = \sum_{M \in \mathcal{M}(\mathcal{G}_3)} w(M)$$
where $\mathcal{M}(\mathcal{G}_3)$ is the set of perfect matchings $M$ of subgraph $\mathcal{G}_3$, and the weight of a matching $M$ is the product
$w(M) = \prod_{e_{ij} \in M} \frac{1}{x_i x_j}$.
Here $e_{ij}$ is an edge in the perfect matching $M$ which borders the faces labeled $i$ and $j$.

A variant of Remark \ref{rem:CM-variant} applies in this case.  We break all octagonal faces labeled $6$ into quadrilaterals as illustrated in Figure \ref{fig:CM-quads3}.  Then the covering monomial is defined as the usual product of all $x_j$ where $j$ is the label of a quadrilateral face contained inside the contour.

\begin{figure}\centering
\includegraphics[width=1.5in]{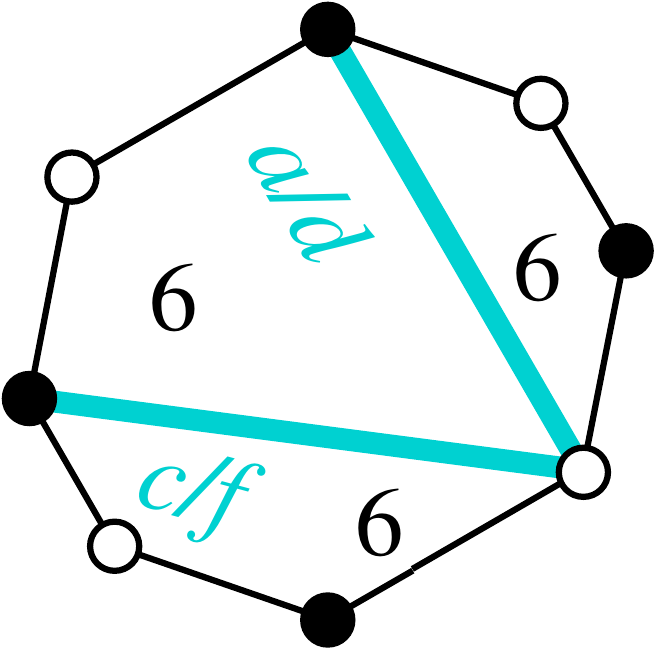}
\caption{Cutting Octagonal Face $6$ into Quadrilaterals in the Model 3 case.}
\label{fig:CM-quads3}
\end{figure}

\begin{theorem} \label{thm:formula3} Let $(a,b,c,d,e,f) \in \mathbb{Z}^6$ be the image $\phi(i,j,k)$ for $(i,j,k) \in \mathbb{Z}^3$.  Then as long as the contour $\mathcal{C}_3(a,b,c,d,e,f)$ has no self-intersections, then
$$z_{i,j,k}^{(3)} = m(\widetilde{\mathcal{G}}_3(a,b,c,d,e,f)) \cdot w(\mathcal{G}_3(a,b,c,d,e,f)).$$
\end{theorem}

\subsection{Model 4}

\begin{figure}\centering
\includegraphics[width=4in]{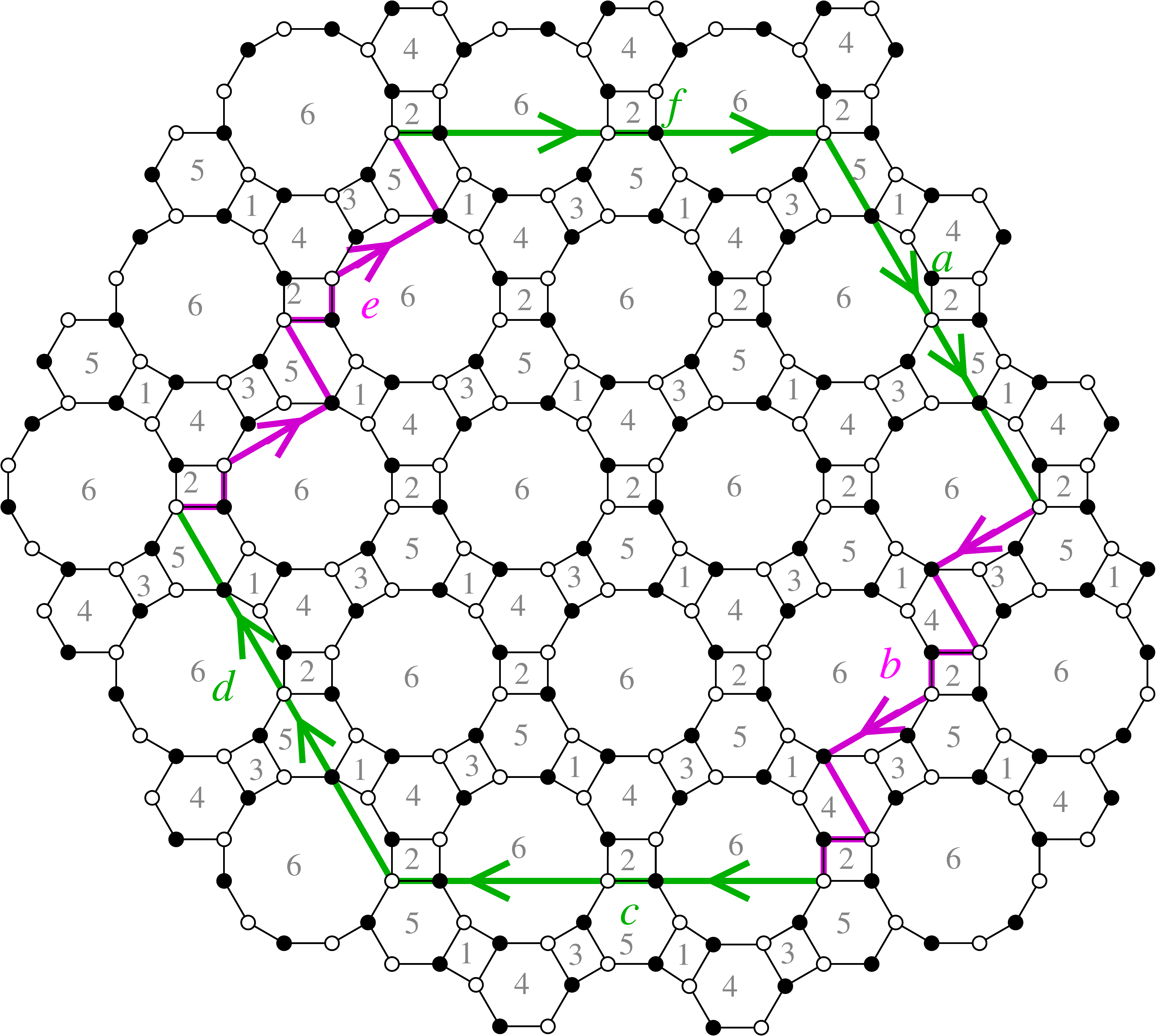}
\caption{Contour for $(Q_4,W_4)$, on top of the brane tiling $\mathcal{T}_4$, in the positive directions.  Segments of length two indicated.}
\label{fig:Model4-contour}
\end{figure}

For a given $(a,b,c,d,e,f) \in \mathbb{Z}^6$ in the image of this map $\phi$, we begin at one of the white vertices of degree $3$ in the brane tiling associated to $(Q_4,W_4)$ that borders the faces $2, 5,$ and $6$.  Analogous to the other models, this is a choice of vertex that is unique in the fundamental domain of this brane tiling.  We then follow the lines as illustrated in Figure \ref{fig:Model4-contour}.  Again, we have shown the orientations if the entry is positive.  For a negative entry, we instead traverse the indicated trajectory in the opposite direction.  The absolute value of an entry of this six-tuple indicates the length of the contour in that direction, where a segment of length one ends at the next translate of the initial vertex. Like the Model 3 case, more than one type of white vertex (and more than one type of black vertex) can appear along a segment of length one; in particular, along the sides labeled $b$ and $e$.  The other four sides behave just as in the Model 2 case.  We abbreviate this contour as $\mathcal{C}_4(a,b,c,d,e,f)$.  We translate these contours into subgraphs just as above.

\begin{definition} \label{def:contour4}
Suppose that the contour $\mathcal{C}_4(a,b,c,d,e,f)$ does not intersect itself (although back-tracking is allowed).  Under this assumption, we use the
contour $\mathcal{C}_4(a,b,c,d,e,f)$ to define two subgraphs, $\widetilde{\mathcal{G}}_4(a,b,c,d,e,f)$ and $\mathcal{G}_4(a,b,c,d,e,f)$, of the Model 4 brane tiling $\mathcal{T}_4$ by the following.

Step 1: Superimpose the contour $\mathcal{C}_4(a,b,c,d,e,f)$ onto $\mathcal{T}_4$ starting from a $3$-valent white vertex bordering the faces $2$, $5$, and $6$ as above.

Step 2: For any side of positive (resp. negative) length, we remove all black (resp. white) vertices along that side.

Step 3: A side of length zero corresponds to a single white vertex.  If one of the adjacent sides is of negative length or also of length zero, then that white vertex is removed during Step 2.  On the other hand, if the side of length zero is adjacent to two sides of positive length, we keep the white vertex. 

Step 4: We define $\widetilde{\mathcal{G}}_4(a,b,c,d,e,f)$ to be the resulting subgraph.  However, this subgraph will often include black vertices of valence one.  After matching these up with the appropriate white vertices and continuing this process until there are no vertices of valence one left, we are left with the simply-connected graph $\mathcal{G}_4(a,b,c,d,e,f)$.
\end{definition}

\begin{figure}\centering
\includegraphics[width=6in]{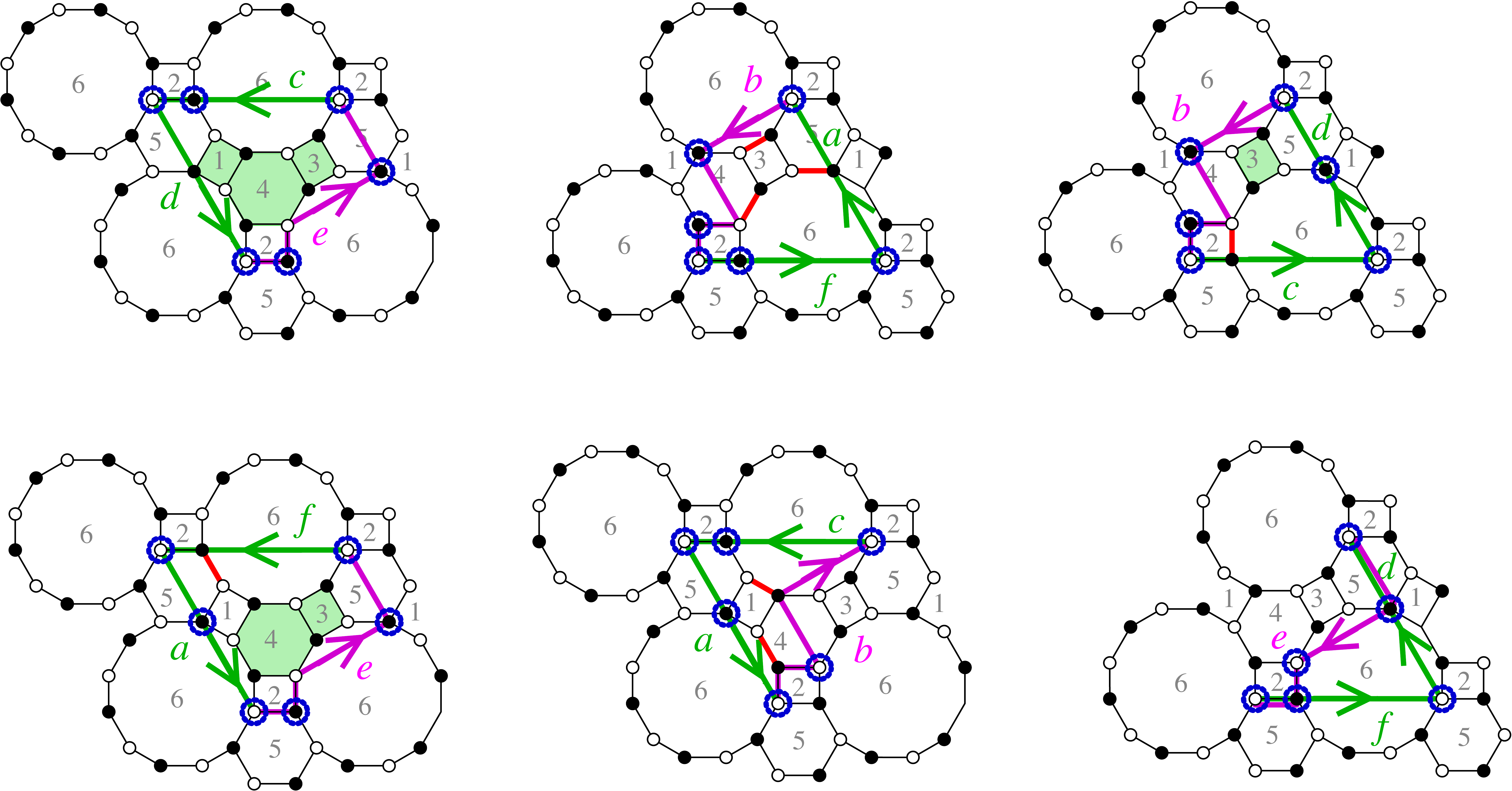}
\caption{Contours $C_1^{(4)},C_2^{(4)},\dots, C_6^{(4)}$, respectively for Model 4.  Blue sundials indicate vertices along the contours which should be removed.}
\label{fig:Model4initial}
\end{figure}

\begin{figure}\centering
\includegraphics[width=2in]{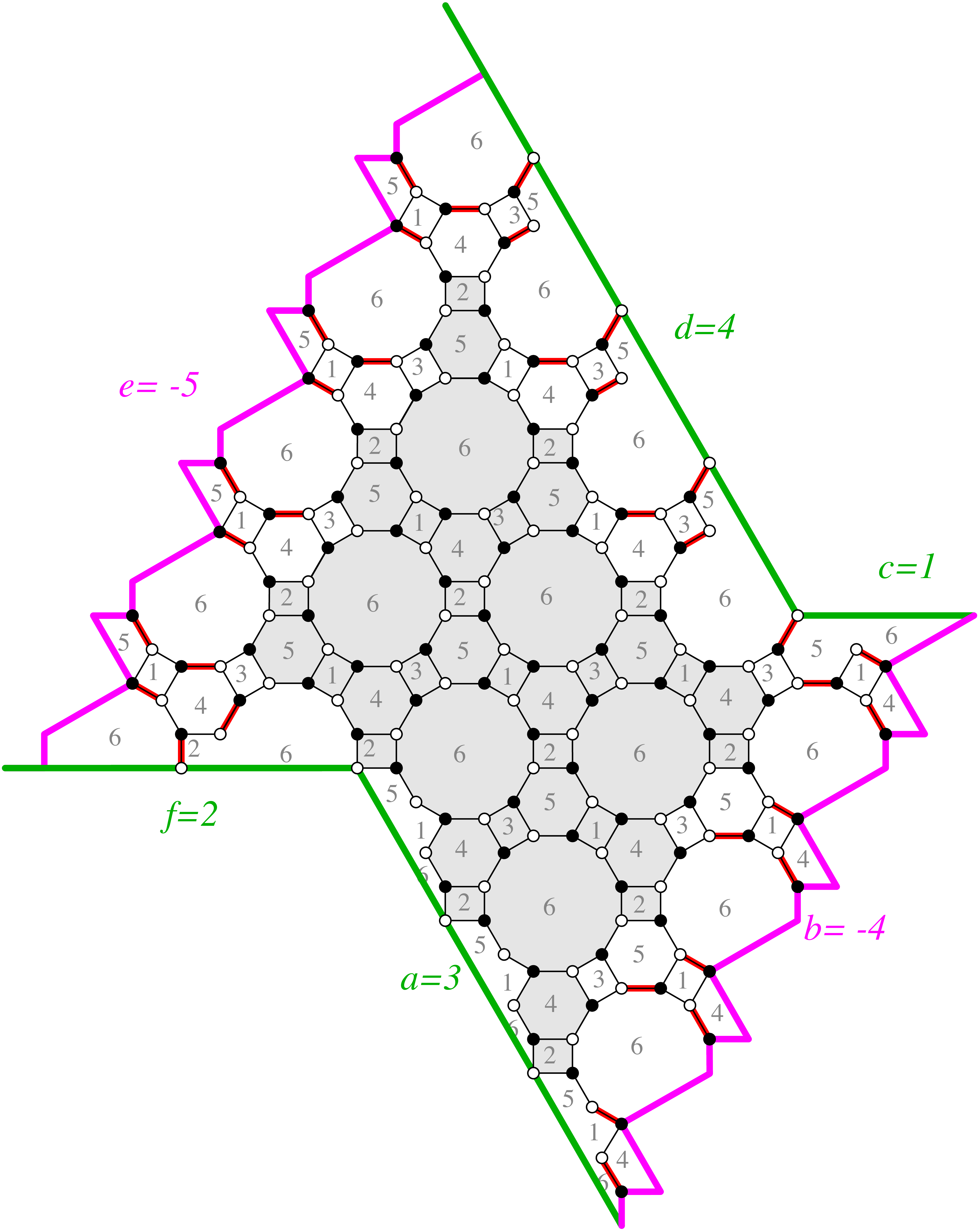} ~ \includegraphics[width=2.5in]{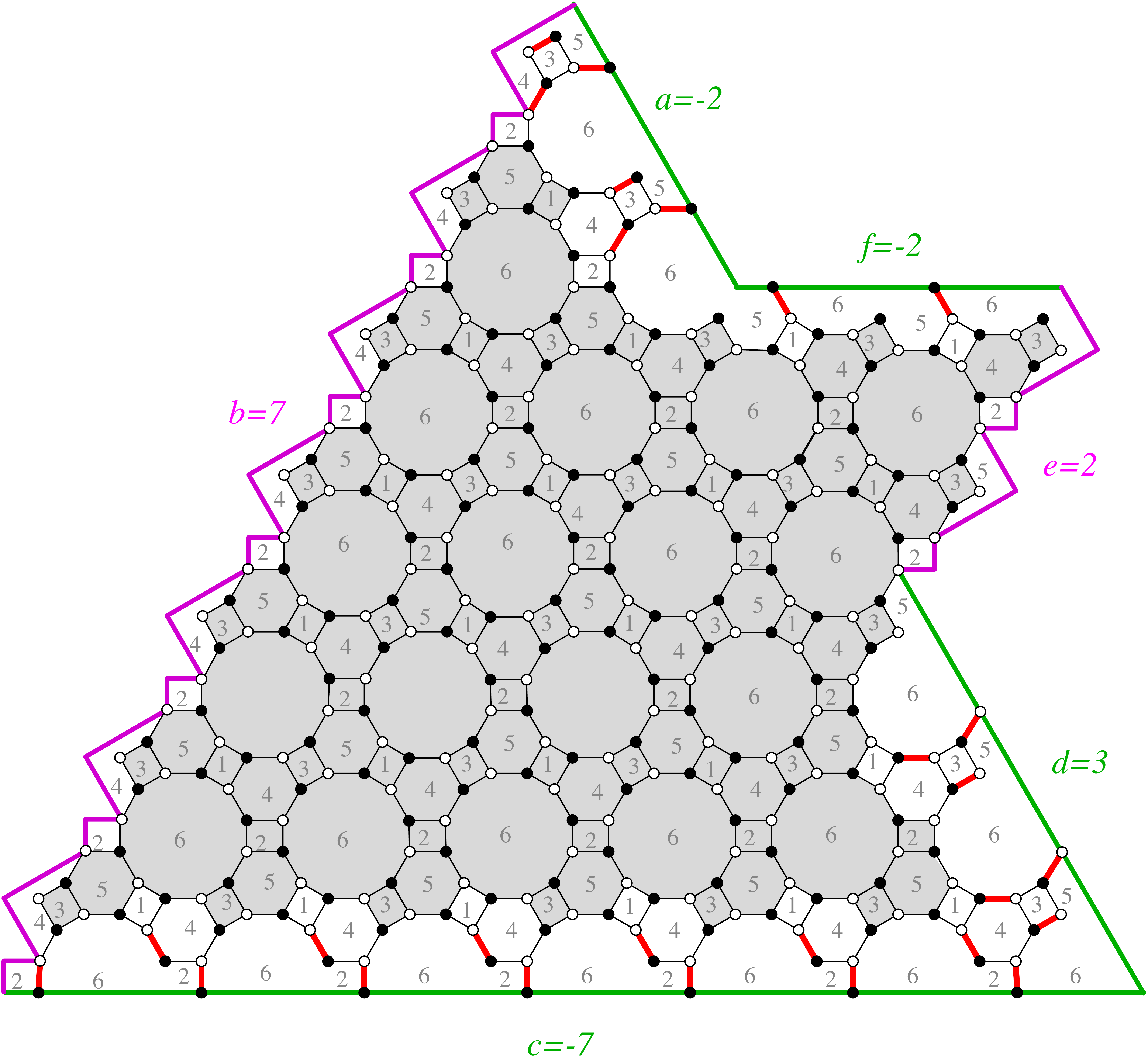} ~ \includegraphics[width=2in]{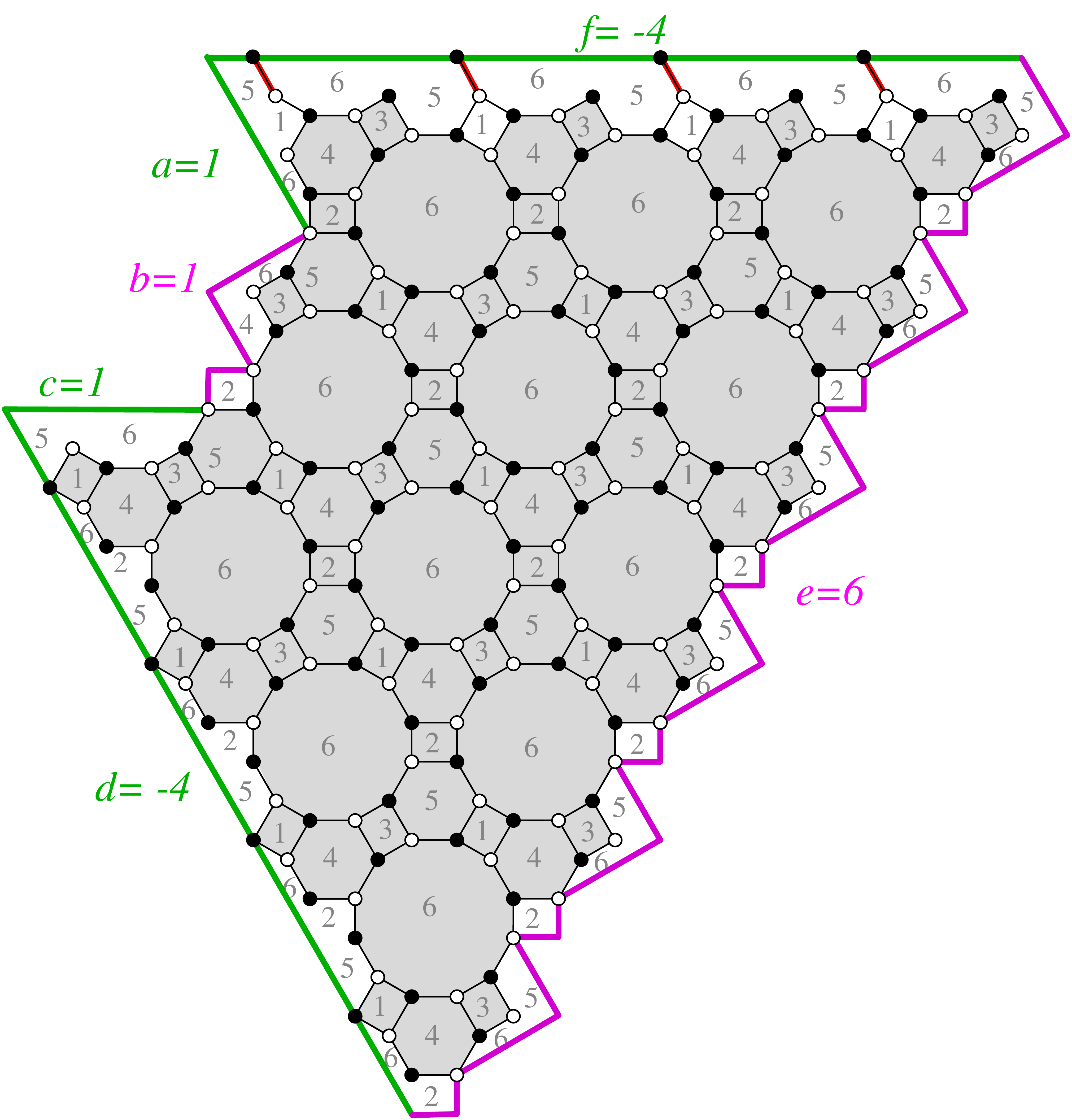}
\caption{Examples of larger contours and the corresponding subgraphs for Model 4.}
\label{fig:Model4egs}
\end{figure}

For the $6$-tuples corresponding to the image of $\phi(\Delta_1)$, we draw the initial contours $C_1^{(4)},C_2^{(4)},\dots, C_6^{(4)}$ and the associated subgraphs in Figure \ref{fig:Model4initial}.  In particular, the initial contour $C_1^{(4)}$ yields a subgraph $\mathcal{G}_4(0,0,1,-1,1,0)$ consisting of three connected quadrilaterals labeled $1$, $4$, and $3$, while the initial contour $C_4^{(4)}$ yields a subgraph $\mathcal{G}_4(1,0,0,0,1,-1)$ consisting of two connected quadrilaterals labeled $4$ and $3$, and $C_3^{(4)}$ yields $\mathcal{G}_4(0,1,-1,1,0,0)$ consisting solely of a quadrilateral labeled $3$.  The remaining initial contours yield empty subgraphs $\mathcal{G}_4$ although their extended subgraphs $\widetilde{\mathcal{G}}_4$ contain some edges incident to $1$-valent black vertices.  Note again that in the case of $C_6^{(4)}$, the white and black vertices are deleted along the sides of the contours, just as described by Steps 2 and 3, despite the fact that this contour involves some back-tracking between edges $d$ and $e$ as well as between edges $e$ and $f$.  Some further examples appear in Figure \ref{fig:Model4egs}.

We obtain Laurent polynomial expressions from the graphs $\mathcal{G}_4(a,b,c,d,e,f)$'s by the same process as above.  For this, we must extend our definition of covering monomials to the Model 4 case.  We now have dodecagonal faces and hexagonal faces, which the contours lines can cut through.  Again, the remaining faces are quadrilaterals and contours do not cut through such faces.
\begin{definition} \label{def:CM4}
Given a contour $\mathcal{C}_4(a,b,c,d,e,f)$ which defines the extended subgraph $\widetilde{\mathcal{G}}_4 = \widetilde{\mathcal{G}}_4(a,b,c,d,e,f)$, we define the {\bf covering monomial for Model 4}
$m(\widetilde{\mathcal{G}}_4)$ as the product $x_1^{a_1}x_2^{a_2}x_3^{a_3}x_4^{2a_4+e_4}x_5^{2a_5+e_5}x_6^{5a_6 + 4b_6 + 3c_6+2d_6 + e_6}$ where $a_j$ is the number of faces labeled $j$ which lie fully inside the contour $\mathcal{C}_4(a,b,c,d,e,f)$, $b_j$ (resp. $c_j$, $d_j$, $e_j$) equals the number of faces labeled $j$ which partially live inside the contour as a 10-gon (resp. octagon, hexagon, quadrilateral).
\end{definition}

We then compute the weight of the subgraph $\mathcal{G}_4 = \mathcal{G}_4(a,b,c,d,e,f)$ as
$$w(\mathcal{G}_4) = \sum_{M \in \mathcal{M}(\mathcal{G}_4)} w(M)$$
where $\mathcal{M}(\mathcal{G}_4)$ is the set of perfect matchings $M$ of subgraph $\mathcal{G}_4$, and the weight of a matching $M$ is the product
$w(M) = \prod_{e_{ij} \in M} \frac{1}{x_i x_j}$.
Here $e_{ij}$ is an edge in the perfect matching $M$ which borders the faces labeled $i$ and $j$.

A variant of Remark \ref{rem:CM-variant} again applies in this case.  We break all dodecagonal faces labeled $6$ and hexagonal faces labeled $4$ or $5$ into quadrilaterals as illustrated in Figure \ref{fig:CM-quads4}.  Then the covering monomial is defined as the usual product of all $x_j$ where $j$ is the label of a quadrilateral face contained inside the contour.

\begin{figure}\centering
\includegraphics[width=4.5in]{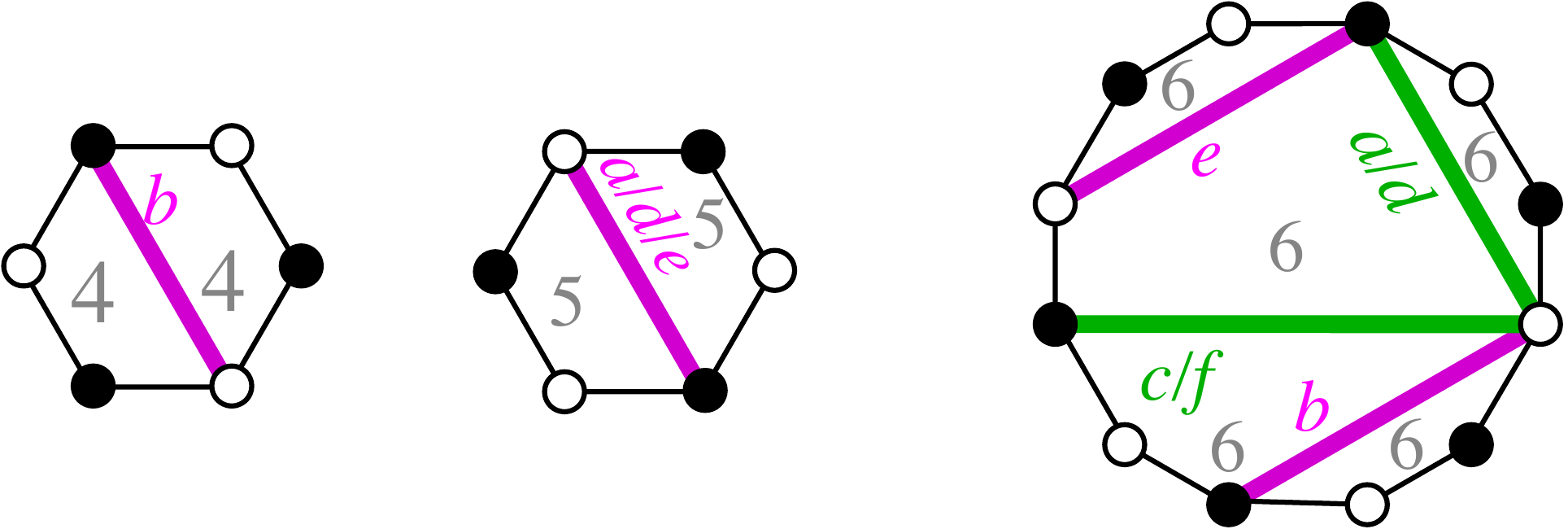}
\caption{Cutting Hexagonal Faces $4$ and $5$ and Dodecagonal Face $6$ into Quadrilaterals in the Model 4 case.}
\label{fig:CM-quads4}
\end{figure}

\begin{theorem} \label{thm:formula4} Let $(a,b,c,d,e,f) \in \mathbb{Z}^6$ be the image $\phi(i,j,k)$ for $(i,j,k) \in \mathbb{Z}^3$.  Then as long as the contour $\mathcal{C}_4(a,b,c,d,e,f)$ has no self-intersections, then
$$z_{i,j,k}^{(4)} = m(\widetilde{\mathcal{G}}_4(a,b,c,d,e,f)) \cdot w(\mathcal{G}_4(a,b,c,d,e,f)).$$
\end{theorem}

\section{Combinatorial History and its Relation to the Current Work}

\label{sec:history}

In 1999, Jim Propp published an article \cite{Propp} tracking the progress on a list of 20 open
problems in the field of exact enumeration of perfect matchings, which he presented
in a lecture in 1996, as part of the special program on algebraic combinatorics
organized at MSRI during the academic year 1996--1997. The article also presented
a list of 12 new open problems. In a review on MathSciNet of the American Mathematical Society (AMS), Christian Krattenthaler (University of Vienna) wrote about this list of 32 problems: ``\emph{This list of problems was very influential; it called forth tremendous activity, resulting in the solution of several of these problems (but by no means all), in the development of interesting new techniques, and, very often, in results that move beyond the problems.}"

On this list of problems, Jim Propp posed a number of analogues of the well-known Aztec Diamond (see Figure \ref{fig:AztecDiamond}) in different lattices, including the \emph{Aztec Dragon} and the \emph{Aztec Dungeon} (see Problem 15 on the list).
In most of the cases, the regions described in Propp's article yield a simple product formula for the number of tilings, in particular a perfect power of small prime numbers.

\begin{figure}\centering
\includegraphics[width=4.5in]{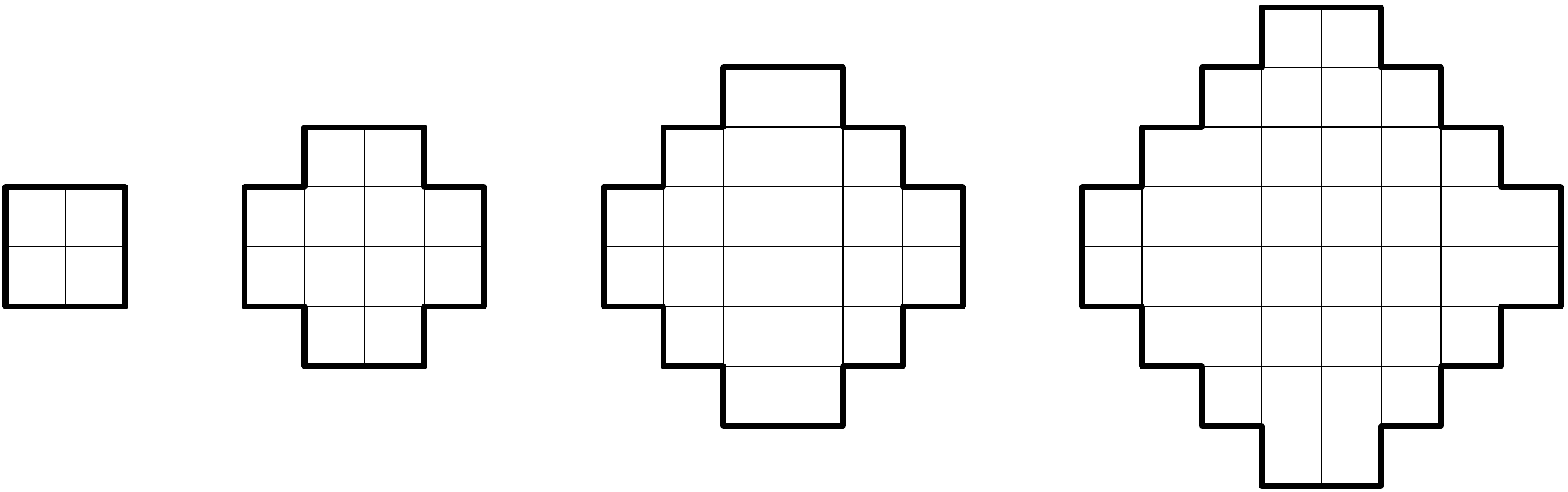}
\caption{The Aztec diamond regions of order $1,2,3,4$ (from left to right).}
\label{fig:AztecDiamond}
\end{figure}

\begin{figure}\centering
\includegraphics[width=4.5in]{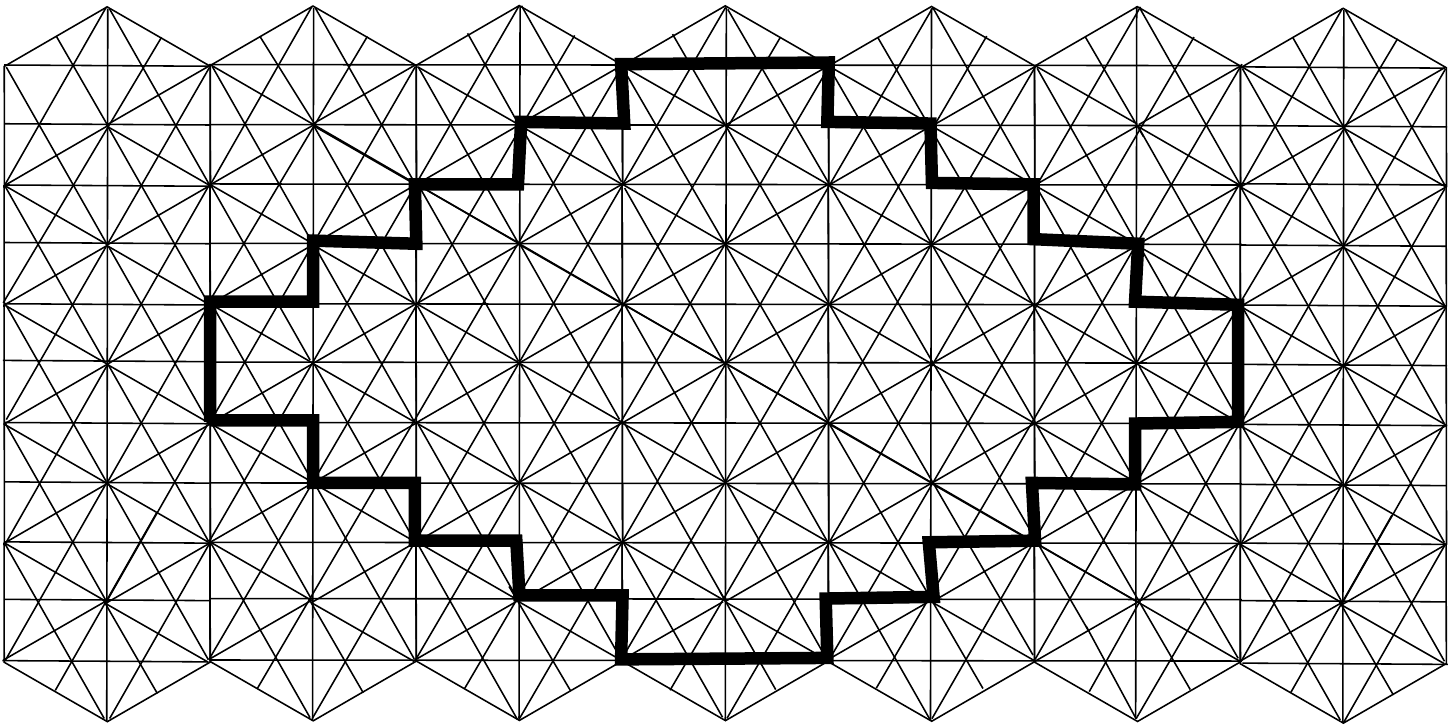}
\caption{The Aztec Dungeon of order 5; Figure 3.6 in \cite{lai'}.}
\label{fig:AztecDungeon}
\end{figure}

\begin{figure}\centering
\includegraphics[width=6in]{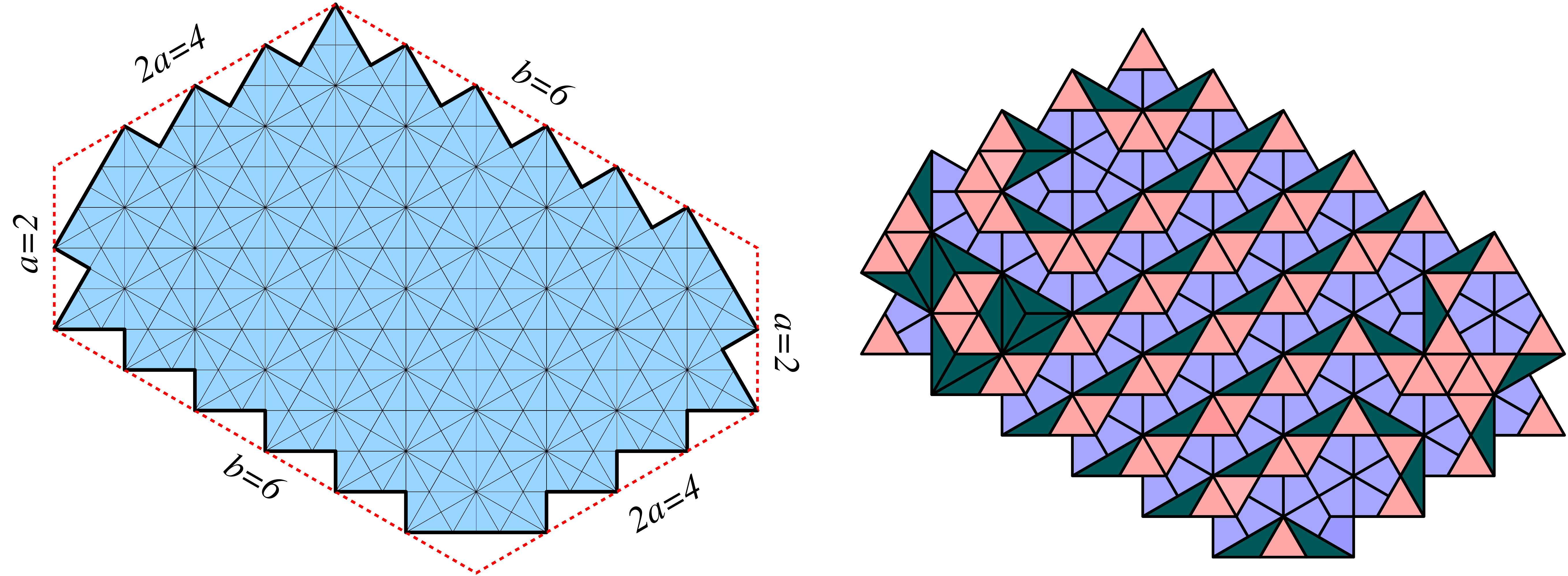}
\caption{A hexagonal dungeon region (Left) and a tiling of its (Right).}
\label{fig:HexagonalDungeon}
\end{figure}

The Aztec Dungeon is a diamond-like region on the $G_2$ lattice (the lattice corresponding with the affine Coxeter group $G_2$). See Figure \ref{fig:AztecDungeon}. Jim Propp conjectured that the number of tilings of an Aztec Dungeon is always a power of $13$ or twice a power of $13$. The conjecture was proved by Mihai Ciucu \cite{Ciucu} by using a linear algebraic version of urban renewal, and later by Kokhas \cite{Kokhas} by using Kuo condensation.
Inspired by Aztec Dungeons, Matt Blum introduced a hexagonal counterpart of them called \emph{Hexagonal Dungeons}. See Figure \ref{fig:HexagonalDungeon}.  He found a striking pattern in the tilings of the Hexagonal Dungeon and conjectured that the number of tilings of a Hexagonal Dungeon is always given by a product of a power of 13 and a power of 14. This conjecture was listed as Problem  25 on the list of Propp. 

The 32 problems  on the list can be divided into $3$ different categories: conjectures stating an explicit formula for the number of tilings (resp., perfect matchings) of the specific family of regions (resp., graphs), problems for which the number of tilings (resp., perfect matchings) does not seem to be given by a simple formula, but presents some nice patterns that are required to be proved, and problems concerned with various aspects of the Kasteleyn matrices of the involved graphs, and not directly with their number of perfect matchings.

To enumerative combinatorialists, the most compelling ones to prove are the problems in the first category. Actually, all of problems in this category were proved shortly after Jim Propp posed them. After few years since the list was published, the only still-open problem in the first category was Blum's conjecture, see Theorem \ref{thm:Blum}. Fourteen years later, the conjecture was proved in 2014 by Ciucu and the first author \cite{lai'} by an application of Kuo condensation to two larger families of six-sided regions. As we explain below in Section \ref{sec:Mod4}, these two larger families correspond to special cases of the subgraphs constructed in the case of Model 4. The first author latter proved a refinement of the conjecture in \cite{lai}. The proof of Blum's conjecture demonstrated that Kuo condensation is a powerful tool in the study of tilings. After this, a number of hard problems in the field of enumeration of tilings have been cracked by using Kuo condensation (we refer the reader to e.g. \cite{CF,CL19,KW,Tri1,Tri2,Halfhex1,Halfhex2, Minor} for recent applications of the method).  

\begin{theorem}[Corollary 3.2 \cite{lai'}; Blum's (ex-)conjecture \cite{Propp}] \label{thm:Blum}
The hexagonal dungeon with sides $a, 2a, b, a, 2a, b$ has
exactly \[13^{2a^2}14^{\lfloor a^2/2\rfloor}\] tilings, for all $b\geq 2a$.
\end{theorem}

\begin{figure}\centering
\includegraphics[width=4.5in]{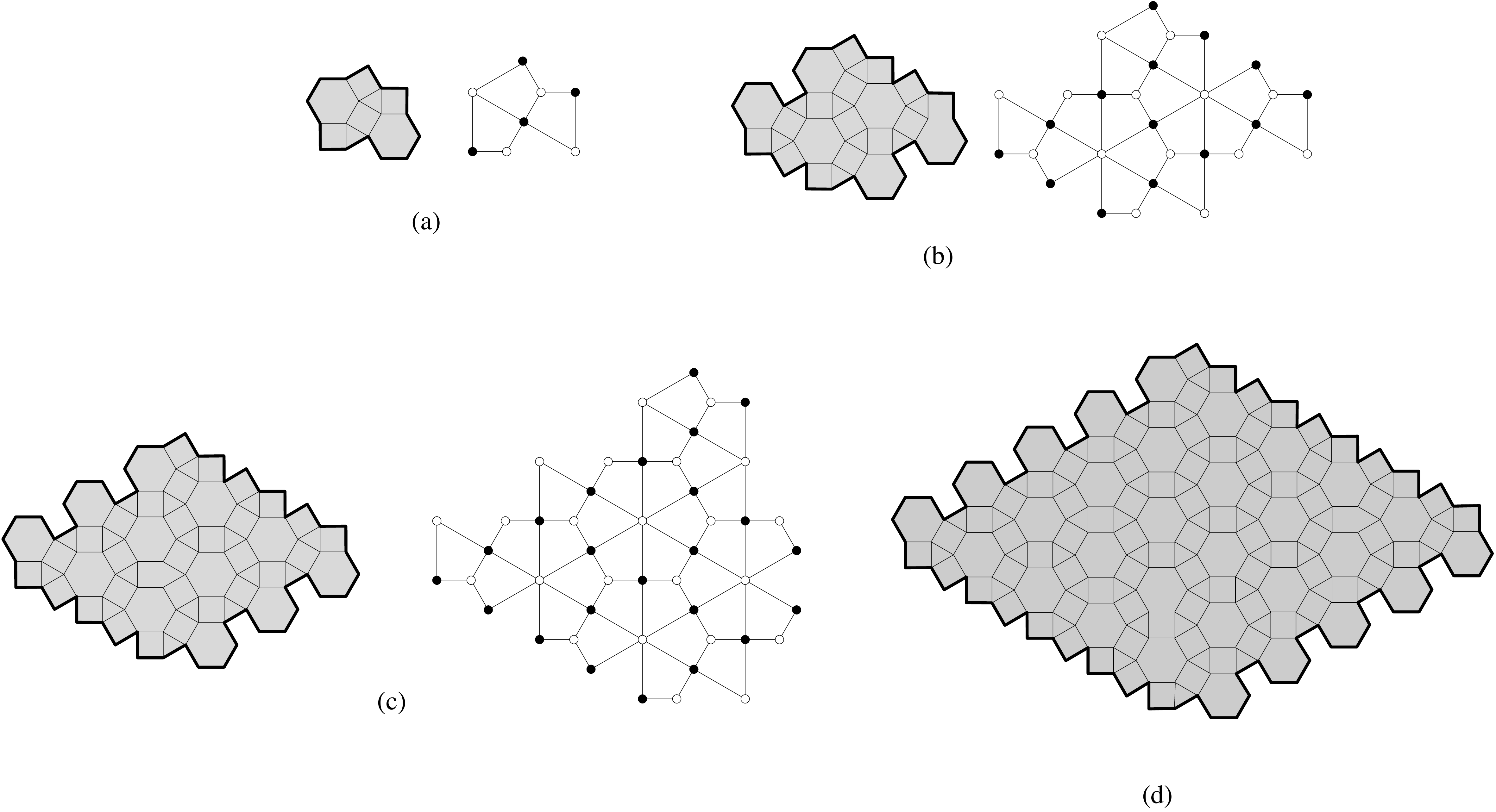}
\caption{ (a)--(c) Aztec Dragon regions (Left) of order $1,2,3$ and their dual graphs (Right). (d) Aztec Dragon region of order 5.}
\label{fig:AztecDragon}
\end{figure}

Originally considered as a separated family of regions from the above Aztec Dungeon, Jim Propp also conjectured that any Aztec Dragon has tilings enumerated by a power of 2. See Figure \ref{fig:AztecDragon}.  This was first proved by Ben Wieland in an unpublished  work (as announced in \cite{Propp}) and by Ciucu in \cite{Ciucu}, using urban renewal. The region was also considered by C. Cottrell and B. Young, in which the Aztec Dragon of half-integer order was introduced \cite{CY}. Sicong Zhang studied the weighted enumeration of Aztec Dragons (of both integer and half-integer order), and their relation to cluster algebras, as part of the 2012 REU\footnote{Research Experiences for Undergraduates (REU) is an NSF-funded program. See more details about this NSF program in the link \url{https://www.nsf.gov/crssprgm/reu/}, and about the REU program in combinatorics at University of Minnesota in the link \url{http://www-users.math.umn.edu/~reiner/REU/REU.html}.} at University of Minnesota \cite{zhang}.

 It turned out that the Aztec Dragons (of integer and half-integer orders) are only a special case of larger families that were found independently by the first author \cite{LaiNewDungeon, lai} in his effort to find similar regions to that in his proof of Blum's conjecture, and by the second author in joint work with 2013 REU   students on combinatorial interpretations of cluster algebras \cite{LMNT} (extending Zhang's work).  These families of regions were generalized even further to our family of subgraphs in the Model 1 case, as fully described in \cite{LaiMus}.

In 1999, Ciucu, when applying urban renewal to the dual graph of the hexagonal dungeon, found an interesting family of Aztec rectangles with two corners cut off called \emph{trimmed Aztec rectangles} \cite{Ciucutrim}.  See Figure \ref{fig:TrimmedAztecRectangle}.  He conjectured that the number of perfect matchings of a trimmed Aztec rectangle is given by perfect powers of prime numbers less than 13. The first author  \cite{trim} proved the conjectures by enumerating matchings of six families of six-sided subgraphs of the square grid. Those graphs are special cases of our subgraphs in Model 3, see Section \ref{sec:Mod3}.

In summary, three different looking problems in enumerative combinatorics can all be viewed as consequences of our theorems giving combinatorial interpretations to toric cluster variables for the $dP_3$ quiver.  Furthermore, we obtain certain $6$-variate weighted deformations of these tiling enumerations as part of this approach.  We now proceed to prove the combinatorial interpretations stated in Section \ref{sec:comb}.

\begin{figure}\centering
\includegraphics[width=4in]{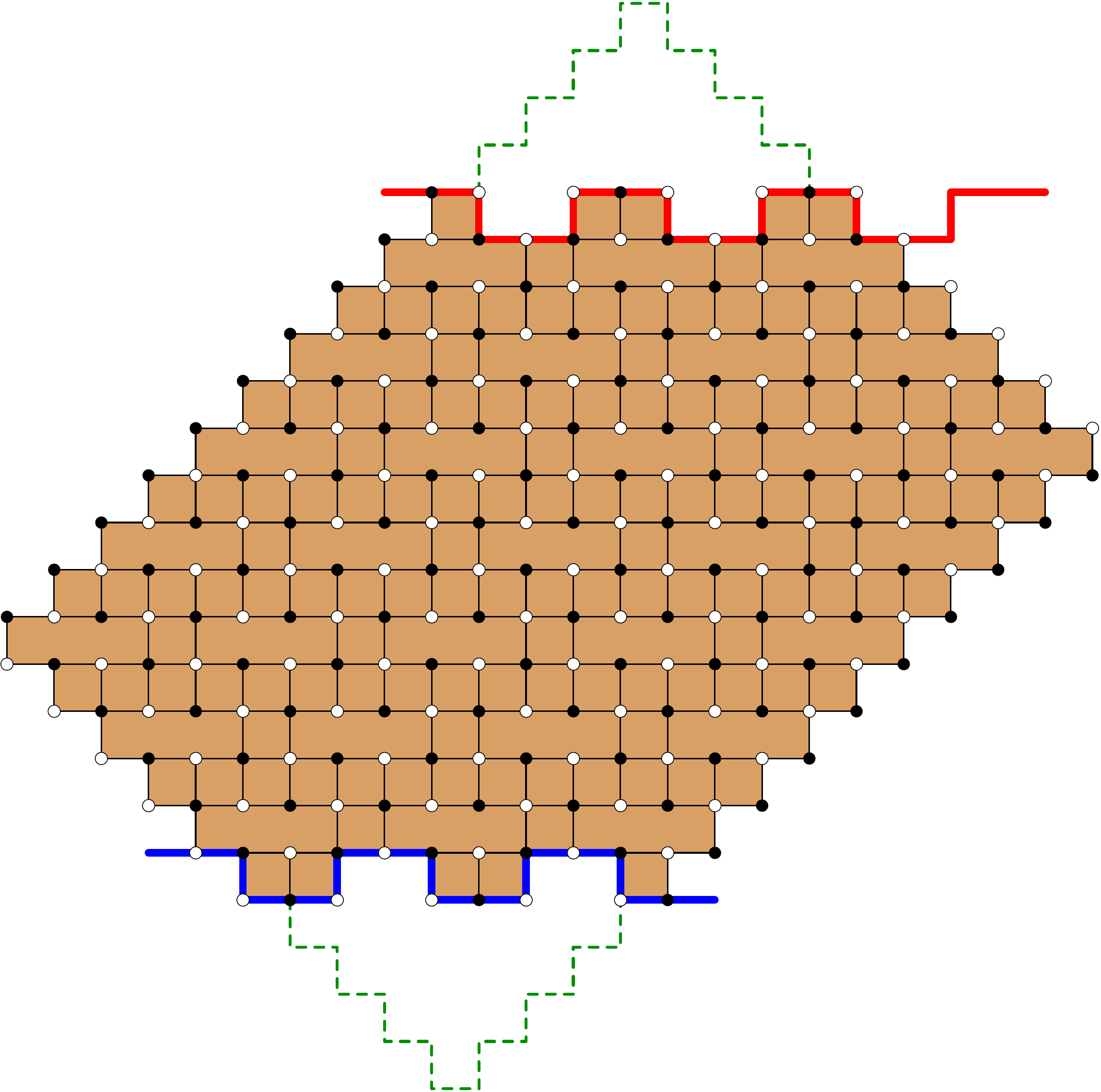}
\caption{A trimmed Aztec rectangle in \cite{trim}.}
\label{fig:TrimmedAztecRectangle}
\end{figure}

\section{Proof of Theorems \ref{thm:formula2}, \ref{thm:formula3}, and \ref{thm:formula4}  }

\label{sec:proofs}

The proofs of these three theorems hinge on the fact that mutations connecting together the quivers $Q_1$, $Q_2$, $Q_3,$ and $Q_4$ also have predictable effects on the brane tilings $\mathcal{T}_1$, $\mathcal{T}_2$, $\mathcal{T}_3$, and $\mathcal{T}_4$.  Namely each such mutation at a toric vertex corresponds to an \emph{urban renewal} transformation or \emph{spider move} applied to the corresponding brane tiling.  See \cite{GK, kuo, speyer}. 
By keeping track of how a contour superimposed ontop of the brane tiling transforms under urban renewal, we are able to obtain combinatorial interpretations that complement our above algebraic formulas.  Our arguments compare with that in Section 4 of \cite{speyer} except for some nuances, concerning vertices traversed by the contours, which we highlight as they arise.
 
 We begin with the proof of Theorem \ref{thm:formula2}, which is a combinatorial formula for cluster variables starting from the Model 2 quiver.  As we saw in the proof of the algebraic formula, Theorem \ref{thm:z2}, we deduce Theorem \ref{thm:formula2} by applying a transformation to the Model 1 quiver and its brane tiling.  Recall that if we begin with the quiver $Q_1$, the initial cluster $\{x_1,x_2,x_3,x_4,x_5,x_6\}$ and the initial brane tiling $\mathcal{T}_1$ as in Figure \ref{fig:quiv_brane} (Right), then mutating at vertex $1$ yields the new quiver $Q_2$ and the new cluster
$\{Y_1,x_2,x_3,x_4,x_5,x_6\}$ with $Y_1 = \frac{x_4x_6+x_3x_5}{x_1}$ (and updating expressions $A$, $B$, $C$, $D$ and $E$ appropriately).  Hence, we obtain the formula for $z_{i,j,k}^{(2)}$ by taking the formula for $z_{i,j,k}^{(1)}$ and substituting in $x_1$ by $Y_1$.  Since vertex $1$ is a toric vertex of $Q_1$, it follows that face $1$ is a quadrilateral of $\mathcal{T}_1$ and that this mutation corresponds to urban renewal on these faces. This transformation pushes the four corners of the quadrilateral inwards, reverses their colors, and produces four new diagonal edges connecting the new vertices to the old ones.  Figure \ref{fig:Model1to2} illustrates the result of urban renewal applied to $\mathcal{T}_1$ simultaneously to all quadrilaterals labeled $1$.

\begin{figure}\centering
\includegraphics[width=4in]{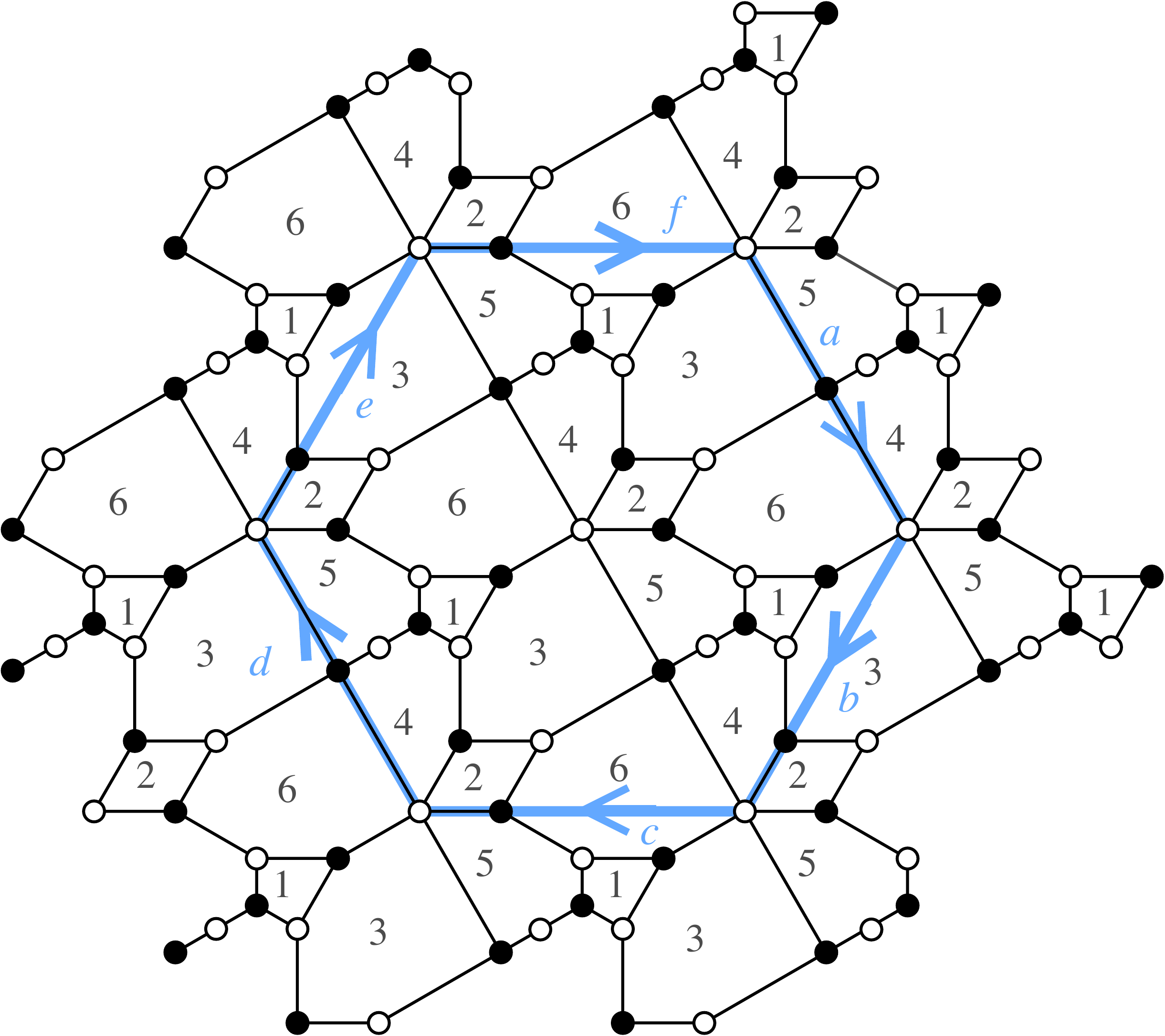}
\caption{Urban Renewal applied to the Model 1 brane tiling $\mathcal{T}_1$ at faces labeled $1$.}
\label{fig:Model1to2}
\end{figure}

Given a toric cluster variable starting from Model 1 whose combinatorial formula is given by a contour without self-intersections, we superimpose the associated contour on the tiling of Figure \ref{fig:Model1to2}.  In this figure, we also draw the unit hexagonal contour for comparison's sake.  We will argue that a subgraph cut out by such a contour has exactly the same weighted partition function of perfect matchings, except that all instances of $x_1$ are replaced with $Y_1$.  We accomplish this by explaining that urban renewal not only affects the tiling itself, but also transforms the set of perfect matchings on subgraphs of such tilings.  In particular, any perfect matching using the two horizontal (or the two vertical) edges of a face labeled $1$ in the Model $1$ tiling $\mathcal{T}_1$ corresponds to a perfect matching using the four new diagonal edges around that specific quadrilateral $1$ in the tiling of Figure \ref{fig:Model1to2}.  Algebraically, this replaces a contribution of $\left(\frac{1}{x_1^2x_4x_6} + \frac{1}{x_1^2x_3x_5}\right)$ to the weighted partition function with a contribution of $\frac{1}{x_3^2x_4^2x_5^2x_6^2}$.  See Figure \ref{fig:urban-matching} (a).

As a second case, any perfect matching that involves exactly one of the four edges of a quadrilateral labeled $1$ in $\mathcal{T}_1$ gets transformed into a perfect matching which uses the two incident diagonal edges and antipodal edge on that quadrilateral in the transformed tiling.  For example, if the original perfect matching uses the top edge of face $1$, i.e. the edge bordering the faces $1$ and $6$, then the contribution to the weighted partition function is transformed from $\frac{1}{x_1x_6}$ to $\frac{1}{Y_1x_3x_4x_5x_6^2}$.  See Figure \ref{fig:urban-matching} (b).

Thirdly, if a perfect matching avoids all four edges of a specific quadrilateral labeled $1$ in $\mathcal{T}$, then this perfect matching corresponds to two perfect matchings after urban renewal.  Namely, the corresponding perfect matchings are the ones that avoid using the four new diagonal edges but either use the two horizontal edges (or the two vertical edges) of that quadrilateral in the tiling of Figure \ref{fig:Model1to2}.  The contribution $\left(\frac{1}{Y_1^2 x_3x_5} + \frac{1}{Y_1^2 x_4x_6}\right)$ is inserted into the relevant terms of the weighted partition function.  See Figure \ref{fig:urban-matching} (c).

\begin{figure}\centering
\includegraphics[width=4in]{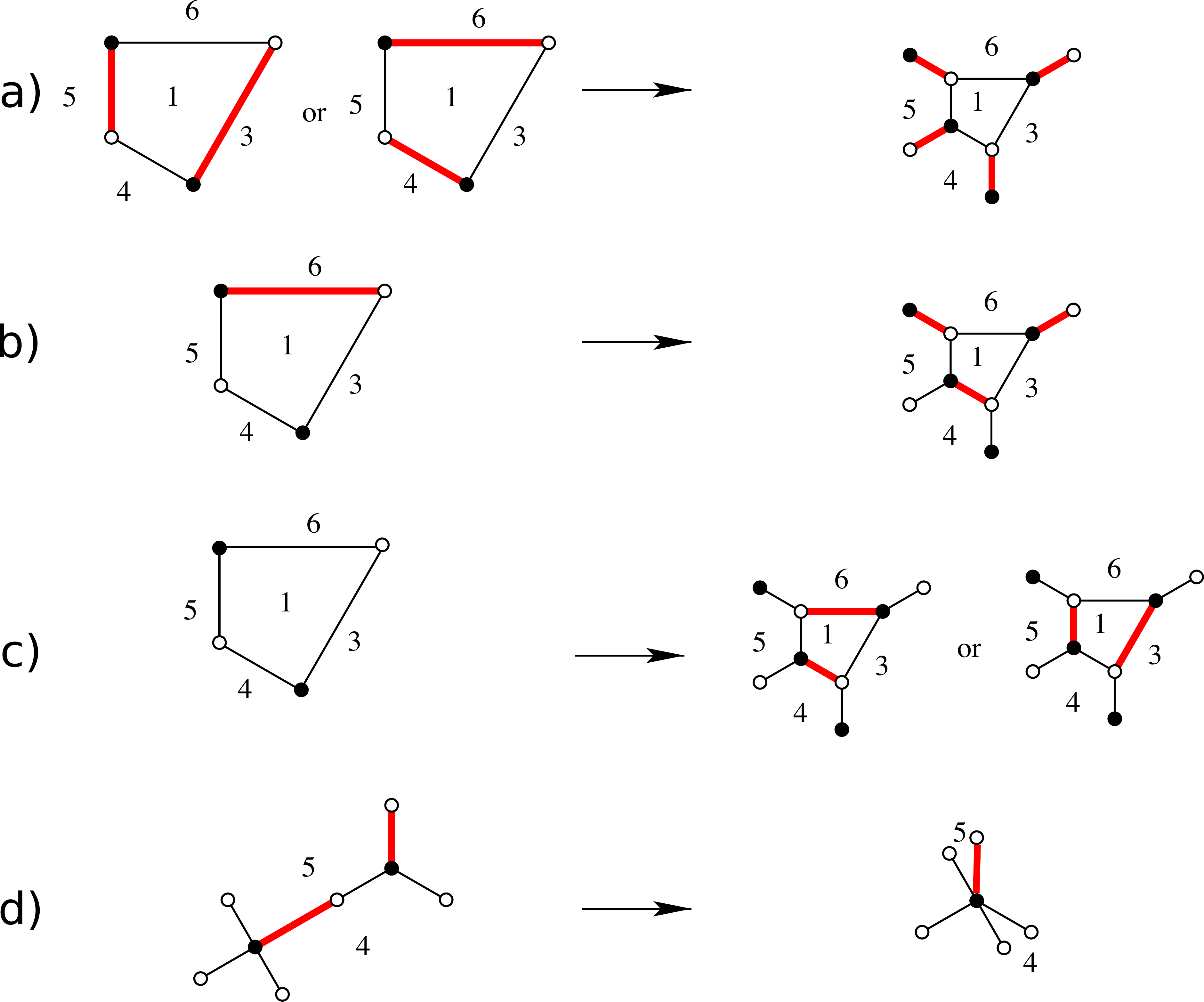}
\caption{Local effects of urban renewal on perfect matchings.  Maps are either $2$-to-$1$, $1$-to-$1$, or $1$-to-$2$ depending on the number of edges in the perfect matching bordering the quadrilateral face $1$.}
\label{fig:urban-matching}
\end{figure}

In all three of these cases, we also must alter the covering monomial associated to subgraphs cut out by such contours.  Firstly, all instances of $x_1$ in the covering monomial are replaced with $Y_1$.  Secondly, since the neighboring faces, those labeled $3$, $4$, $5$ or $6$, are now two sides longer than before (i.e. hexagonal rather than quadrilateral), we must multiply through by $x_3x_4x_5x_6$.  We have the following three identities (after applying $Y_1 = \frac{x_4x_6+x_3x_5}{x_1}$)
$$\left(\frac{1}{x_1^2x_4x_6} + \frac{1}{x_1^2x_3x_5}\right) = \frac{1}{x_3^2x_4^2x_5^2x_6^2}\left(\frac{Y_1x_3x_4x_5x_6}{x_1}\right),$$
$$\frac{1}{x_1x_6} = \frac{1}{Y_1x_3x_4x_5x_6^2}\left(\frac{Y_1x_3x_4x_5x_6}{x_1}\right),$$
$$1 = \left(\frac{1}{Y_1^2 x_3x_5} + \frac{1}{Y_1^2 x_4x_6}\right)\left(\frac{Y_1x_3x_4x_5x_6}{x_1}\right),$$
hence it follows that combining together the above transformations to the weighted partition function of perfect matchings and to the covering monomial correctly yields the combinatorial formula for cluster variables in the Model 2 case.

Lastly, there is a small difference between the tiling of Figure \ref{fig:Model1to2} and the Model 2 tiling $\mathcal{T}_2$ of Figure \ref{fig:Model2-contour}.  Namely, we must delete the $2$-valent white vertices bordering the faces labeled $4$ and $5$, thus collapsing the two adjacent black vertices into one.  The effect is a $1$-to-$1$ map on perfect matchings that deletes an edge incident to this white vertex (which means multiplying by $x_4x_5$) while simultaneously decreasing the number of sides on faces labeled $4$ and $5$ by two (dividing the covering monomial by $x_4x_5$).  See Figure \ref{fig:urban-matching} (d). The net effect leaves the cluster variable formulas unaffected and since such $2$-valent white vertices always lie in the interior of such a contour (as opposed to along a contour), this completes the proof of Theorem \ref{thm:formula2}, using the Model 2 brane tiling $\mathcal{T}_2$ of Figure \ref{fig:Model2-contour} (in place of that of the hybrid tiling of Figure \ref{fig:Model1to2}).

\vspace{1em}

For the case of combinatorial formulas for cluster variables starting from an initial quiver of Model $3$, we again use the method of urban renewal.  We start with the contour drawn on the brane tiling $\mathcal{T}_2$ of Model 2, Figure \ref{fig:Model2-contour}, and apply urban renewal to all of the faces labeled $4$ (corresponding to mutating the cluster variable $x_4$) to get between these two models.  This results in the hybrid tiling of Figure \ref{fig:Model2to3}.

\begin{figure}\centering
\includegraphics[width=4.5in]{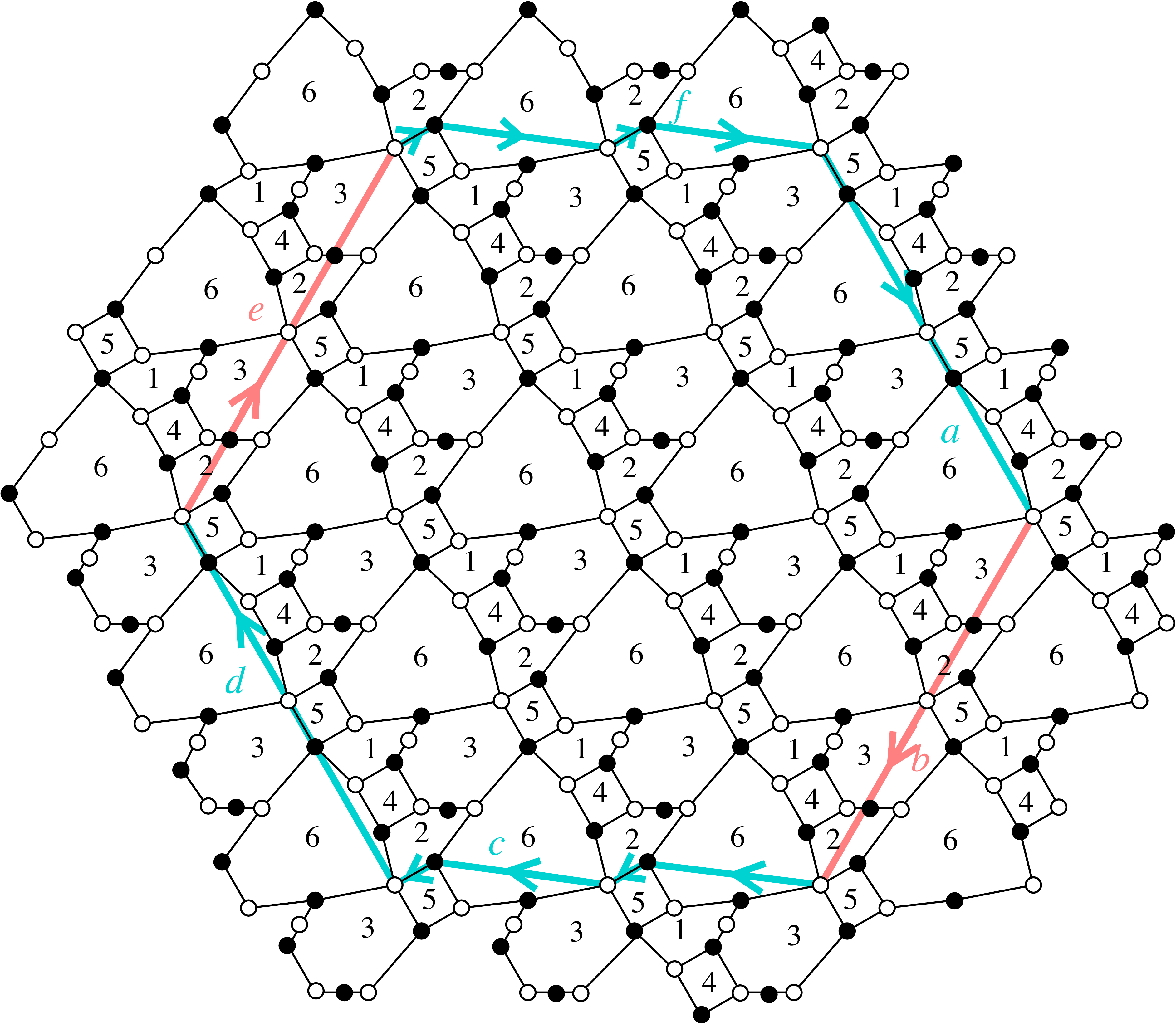}
\caption{Urban Renewal applied to the Model 2 brane tiling $\mathcal{T}_2$ at faces labeled $4$.}
\label{fig:Model2to3}
\end{figure}

Like the above case, the resulting tiling has perfect matchings transformed according to urban renewal and the resulting subgraphs contain $2$-valent vertices (white vertices bordering the faces $1$ and $3$ and black vertices bordering the faces $2$ and $3$) afterwards.  However, unlike the above, some of these black $2$-valent vertices lie directly on the contours cut out by the hexagon of sides $a,b,c,d,e$ and $f$ (namely on sides $b$ and $e$).  In particular, the weighted partition functions do not consistently transform by the appropriate $1$-to-$1$ map unless we included the edges in our subgraph that hang outside of the contour, namely the edges incident to the black $2$-valent vertices and bordering the faces $2$ and $3$.  Consequently, for the Model 3 case we apply the transformation of Figure \ref{fig:urban-matching} (d) collapsing together the neighbors of the $2$-valent vertices, as in the Model 2 case, in two steps.

First, we collapse together the neighbors of the $2$-valent vertices for all white $2$-valent vertices as well as the black $2$-valent vertices that do not lie on the contour.  Secondly, since the remaining black $2$-valent vertices lie on the contour, they are either deleted by the rule of Definition \ref{def:Model3subgraph} (when the contour $b$ or $e$ is positive) or left as $1$-valent vertices (when the contour $b$ or $e$ is negative) once vertices outside the contour are deleted.
In the case that the contour $b$ or $e$ is positive, the resulting boundary of the subgraph matches up with that of Figure \ref{fig:Model3-contour}.  See Figure \ref{fig:Mod2to3-boundaries} (a) and (b).  Similarly, in the case that the contour $b$ or $e$ is negative, then matching up these $1$-valent vertices with their unqiue neighbors leads to a cascade of forced matched edges and a boundary as in Figure \ref{fig:Mod2to3-boundaries} (c).  Comparing the result with the boundary starting instead with the contours in Figure \ref{fig:Model3-contour}, see Figure \ref{fig:Mod2to3-boundaries} (d), we see that these two rules match up.  This completes the proof of Theorem \ref{thm:formula3}. 

\begin{figure}\centering
\includegraphics[width=4.5in]{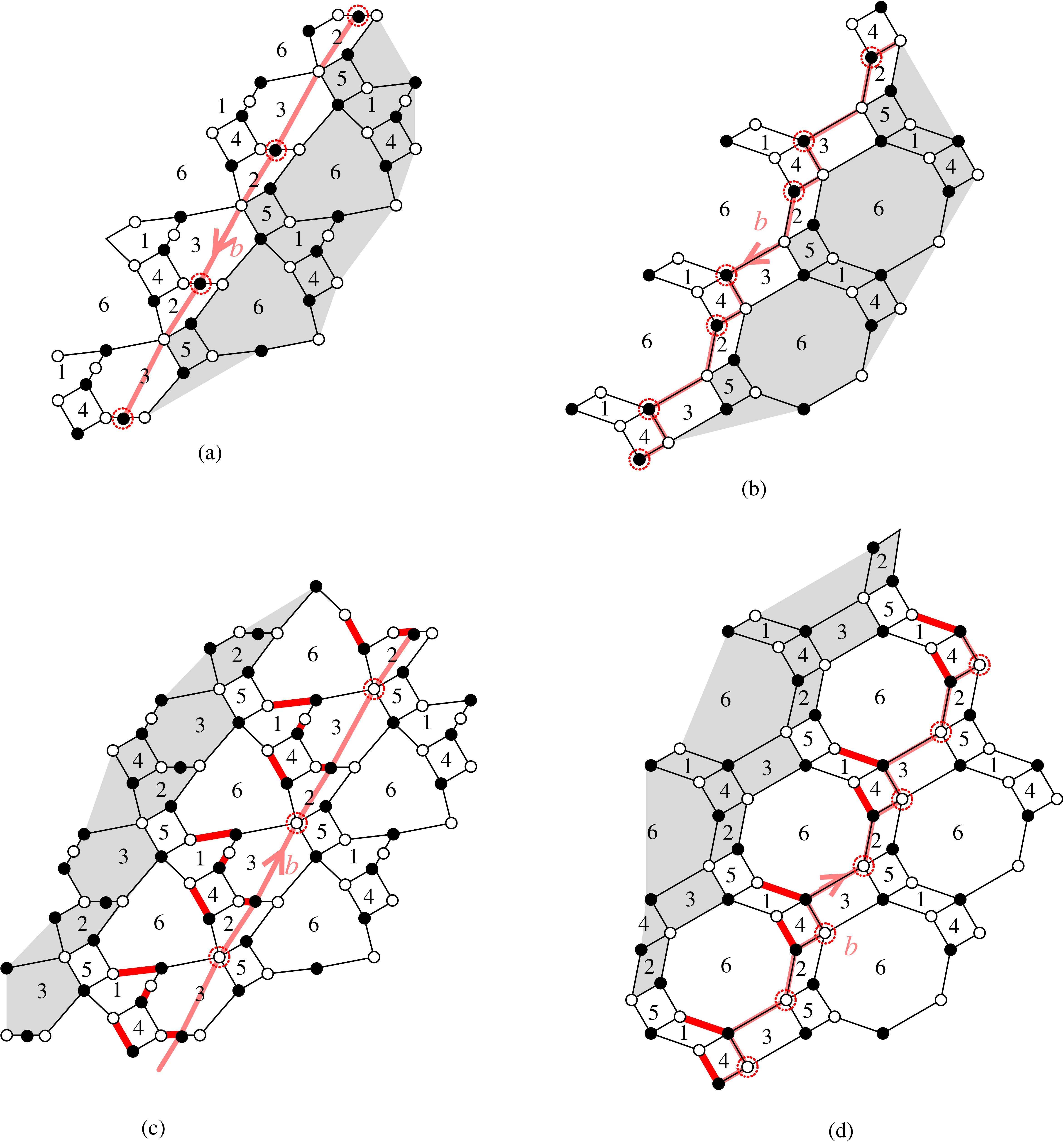}
\caption{Close-ups of the boundaries along side $b$ of the contour.  A partial application of Definition \ref{def:Model3subgraph} to the tiling of Figure \ref{fig:Model2to3} is shown on the left while the full application of Definition \ref{def:Model3subgraph} to the brane tiling $\mathcal{T}_3$ of Figure \ref{fig:Model3-contour} is shown on the right.  The case of positive (resp. negative) $b$ is shown on the top (resp. bottom).  Circled vertices are deleted according to Definition \ref{def:Model3subgraph}.}
\label{fig:Mod2to3-boundaries}
\end{figure}

Finally, the proof of Theorem \ref{thm:formula4} in the Model 4 case uses the same logic as above with the advantage that no $2$-valent vertices appear after urban renewal.  Consequently, we apply urban renewal to all quadrilaterals labeled $3$, which directly results in the subgraphs of $\mathcal{T}_4$ cut out by the contours as in Figure \ref{fig:Model4-contour}.  In particular, the broken symmetry of the contours (i.e. zig-zag sides $b$ and $e$ as opposed to straight sides) in the Model 3 case is therefore inherited by the contours in the Model 4 case.

\section{Model 3 and Trimmed Aztec Rectangles}

\label{sec:Mod3}

In this subsection, we discuss how one can recover the number of perfect matchings of the six families of graphs discussed  in \cite{trim} from our result for Model 3. In order to prove Ciucu's conjecture \cite{Ciucutrim} about the number of perfect matchings of the trimmed Aztec rectangle (see Figure \ref{fig:TrimmedAztecRectangle}), the first author gave explicit formulas for the number of matchings of the six families of graphs, named $A^{(i)}_{a,b,c}$ and $F^{(i)}_{a,b,c}$, for $i=1,2,3$. These formulas are all certain products of perfect powers of $2,5,11$ (see Theorem 2.1 in \cite{trim}). They can be obtained by letting all of the $x_i$'s equal $1$ in our formula for Model 3, i.e. $z_{i,j,k}^{(3)}$, in Theorems \ref{thm:z3} and \ref{thm:formula3}.

In \cite{trim}, the graphs $A^{(i)}_{a,b,c}$ and $F^{(i)}_{a,b,c}$ are defined similarly to our graphs, i.e. based on a six-sided contour, but are parameterized slightly different from that of our subgraphs in Model 3. In particular, the side lengths of the contour of the $A^{(i)}$- and the $F^{(i)}$-graphs are all non-negative and we did not consider the use of signed side-lengths in \cite{trim}. Moreover, if we assign signs to side lengths of the contour in the definition of those graphs in the same ways as our contours in Model 3, then the sum of all six sides of the contour is $0$ (as opposed to being $1$ as it is in our subgraphs of Model 3).

As in \cite{trim}, contours are parameterized by three sides, $a$, $b$, and $c$, which then determine the other three by $d = 2b - a - 2c$, $e = 3b-2a-2c$, $f = |2a-2b+c|$.  However, despite the similar lettering, the $6$-tuples $(a,b,c,d,e,f)$ do not agree with the contours $\mathcal{C}_3(a,b,c,d,e,f)$ of our current work unless we apply certain affine transformations first.  
To convert between our two coordinate systems, in addition to applying a rotation and negation of coordinates, it is necessary to decrease the values of $d$ and $f$ (resp. increase the value of $e$) by one\footnote{This description holds for five out of the six graph families, except for $F^{(3)}$ (as we detail below).  This exception is inherited from the description from \cite{trim}.}.

\begin{figure}
\includegraphics[width=5.8in]{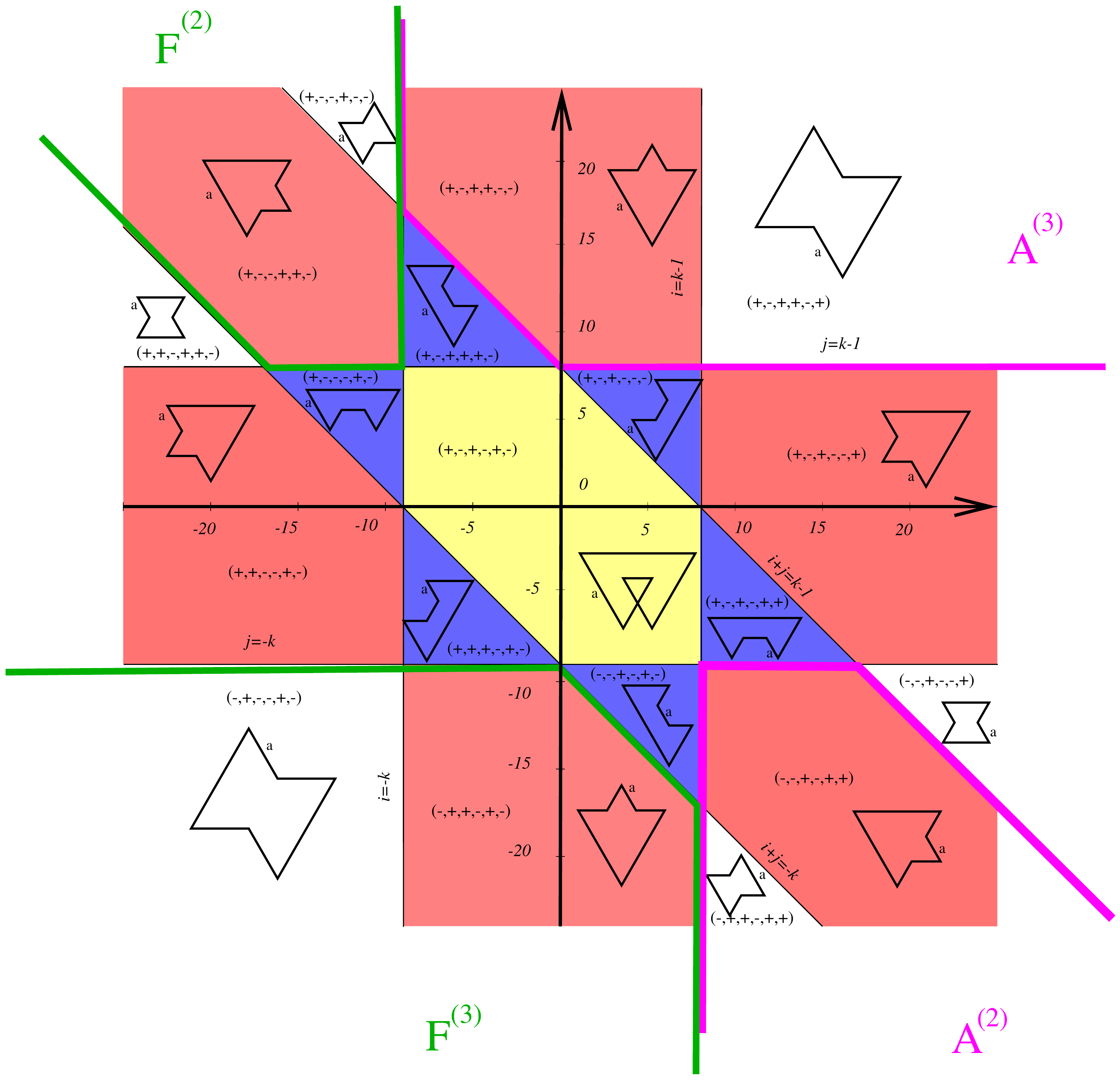}
\caption{Possible sign-patterns for a fixed $k \geq 1$, where the six lines illustrate the $(i,j)$-coordinates so that one of the elements of the $6$-tuple equals zero.  Four of the six families of graphs from \cite{trim} match up with the unbounded regions as illustrated. 
}
\label{fig:decomposition}
\end{figure}

\begin{figure}
\includegraphics[width=5.8in]{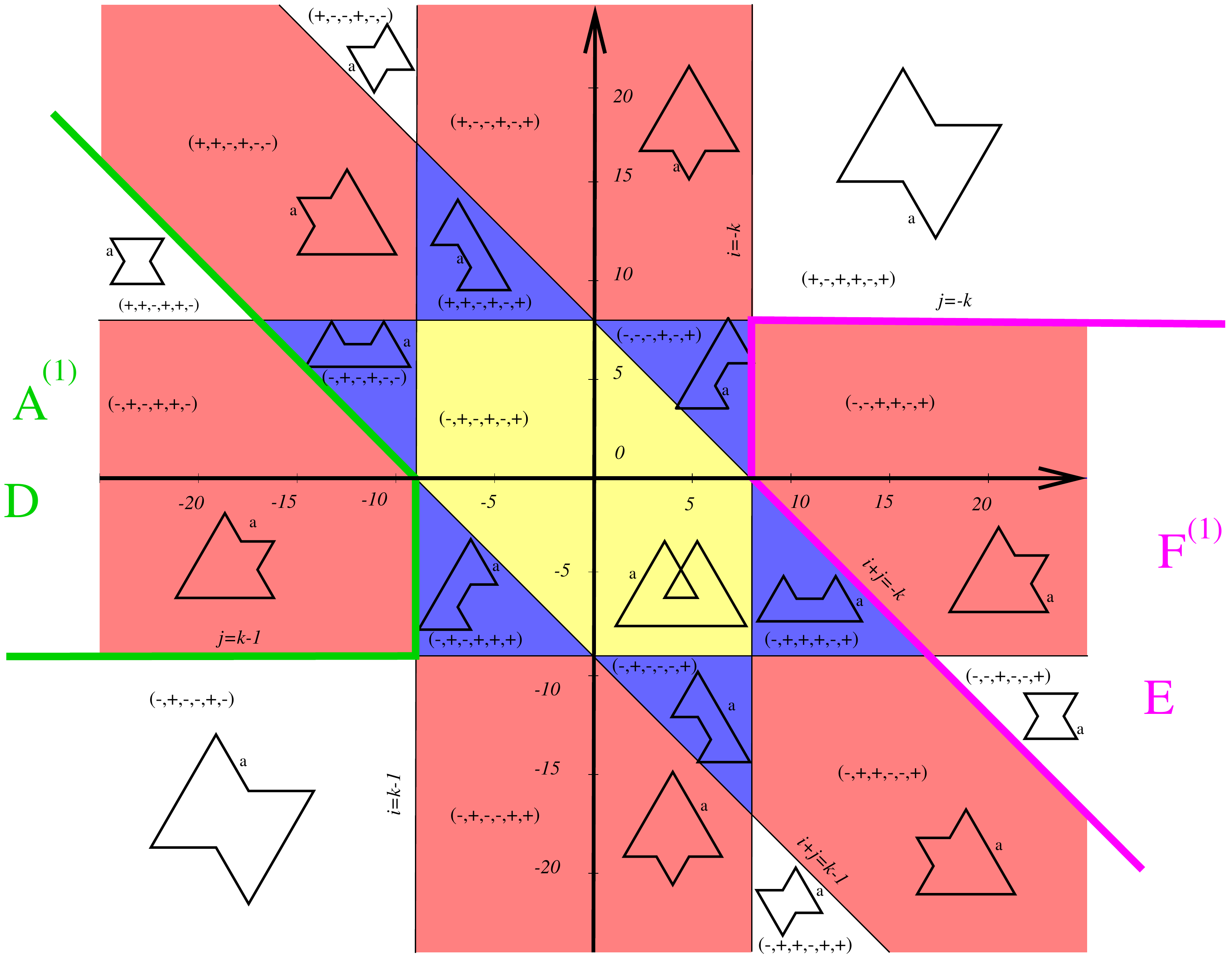}
\caption{Possible sign-patterns for a fixed $k \leq 0$.  Observe that the uncolored unbounded regions correspond to the same sign patterns as they did in the $k \geq 1$ case. The $A^{(1)}$ and $F^{(1)}$ families from \cite{trim} ($D$- and $E$- families from \cite{lai', lai}) match up with the Model 3 (resp. Model 4) contours as illustrated. 
}
\label{fig:decomposition2}
\end{figure}

Nonetheless, as shown in the examples below, the subgraphs of Model 3 coincide with the six families graphs in \cite{trim}, up to a simple deformation.  We illustrate the translation between our two coordinate systems in Figures \ref{fig:decomposition} and \ref{fig:decomposition2} (compare with Figures 20 and 21 of \cite{LaiMus}).  It is worth noting that in \cite{trim, lai', lai}, the three families of $A^{(i)}_{a,b,c}$ graphs are $60$ degree rotations of each other, but in our current work, we exhibit families that are $120$ degree rotations of each other instead.  Additionally, the contour shapes in the purple (triangular bounded) areas did not appear in the first author's previous work \cite{trim, lai', lai}, nor did the self-intersecting contour shapes of the yellow (hexagonal central) areas.

Consider the contour $ \mathcal{C}_3(a,b,c,d,e,f)$ with the sign pattern $(+,-,+,+,-,+)$. We rotate the corresponding core graph $\mathcal{G}_3(a,b,c,d,e,f)$, for example see the shaded region in Figure \ref{fig:Model3egs} (Left), $60^{\circ}$ counter-clockwise and `straighten out' all the  hexagonal faces to get a subgraph of the square grid. The hexagonal faces become $1\times 3$ rectangles, and all of the other faces become $1\times 1$ squares. This way, the core graph becomes the graph $A^{(3)}_{a,-b,c}$ (in the case $a\leq c+d$ in \cite{trim}). More precisely, the six-sided contour of the $A^{(3)}$-graph is  $(a,-b,c,d-1,-e-1,f-1)$.  For example, see Figure \ref{fig:Model3toTrim1} (Left). Similarly, the core graph corresponding to the sign pattern $(+,-,+,+,-,-)$ is isomorphic to the graph $A^{(3)}_{a,-b,c}$
 in the case $a>c+d$ in \cite{trim}.

By  the same arguments, the core graph $\mathcal{G}_3(a,b,c,d,e,f)$  with the sign pattern $(+,+,-,+,+,-)$ gives the graph $ A^{(1)}_{b,-c,d-1}$ in the case $a\leq c+d$ in \cite{trim}.  See Figure \ref{fig:Model3egs} (Middle) and Figure \ref{fig:Model3toTrim1} (Right). Similarly, the pattern $(-,+,-,+,+,-) $ yields the the $A^{(1)}$-type graph in case $a>c+d$.   See Figure \ref{fig:Model3egs} (Middle).

\begin{figure}
\includegraphics[width=2in]{Sample1Model3.pdf} \includegraphics[width=2in]{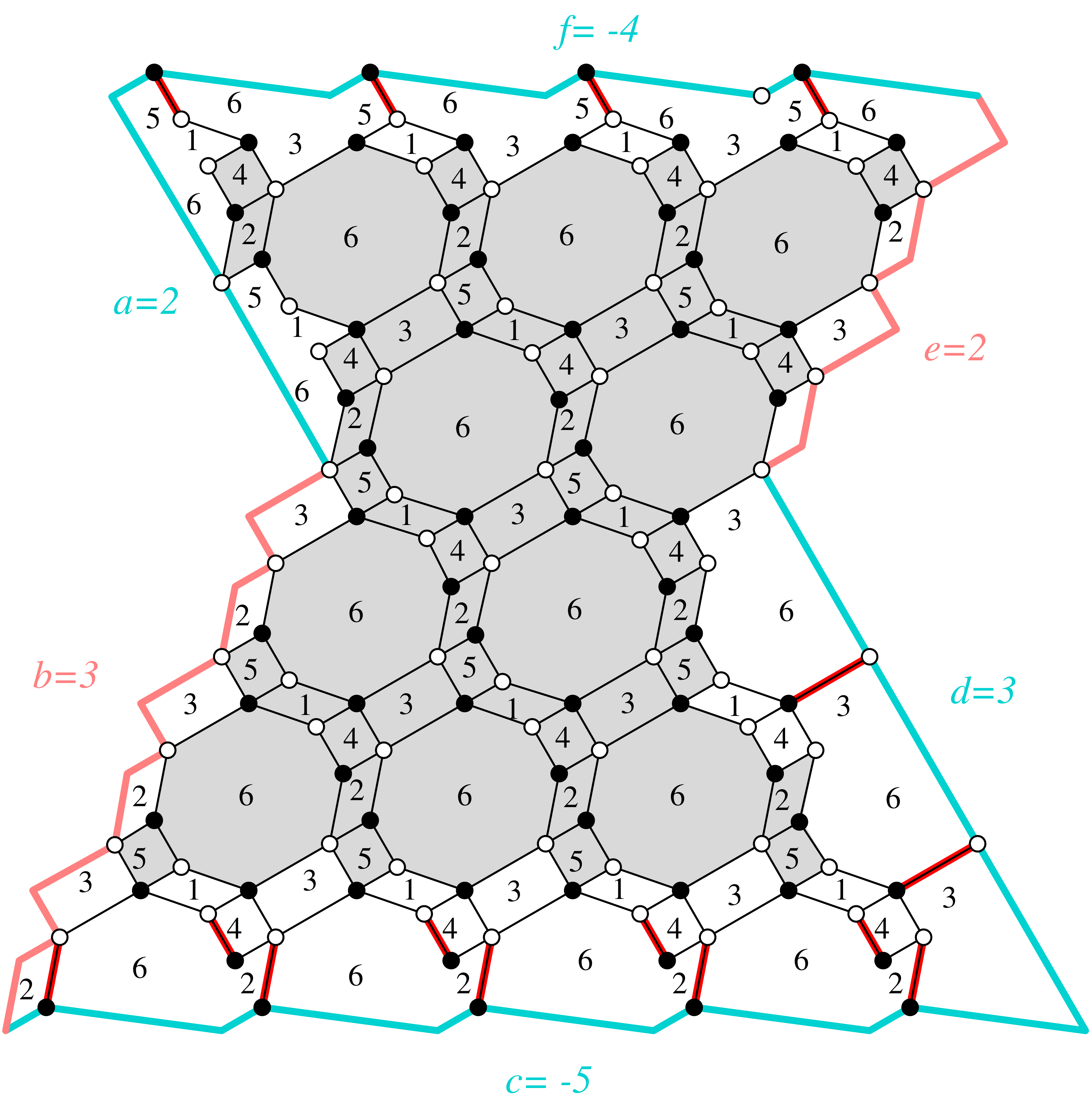}
\includegraphics[width=5in]{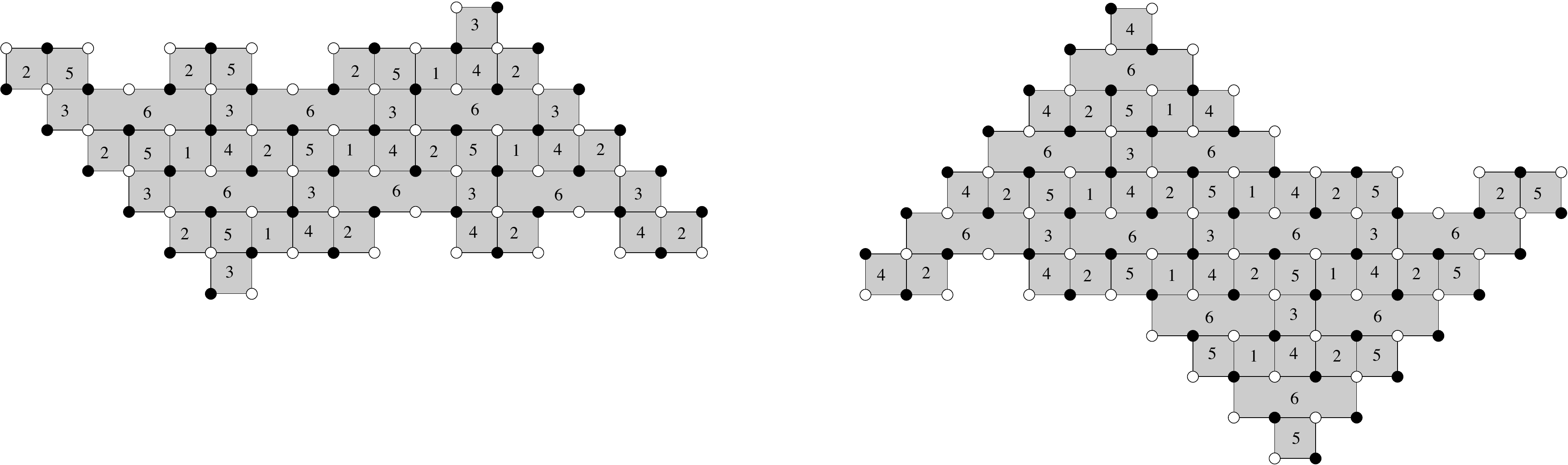}
\caption{Graphs $A^{(3)}_{3,4,1}$ (Left) and $A^{(1)}_{3,5,2}$ (Right) in \cite{trim} correspond to the shaded core graphs for Model 3 with the contours $(3,-4,1,4,-5,2)$ and $(2,3,-5,3,2,-4)$, respectively.}
\label{fig:Model3toTrim1}
\end{figure}

\begin{figure}
\includegraphics[width=5in]{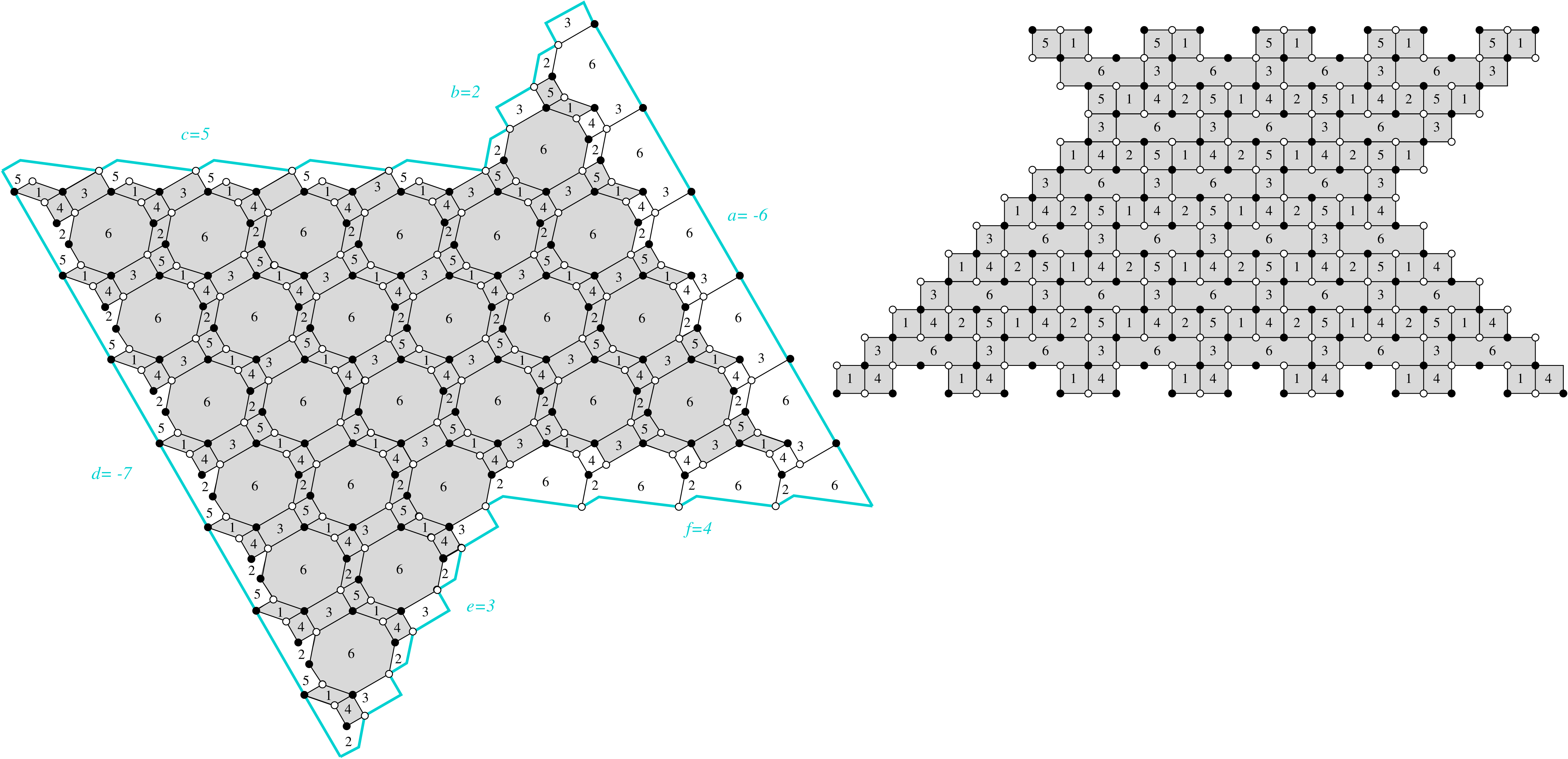}
\caption{Graphs $A^{(2)}_{5,8,4}$ (Right) in \cite{trim} corresponds to the  shaded core graph of the graph for Model 3 with the contour $(-6,2,5,-7,3,4)$ (Left).}
\label{fig:Model3toTrim2}
\end{figure}

Figure \ref{fig:Model3toTrim2} shows that one can get the graph $ A^{(2)}_{c,-d+1,e+1}$ (in the case $a\leq c+d$ in \cite{trim}) from the core subgraph $\mathcal{G}_3(a,b,c,d,e,f)$ corresponding to the contour  $ \mathcal{C}_3(a,b,c,d,e,f)$ with the sign pattern $(-,+,+,-,+,+)$.  The case $a>c+d$ in the $A^{(2)}$-type graph corresponds to the pattern $(-,-,+,-,+,+)$ in our Model 3.

\begin{figure}
\includegraphics[width=5in]{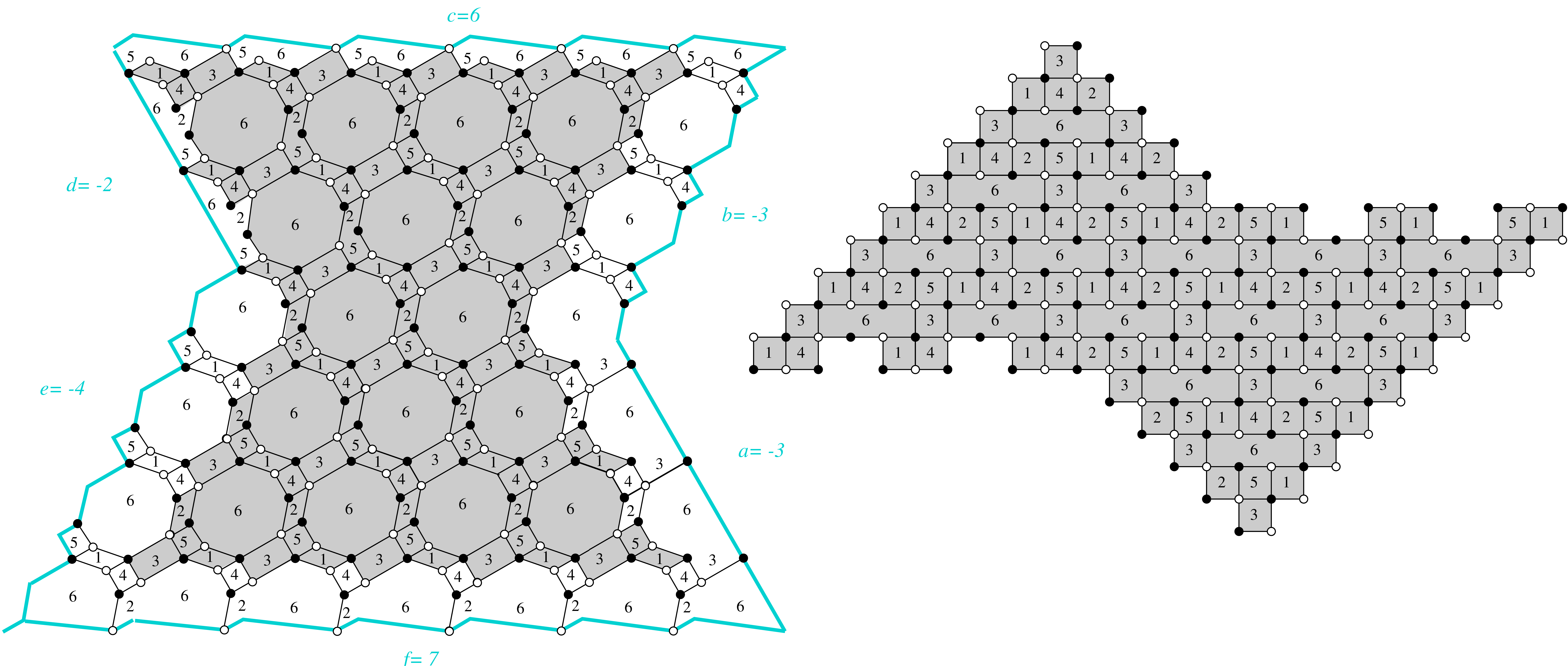}
\caption{Graphs $F^{(1)}_{3,6,3}$ (Right) in \cite{trim} corresponds to the  shaded core graph of the graph for Model 3 with the contour $(-3,-3,6,-2,-4,7)$ (Left).}
\label{fig:Model3toTrim3}
\end{figure}

For $i=1,2,$ or $3$, the graph $F^{(i)}(a,b,c)$ is obtained from $A^{(i)}(a,b,c)$, in our coordinate system, by reversing the signs of each entry and rotating by $180$ degrees.  Figure \ref{fig:Model3toTrim3} illustrates how the graph $F^{(1)}_{-e-1,f-1,-a}$ (in the case $a\leq c+d$ in \cite{trim}) can be obtained from $\mathcal{G}_3(a,b,c,d,e,f)$ with the sign pattern $(-,-,+,-,-,+)$.  The case $a>c+d$ in the $F^{(1)}$-type graph corresponds to the pattern $(-,-,+,+,-,+)$ in our Model 3.

\begin{figure}
\includegraphics[width=5in]{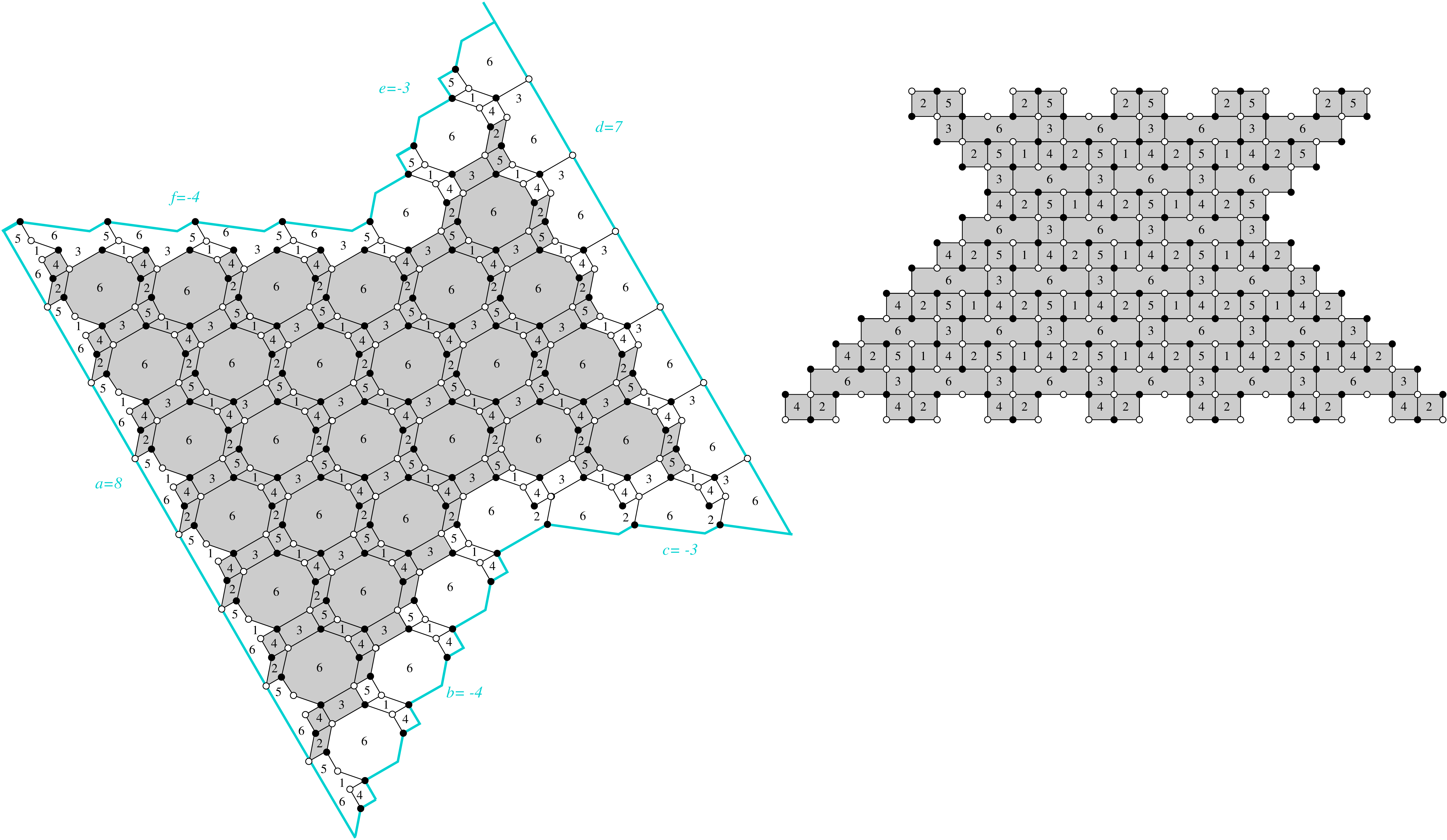}
\caption{Graphs $F^{(2)}_{5,8,4}$ (Right) in \cite{trim} corresponds to the  shaded core graph of the graph for Model 3 with the contour $(8,-4,-3,7,-3,-4)$ (Left).}
\label{fig:Model3toTrim4}
\end{figure}

 By the same arguments, one can get the graph $F^{(2)}_{-f+1,a,-b}$ (the case $a\leq c+d$) from the subgraph with the contour $\mathcal{C}_3(a,b,c,d,e,f)$ in the sign pattern $(+,-,-,+,-,-)$ (see Figure \ref{fig:Model3toTrim4} for an example). In the case $a> c+d$, the graph $F^{(2)}_{-f+1,a,-b}$ is obtained from the subgraph for Model 3 with the sign pattern $(+,-,-,+,+,-)$.

 \begin{figure}
\includegraphics[width=5in]{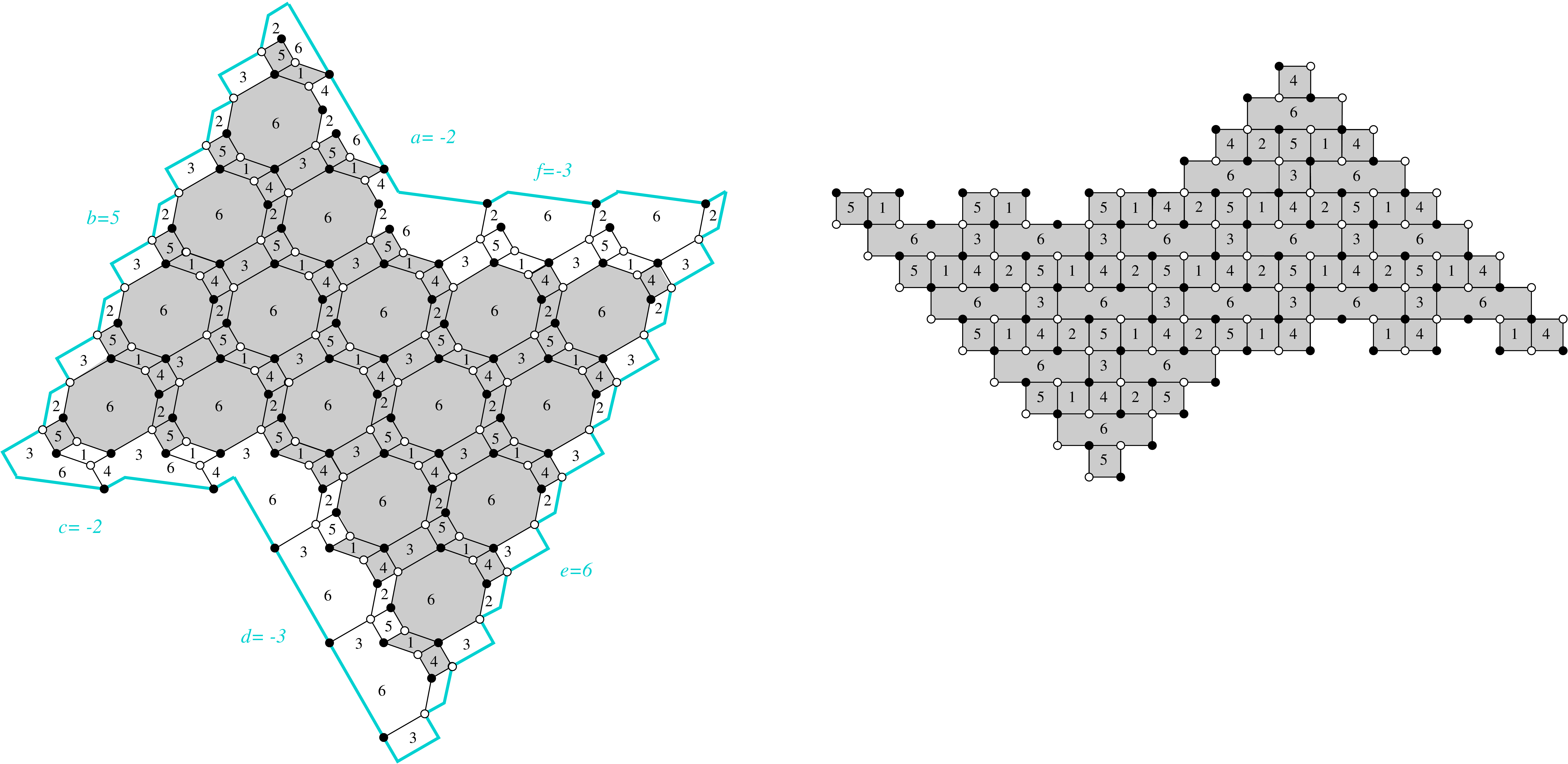}
\caption{Graphs $F^{(3)}_{3,6,3}$ (Right) in \cite{trim} corresponds to the  shaded core graph of the graph for Model 3 with the contour $(-2,5,-2,-3,6,-3)$ (Left).}
\label{fig:Model3toTrim5}
\end{figure}

Finally, the  Model 3 subgraphs $\mathcal{G}_3(a,b,c,d,e,f)$ corresponding to the contours  $\mathcal{C}_3(a,b,c,d,e,f)$ with the sign patterns      $(-,+,-,-,+,-)$ or $(-,+,+,-,+,-)$ are isomorphic to graphs of the form $ F^{(3)}_{-d,e,-f}$.  See Figure \ref{fig:Model3toTrim5} for an example.

\section{Model 4 and Blum's Conjecture}

\label{sec:Mod4}

Similar to the case of Model 3, our result for Model 4 also implies several known results for the enumeration of tilings. In particular, we show below that the subgraphs corresponding to certain shapes resulting from the Model 4 contours are isomorphic to the dual graphs of the $D$- and $E$-regions appearing in Ciucu and the first author's work on Blum's Conjecture \cite{lai', lai}. Thus, our formula $z_{i,j,k}^{(4)}$ in Theorems \ref{thm:z4} and \ref{thm:formula4} implies the tiling formulas for these regions (see Theorem 3.1 in \cite{lai'}) by simply letting all of the $x_i$'s equal $1$.  To obtain the corresponding weighted tiling formulas for these regions (see Theorem 2.1 in \cite{lai}), we use the evaluations $x_1=x_2=x_3=\sqrt{\frac{z}{xy}}$, $x_4=x_5=\sqrt{\frac{y}{zx}}$, and $x_6=\sqrt{\frac{x}{yz}}$.

Figure \ref{fig:decomposition2} illustrates the translation between the contour coordinate system of our current work and that of \cite{lai',lai}. (Compare with Figure 21 of \cite{LaiMus}.) For more details on how our Model 4 subgraphs $\mathcal{G}_4(a,b,c,d,ef)$ compare to the $D$- and $E$- regions of \cite{lai', lai}, we break our set of subgraphs into different cases.  First, we consider the subgraphs $\mathcal{G}_4(a,b,c,d,e,f)$ associated to the contours  $\mathcal{C}_4(a,b,c,d,e,f)$ with the sign pattern $(-,+,-,+,+,-)$.  See Figure \ref{fig:Model4egs} (Middle). The core graph $\mathcal{G}_4(a,b,c,d,ef)$ in this case, indicated by the shaded graph, is exactly the dual graph of the region $D_{b,-c,d-1}$ in \cite{lai',lai} (in the case $a>c+d$).  The six-sided contour of the $D$-region in this case is $(b,-c,d-1,e+1,-f+1,a)$.   Similarly, the dual graph of the $D$-region in the case $a\leq c+d$ corresponds to the sign pattern $(+,+,-,+,+,-)$ of the contour $\mathcal{C}_4(a,b,c,d,e,f)$ in Model 4.

\begin{figure}
\includegraphics[width=3in]{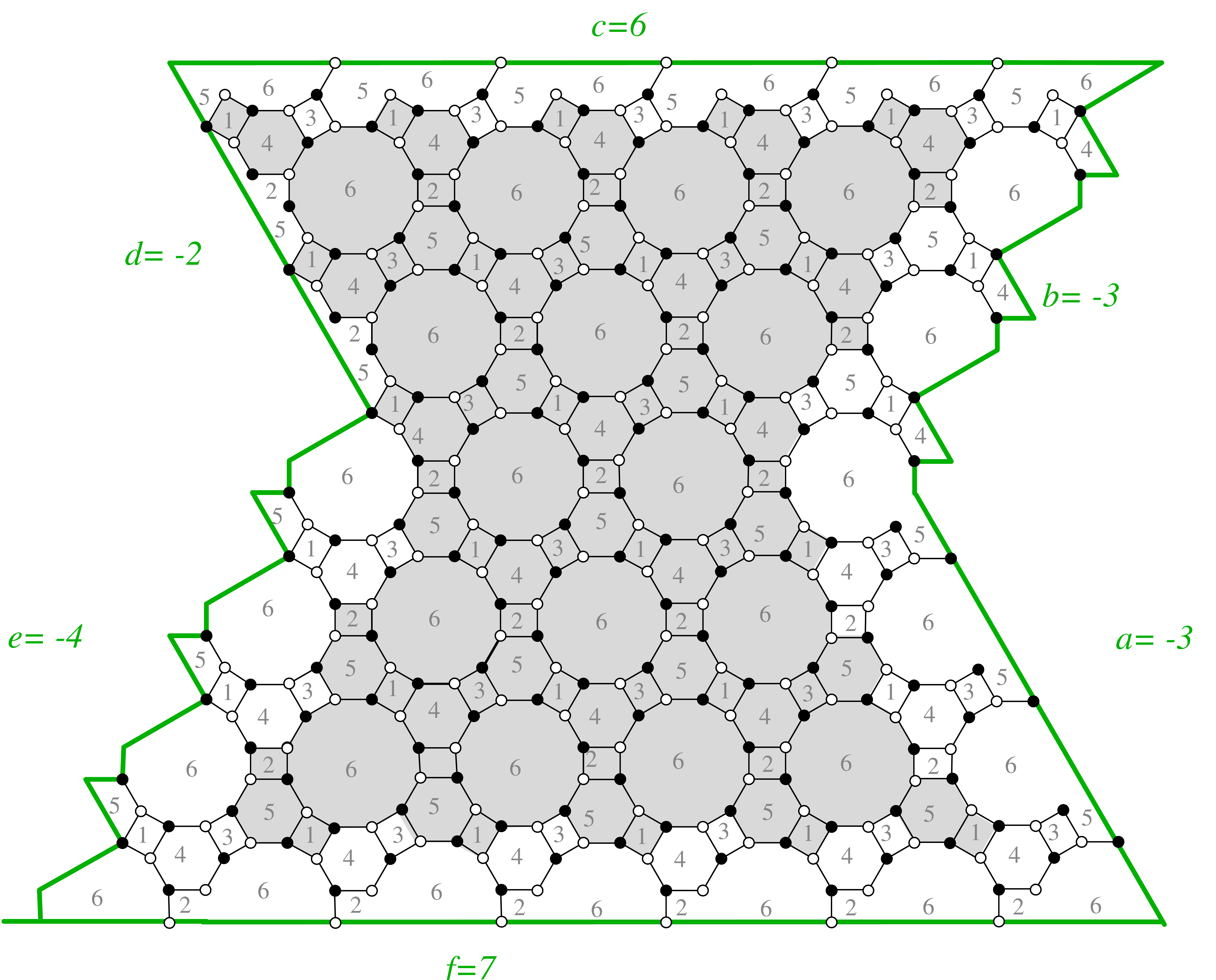}
\caption{The dual graph of the region $E_{5,8,3}$ in \cite{lai', lai} corresponds to the  shaded core graph of the graph for Model 4 with the contour $(-3,-3,5,-2,-4,6)$.}
\label{fig:Model4toDungeon}
\end{figure}

Next, Figure \ref{fig:Model4toDungeon} shows that the subgraphs $\mathcal{G}_4(a,b,c,d,e,f)$ associated to the contour $\mathcal{C}_4(a,b,c,d,e,f)$ with the sign pattern
 $(-,-,+,-,-,+)$ is exactly the dual graph of the region $ E_{-e-1,f-1,-a}$  in \cite{lai',lai} (for the case $a\leq c+d$).   Finally, when $a>c+d$, the dual graph of the $E$-region corresponds to the pattern $(-,-,+,+,-,+)$.

\section{Model 4 and the Hexahedron Recurrence}

\label{sec:super}

Our algebraic formula for Model $4$, i.e. Theorem \ref{thm:z4}, also sheds further light on the Laurent polynomials $A(a,b,c)$ appearing in Rick Kenyon and Robin Pemantle's work \cite{KP} on the hexahedron recurrence.  For each $(a,b,c) \in (\mathbb{Z}/2)^3$ that has an integral sum, the coordinate $A(a,b,c)$ represents a solution to the hexahedron recurrence, and can be expressed as a Laurent polynomial\footnote{This is because of its cluster algebraic interpretation or by applying the Laurent Phenomenon \cite{laurent}.} in the infinite set of initial variables 
$$\mathcal{I}_2 = \{A(i,j,k) : 0 \leq i + j + k \leq 2\} ~~\cup~~ 
\{A(i+\frac{1}{2},j+\frac{1}{2}, k): i+j+k=0\}$$ $$~~\cup~~  \{A(i+\frac{1}{2},j, k+\frac{1}{2}): i+j+k=0\} ~~\cup~~  \{A(i,j+\frac{1}{2}, k+\frac{1}{2}): i+j+k=0\}.$$
The set $\mathcal{I}_2$ corresponds to the faces of $4-6-12$ graph $\Gamma(Z_2)$, as in \cite[Sec. 3]{KP}, as illustrated in Figure \ref{fig:KP4612}.

\begin{figure}
\includegraphics[width=5.5in]{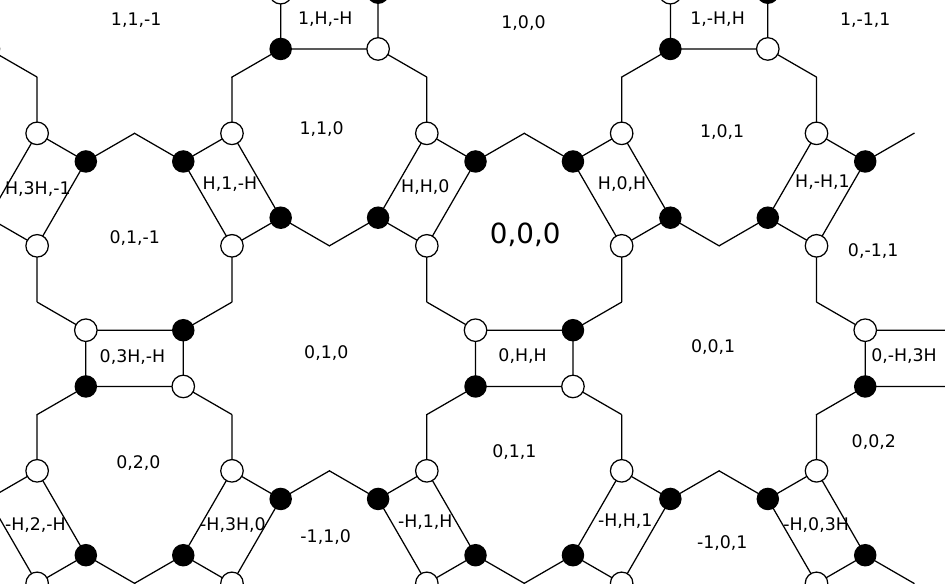}
\caption{The $4-6-12$ graph of \cite{KP} with face labels.  We use $H$ instead of $\frac{1}{2}$ for a better visualization. (Based on \cite[Fig. 7]{KP}.)}
\label{fig:KP4612}
\end{figure}

Up to a relabeling of the initial variables by translation, the Laurent polynomials are unchanged if we replace $(a,b,c)$ with $(a+x,b+y,c+z)$ where $x,y,z\in \mathbb{Z}$ and $x+y+z=0$.  Equivalently we are translating in $(\mathbb{Z}/2)^3$ by integral linear combinations of $(1,-1,0)$ or $(1,0,-1)$, as well as $(0,1,-1)$ which is in the span of the first two.  Hence, without loss of generality, we may focus on the formulas for $A(0,0,n+2)$, 
$A(0, \frac{1}{2},n+\frac{1}{2})$,
$A(\frac{1}{2}, 0,n+\frac{1}{2})$, and
$A(\frac{1}{2}, \frac{1}{2},n)$ for $n \in \mathbb{Z}$ (e.g. we have let $n+2 = a+b+c$ for the case of $A(a,b,c)$ and $a,b,c\in \mathbb{Z}$).  

\begin{remark} Note that computing $A(0,0,n+2)$ in terms of $\Gamma(Z_2)$ is equivalent to computing $A(0,0,0)$ in terms of $\Gamma(Z_{-n})$, with faces labeled by 
$$\mathcal{I}_{-n} = \{A(i,j,k) : -n-2 \leq i + j + k \leq -n\} ~~\cup~~ 
\{A(i+\frac{1}{2},j+\frac{1}{2}, k): i+j+k=-n-2\}$$ $$~~\cup~~  \{A(i+\frac{1}{2},j, k+\frac{1}{2}): i+j+k=-n-2\} ~~\cup~~  \{A(i,j+\frac{1}{2}, k+\frac{1}{2}): i+j+k=-n-2\},$$ which is the case discussed in \cite{KP}.
\end{remark}

We define a projection $\varphi$ from $\mathcal{I}_2$ to $\{x_1,x_2,\dots, x_6\}$ that maps this infinite set of initial variables to a finite set in a doubly-periodic way. 
$$\varphi\left(A(a,b,c)\right) = 
\begin{cases} 
x_1 \mathrm{~if~} (a,b,c) = (i+\frac{1}{2},j+\frac{1}{2},k) \mathrm{~with~}i,j,k \in \mathbb{Z}, \mathrm{~and~} i+j+k  = 0 \\ 
x_2 \mathrm{~if~} (a,b,c) = (i, j+\frac{1}{2},k+\frac{1}{2}) \mathrm{~with~}i,j,k \in \mathbb{Z}, \mathrm{~and~} i+j+k  = 0 \\ 
x_3 \mathrm{~if~} (a,b,c) = (i+\frac{1}{2},j,k+\frac{1}{2}) \mathrm{~with~}i,j,k \in \mathbb{Z}, \mathrm{~and~} i+j+k  = 0 \\ 
x_4 \mathrm{~if~} a,b,c \in \mathbb{Z}, \mathrm{~and~} a+b+c  = 0 \\
x_5 \mathrm{~if~} a,b,c \in \mathbb{Z}, \mathrm{~and~} a+b+c  = 2 \\
x_6 \mathrm{~if~} a,b,c \in \mathbb{Z}, \mathrm{~and~} a+b+c  = 1
\end{cases},$$
Under this projection, the $4-6-12$ graph with the face labelings of Figure \ref{fig:KP4612} becomes the brane tiling $\mathcal{T}_4$ with the labeling as in Figure \ref{fig:Model4-contour}.  For example, we have the equalities $\varphi\left(A(1/2,1/2,0)\right) = z_{-1,1,0}^{(4)} = x_1, ~~~~\varphi\left(A(0,1/2,1/2)\right) = z_{0,-1,0}^{(4)} = x_2, ~~~\varphi\left(A(1/2,0,1/2)\right) = z_{1,0,0}^{(4)} =  x_3, ~~~~\varphi\left(A(0,0,0)\right) = z_{0,0,-1}^{(4)} = x_4, ~~~~\varphi\left(A(0,0,2)\right) = z_{0,0,1}^{(4)} = x_5, \mathrm{~~and~~} \varphi\left(A(0,0,1)\right) = z_{0,0,0}^{(4)} = x_6.$ Hence, we recover the initial variables corresponding to the deformed prism $\Delta_4$.

\begin{proposition} \label{prop:AZ}
Comparing the coordinates of \cite{KP} to Theorem \ref{thm:z4}, and breaking into cases based on the parity of $n$, we have the identities
$$\varphi\left(A(0,0,2N+1)\right)  = z_{0,0,2N}^{(4)} = x_6 D^{N(N-1)} E^{N^2}$$ 
$$\varphi\left(A(0,0,2N+2)\right)  = z_{0,0,2N+1}^{(4)} = x_5 D^{N^2} E^{N(N+1)}$$ 
$$\varphi\left(A(1/2,1/2,2N)\right)  = z_{-1,1,2N}^{(4)} = x_1 D^{N(N-1)} E^{N^2}$$ 
$$\varphi\left(A(1/2,1/2,2N+1)\right)  = z_{1,-1,2N+1}^{(4)} = \frac{\widetilde{Y}_4''}{x_1x_4} D^{N^2} E^{N(N+1)}$$ 
$$\varphi\left(A(0,1/2,2N+1/2)\right)  = z_{0,-1,2N}^{(4)} = x_2 D^{N(N-1)} E^{N^2}$$ 
$$\varphi\left(A(0,1/2,2N+3/2)\right)  = z_{0,1,2N+1}^{(4)} = \frac{\widetilde{Y}_4''}{x_2x_4} D^{N^2} E^{N(N+1)}$$ 
$$\varphi\left(A(1/2, 0,2N+1/2)\right)  = z_{1,0,2N}^{(4)} = x_3 D^{N(N-1)} E^{N^2}$$ 
$$\varphi\left(A(1/2, 0,2N+3/2)\right)  = z_{-1,0,2N+1}^{(4)} = \frac{\widetilde{Y}_4''}{x_3x_4} D^{N^2} E^{N(N+1)}$$ 
where $D = \frac{x_4}{x_5}$, $E =
 \frac{x_6^6 + 2x_1x_2x_3x_6^3 + x_1^2x_2^2x_3^2 + 3x_4x_5x_6^4 + 3x_1x_2x_3x_4x_5x_6 + 3x_4^2x_5^2x_6^2 + x_4^3x_5^3}{x_1x_2x_3x_4^2x_6}$, and \\
 $\widetilde{Y}_4'' = x_1x_2x_3 + x_4x_5x_6 + x_6^3$.
\end{proposition}

\begin{proof}
To get the desired identities, we apply superurban renewal, as illustrated in \cite[Fig. 6]{KP}.  Under the projection $\varphi$, we mimic this sequence of six urban renewals by the cluster mutation sequence $\mu_1, \mu_4, \mu_2, \mu_1, \mu_3, \mu_2$ in that order, starting from $\Delta_4$.  This leads to a sequence of quivers of Models 3, 2, 1, 2, 3, and back to Model 4.  Translating the initial configuration of $\Delta_4$ by an arbitrary amount in $\mathbb{Z}^3$ yields a local configuration of the form $$\widetilde{\Delta_4} = [(i-1,j+1,k), (i,j-1,k), (i+1,j,k), (i,j,k-1), (i,j,k+1),(i,j,k)].$$
See Figure \ref{fig:IJ} (a). The mutation sequence induces an action on clusters that is visualized in Figures \ref{fig:IJ} (b)-(g) with Figure \ref{fig:IJ} (g) showing the cluster associated to the final Model 4 quiver.  The final quiver of Model 4 is isomorphic to $Q_4$ up to the permutation $(34)(56)$ applied to the vertices and the reversal of all arrows.  

Furthermore, because of the symmetry of this quiver, see Figure \ref{fig:dP3QuiverModels} (Right), we can also apply the permutation $(123)(45)$, which leaves the quiver unchanged.  Noting that $(123)(45)\cdot (34)(56) = (123564)$, we see that
the combined transformation $\mu_1\mu_4\mu_2\mu_1\mu_3\mu_2 \circ (123564)$ acts as an involution that reverses all arrows of $Q_4$. On the level of clusters, applying this transformation twice in a row yields a cluster whose $\mathbb{Z}^3$-parameterization agrees with the original cluster except that the third coordinate of all six entries has been increased by $2$. 

\begin{figure}
(a) \includegraphics[width=1.8in]{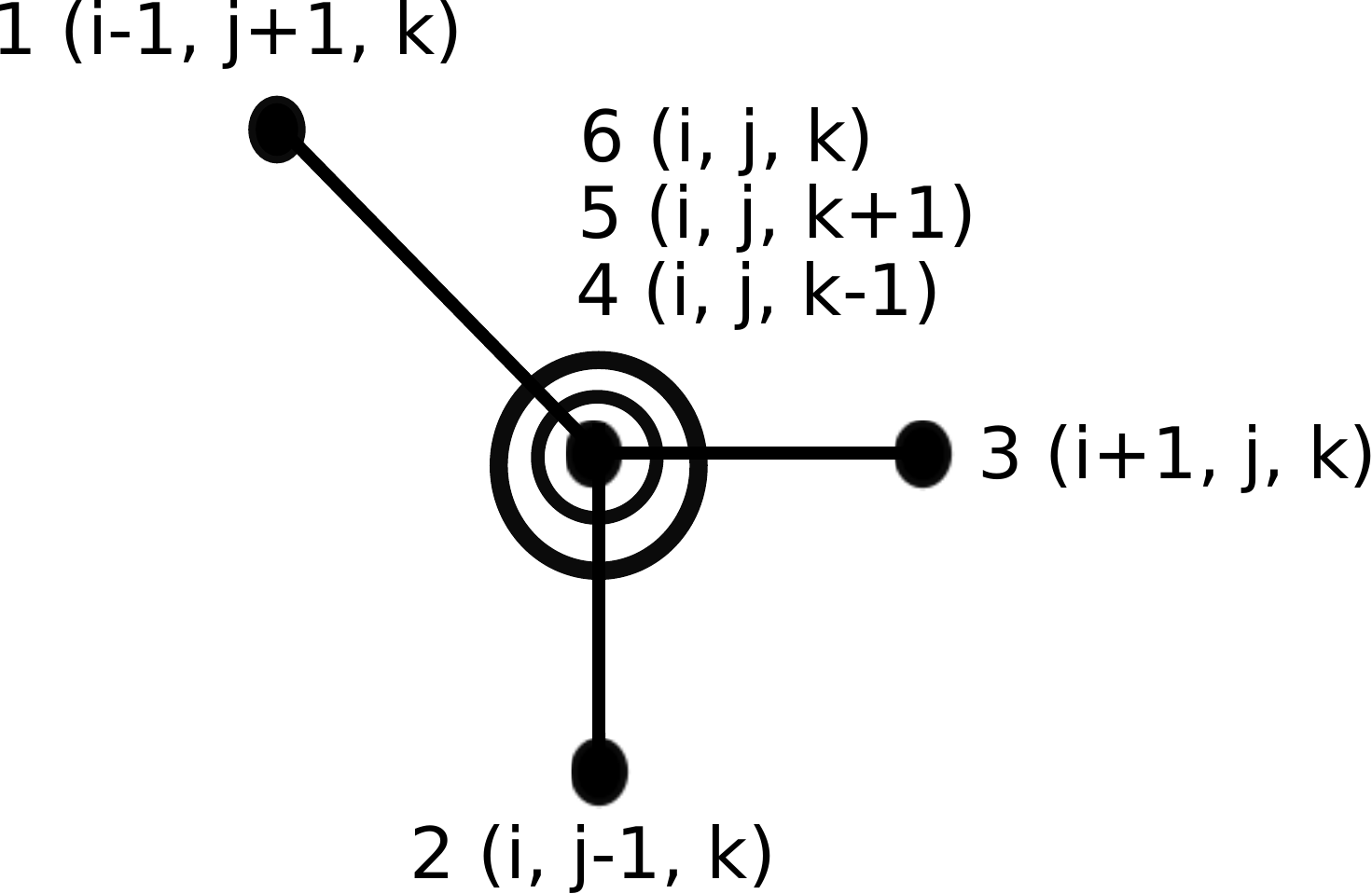} \hspace{9em}
(g) \includegraphics[width=2in]{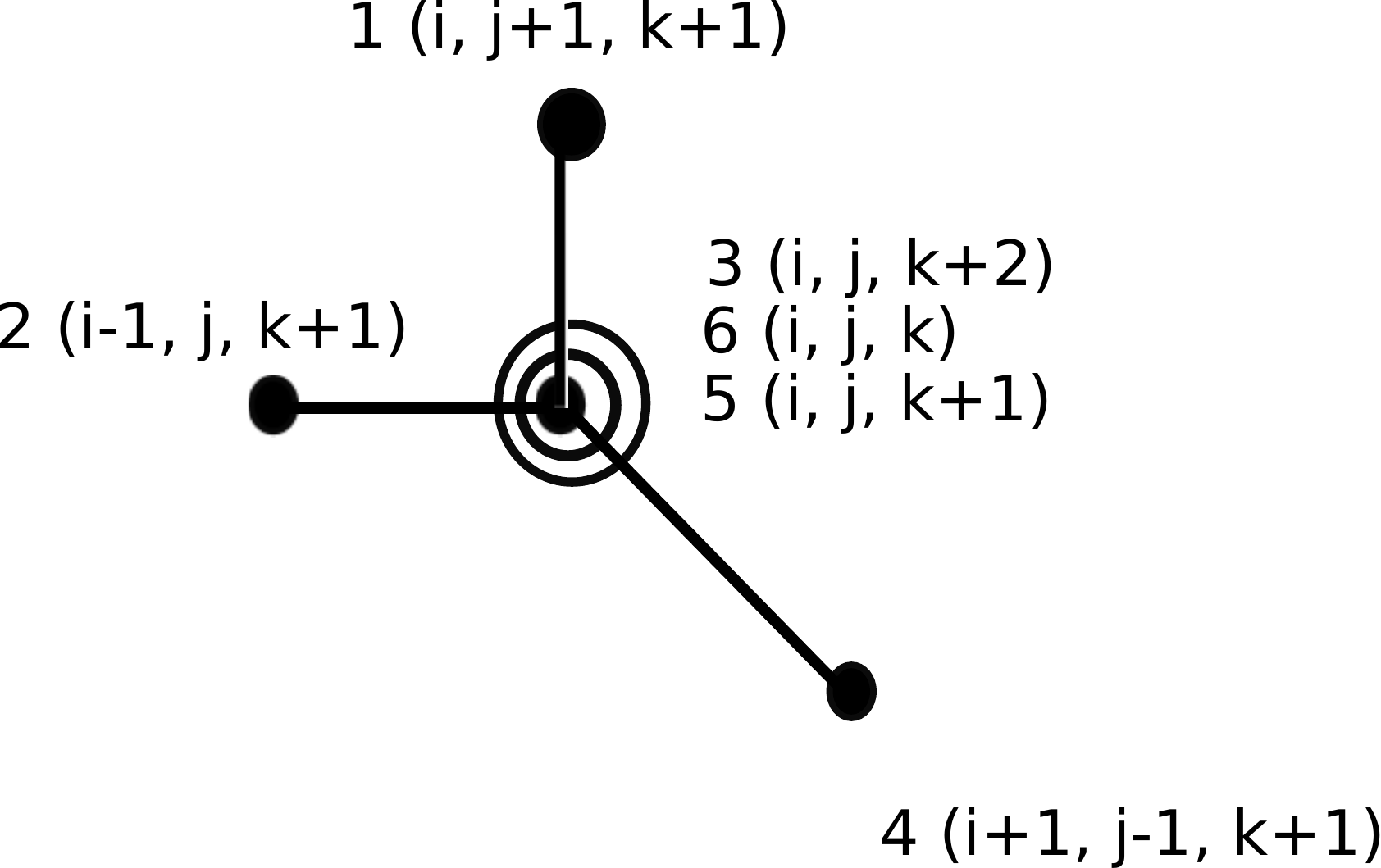} \\ \vspace{2em}
(b) \includegraphics[width=1.5in]{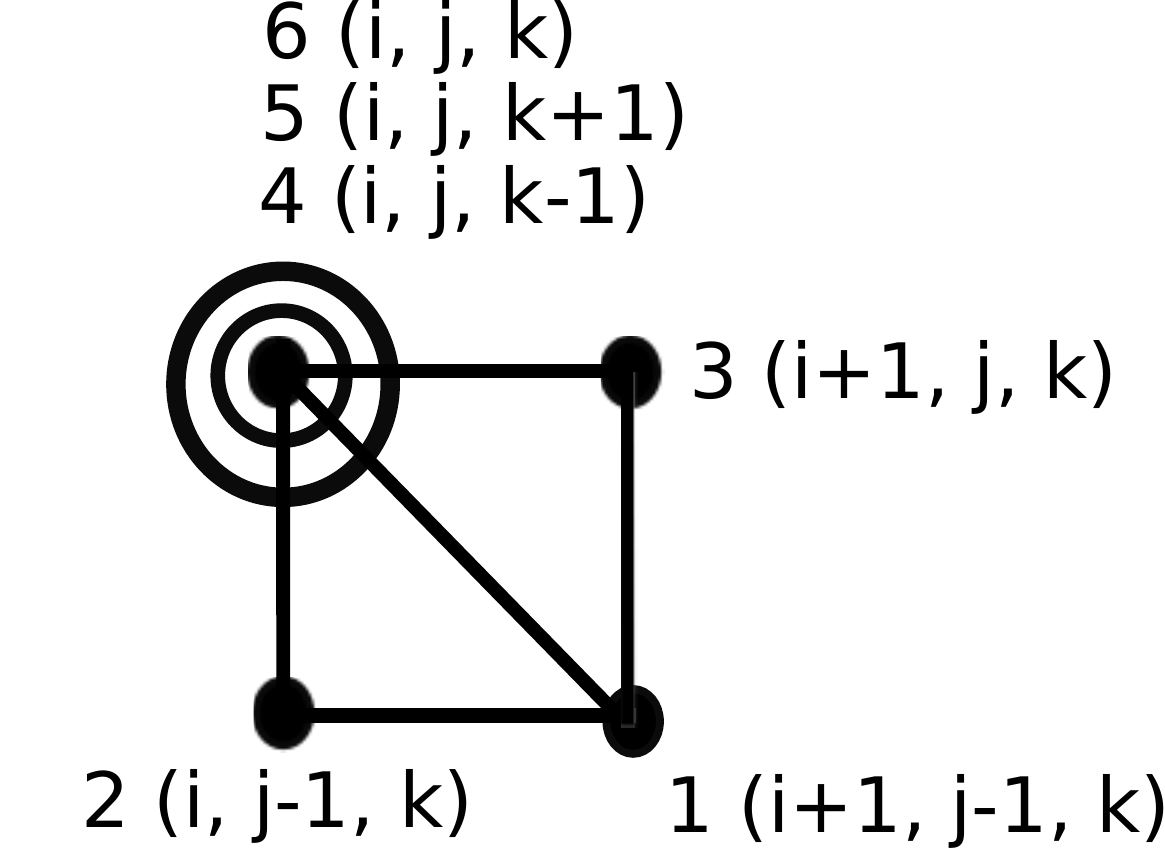} \hspace{0.5em}
(d) \includegraphics[width=1.5in]{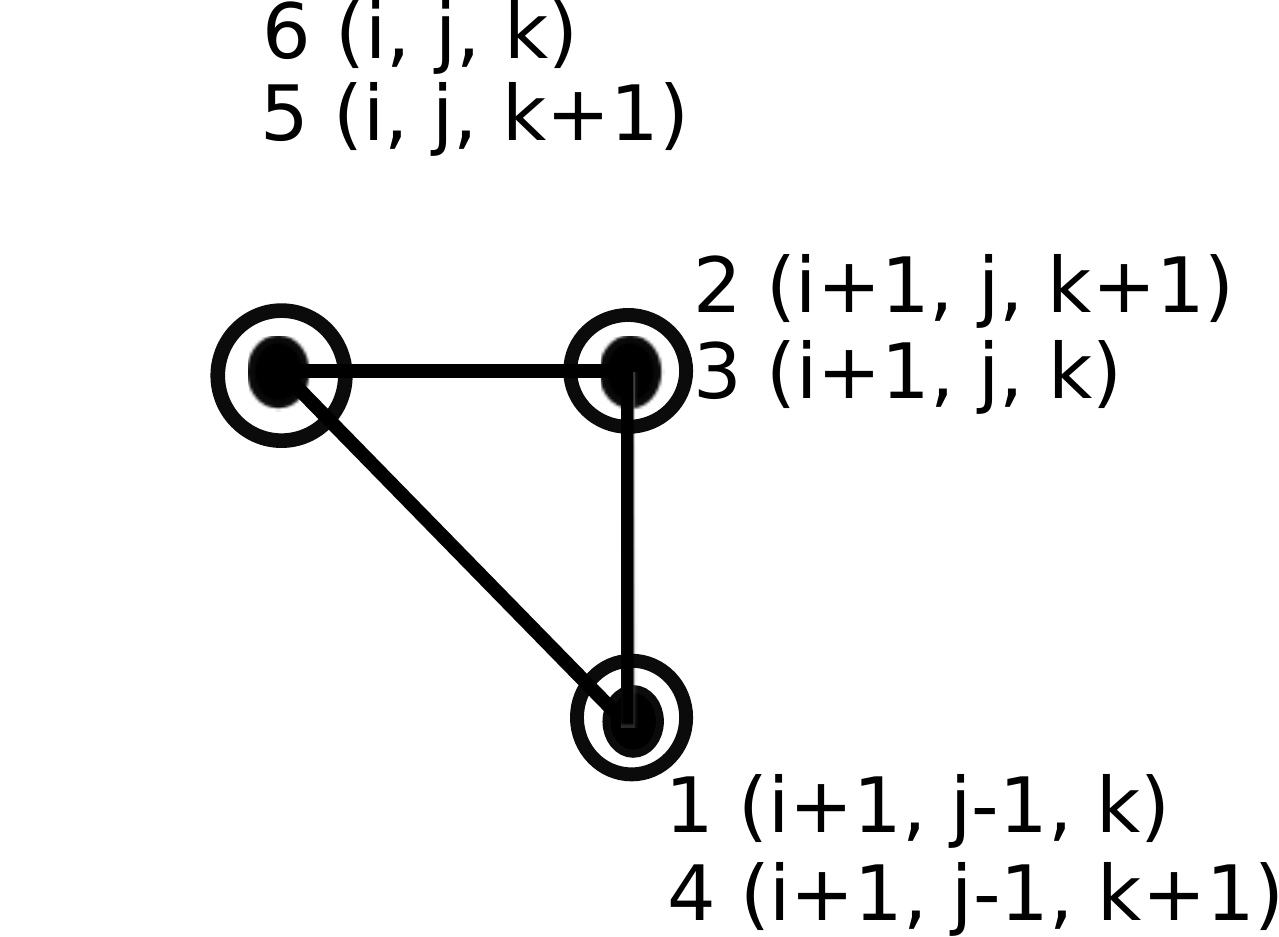} \hspace{0.5em}
(f) \includegraphics[width=1.7in]{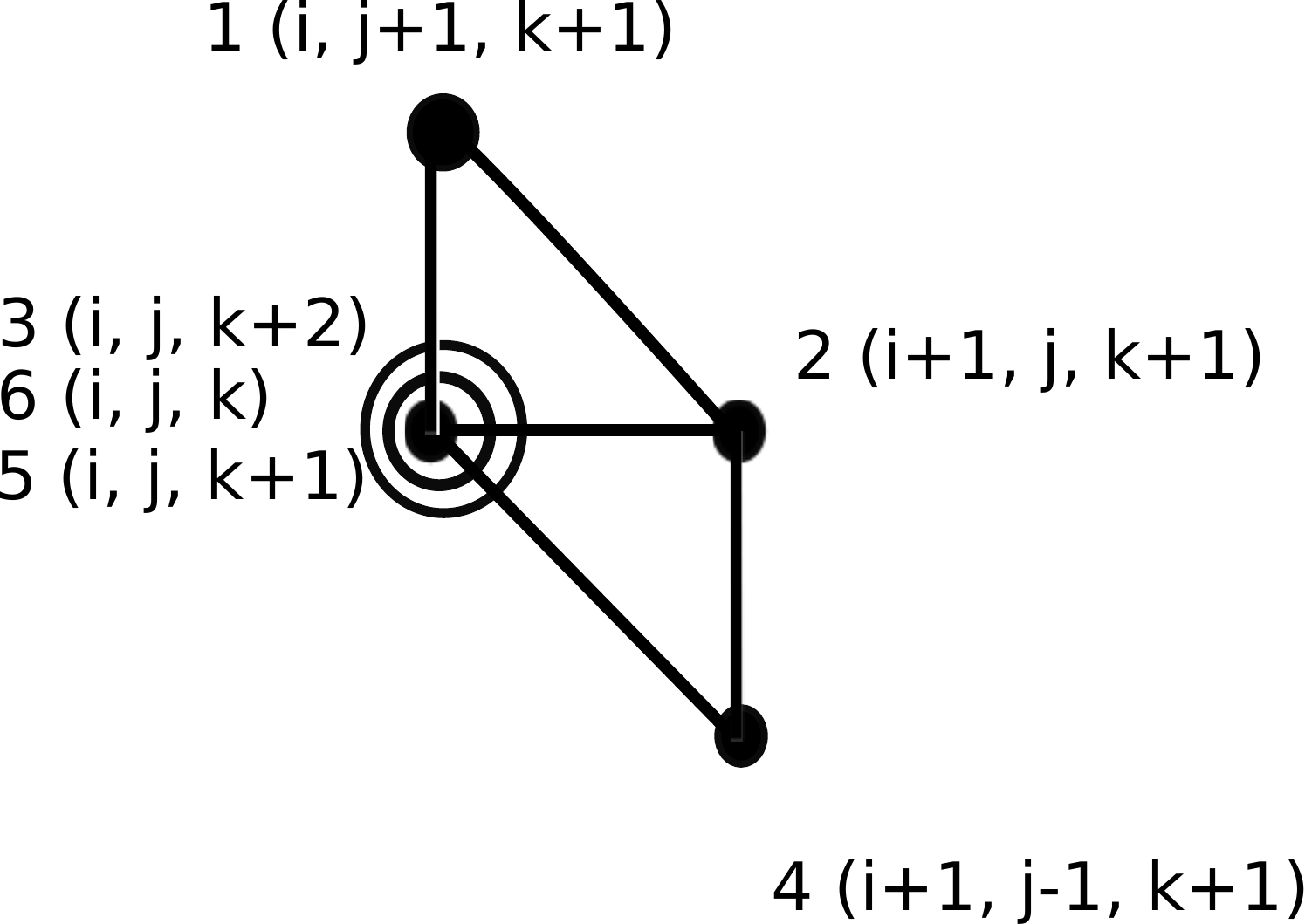} \\ \vspace{2em}
(c) \includegraphics[width=1.8in]{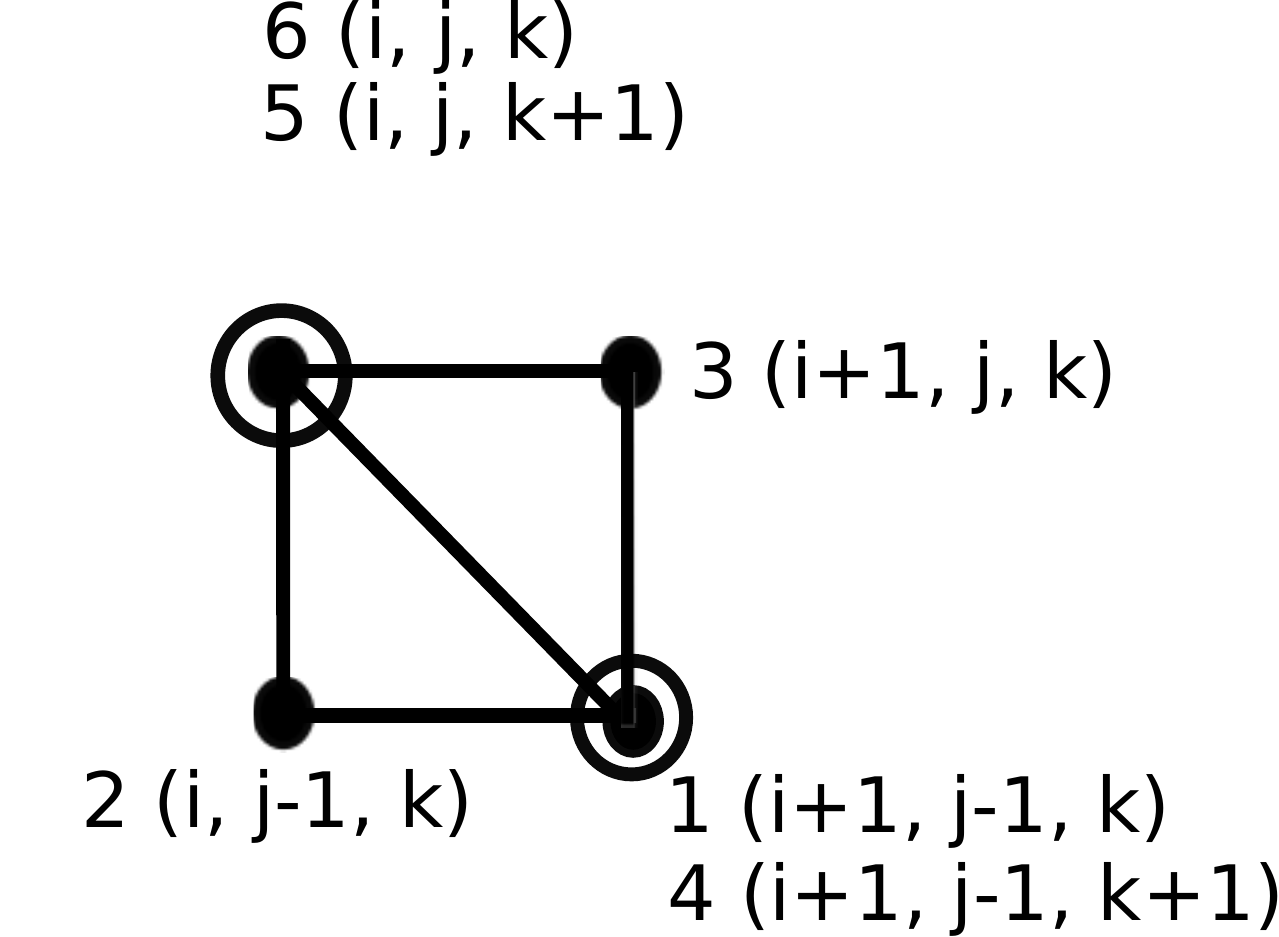} \hspace{2em}
(e) \includegraphics[width=1.8in]{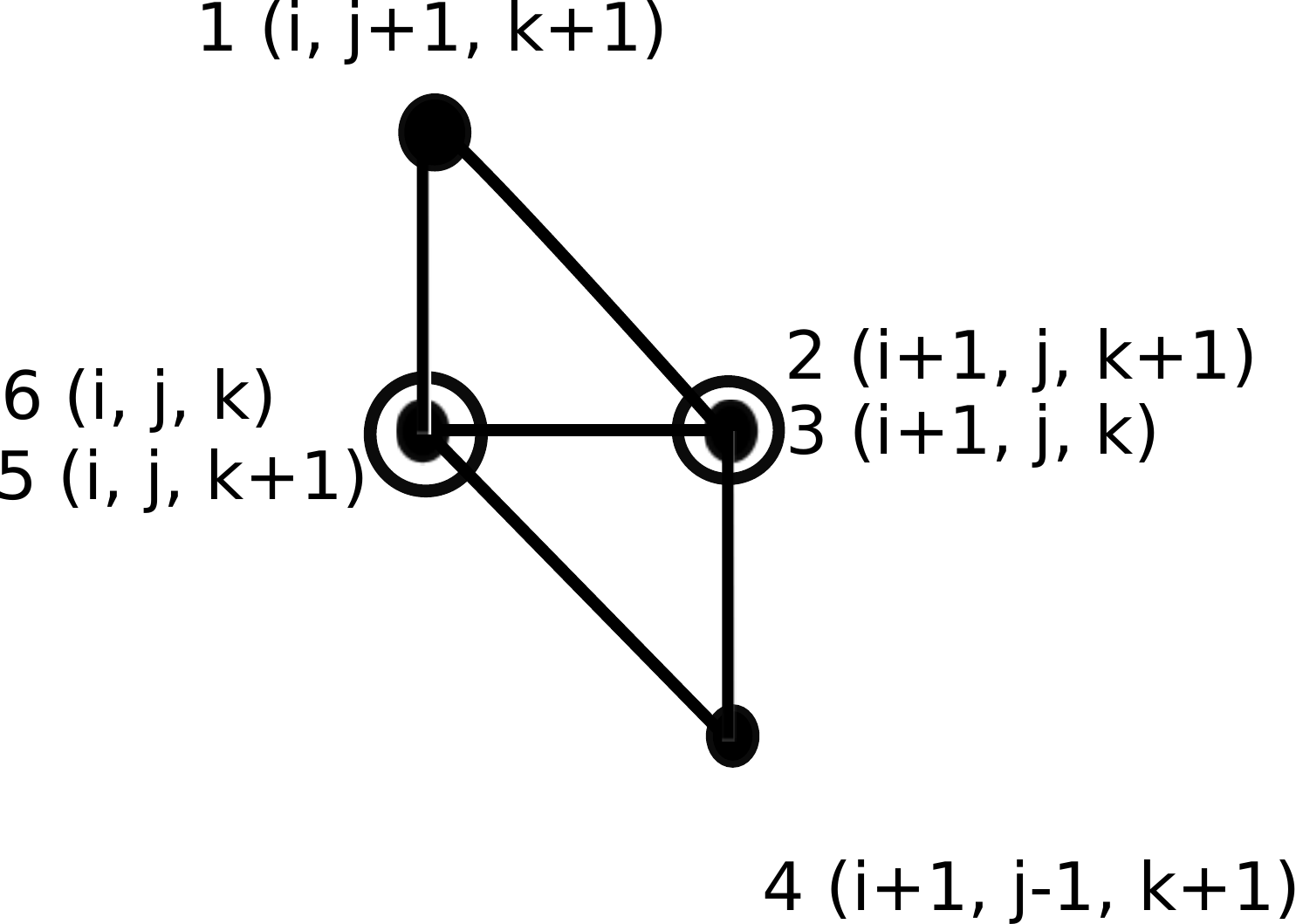}
\caption{Projection of $\widetilde{\Delta_4}$ to the $(i,j)$-plane and the geometric transformations associated to the mutation sequence $\mu_1 \mu_4 \mu_2 \mu_1 \mu_3 \mu_2$.}
\label{fig:IJ}
\end{figure}

Combining together the odd and even cases, we see that starting in the $dP_3$ cluster algebra with the initial cluster parameterized by $\Delta_4$ and iterating superurban renewal yields toric cluster variables whose $\mathbb{Z}^3$-parametrization is one of the following eight possibilities (where $N \in \mathbb{Z}$):
$$(-1,1,2N), ~(0,-1,2N), ~(1,0,2N), ~(0,0,2N),$$ 
$$(1,-1,2N+1), ~(0,1,2N+1), ~(-1,0,2N+1),  ~(0,0,2N+1).$$
Focusing on the toric cluster variables parameterized by points in $\mathbb{Z}^3$ of each of these eight one-parameter families yields the leftmost equality in each identity appearing in Proposition \ref{prop:AZ}.  The rightmost equality in each of these identities follows accordingly as an application of Theorem \ref{thm:z4} and simple algebra.
\end{proof}

\begin{remark} \label{rem:AB} Up to $(x_1/x_2/x_3)$-symmetry, the possible toric variables reached by superurban renewal come in four different one-parameter families: 
{\small $$\varphi\left(A(0,0,2N+1)\right), \varphi\left(A(0,0,2N+2)\right), \varphi\left(A(1/2,1/2,2N)\right), \mathrm{~and~} \varphi\left(A(1/2,1/2,2N+1)\right).$$}  

\vspace{-1em} \noindent Furthermore, letting $x_1=\dots=x_6=1$, the resulting one-parameter families agree with $A_{2N+1}$, $A_{2N+2}$, $B_{2N}$, and $B_{2N+1}$, respectively, where the integer sequences $A_n$ and $B_n$ are defined in \cite[Sec. 6.1]{KP} by the initial conditions $A_0=A_1=A_2=B_0=1$ and the recurrences (for $n \geq 1$)
$$A_{n+2} =  \frac{B_{n-1}^6 + 2A_{n}^3B_{n-1}^3 + 3A_{n-1}A_{n}A_{n+1} B_{n-1}^3 + (A_{n-1}A_{n+1}+A_{n}^2)^3}{A_{n-1}^2 B_{n-1}^3}, \mathrm{~and~}$$
$$B_n = \frac{A_n^3 + A_{n-1}A_nA_{n+1} + B_{n-1}^3}{A_{n-1}B_{n-1}}.$$
In fact, the first and third families coincide and we obtain $A_{2N+2} = 14^{N(N+1)}$, $A_{2N+1} = B_{2N} = 14^{N^2}$, and $B_{2N+1} = 3\times 14^{N(N+1)}$, as in Equation (6.3) of \cite{KP}. \end{remark}

In addition to an algebraic interpretation, Kenyon and Pemantle give a combinatorial interpretation for the $(\mathbb{Z}/2)^3$-coordinates in terms of $\mathcal{I}_2$ using {\bf taut double-dimer configurations} for two of the eight above families, i.e. for $A(0,0,2N)$ and for $A(0,0,2N+1)$, which can be thought of as the combined one-parameter family $A(0,0,n+2)$.
Their combinatorial model is defined in terms of height functions and utilizes an initial configuration $M_0$ on an infinite graph $G_\infty$ (see Figure \ref{fig:KP9}, which is Figure 9 from \cite{KP}).  They define a double-dimer configuration on a graph $G$ to be {\bf taut} if the paths induced by the double-dimers looks like $M_0$ away from the center of $G$.  By applying a sequence of superurban renewals to $G_\infty$, starting at the center, then at the three hexagaonal faces closest to the center, and so on, we obtain a new graph equalling an intersection of a subgraph of the $4-6-12$ graph with $G_\infty$.  Hence, we are able to translate Kenyon and Pemantle's notion of taut double-dimers to our setting.  

In particular, this specific family for $z_{0,0,n+1}^{(4)}$, corresponding to $\varphi\left(A(0,0,n+2)\right)$, can be constructed by placing an equilateral triangle (with side lengths $n$) on $\mathcal{T}_4$ centered around a hexagonal face labeled $4$ (resp. dodecagonal face labeled $6$ or hexagonal face labeled $5$) if $n$ is congruent to $1$ (resp. $2$ or $3$) modulo $3$.  We then consider mixed dimer/double-dimer configurations such that the vertices on the boundary are all $1$-valent and the vertices in the interior are $2$-valent.  The taut condition then restricts us to configurations with paths between the boundary vertices which looks like $n$ nested arcs\footnote{Sometimes this is called the ``rainbow condition'' and agrees with the tripartite pairing boundary condition of \cite[Sec. 6]{KW-tripart}; we thank Helen Jenne for turning our attention to \cite{KW-tripart}.}  along each of the three corners of the equilateral triangle.  See the bottom rows of Figures \ref{fig:A003}, \ref{fig:A004}, and \ref{fig:A005} for the relevant models for $A(0,0,3)$, $A(0,0,4)$, and $A(0,0,5)$, respectively.

Our combinatorial result, Theorem \ref{thm:formula4}, does not directly apply to the cases of $z_{0,0,n+1}^{(4)}$ where $n\geq 1$ since the correspoinding contours $$\mathcal{C}_4(n+1,-n-1,n+1,-n,n,-n)$$ are self-intersecting, as in Figure \ref{fig:contours} (Right).  However, by directly comparing with Theorem 4.1 of \cite{KP} and using the dictionary from Proposition \ref{prop:AZ}, we can make sense of the combinatorics of this specific one-dimensional family of self-intersecting contours by the following rule (which is a variant of Definition \ref{def:contour4}, which applied in the case of contours without self-intersections):

\begin{definition} \label{def:self-int}
Given a contour $\mathcal{C}_4(a,b,c,d,e,f)$ which has the sign pattern given as $(+,-,+,-,+,-)$, where further we require $a$, $b$, and $c$ are nonzero, we define the core subgraph $\mathcal{G}_4(a,b,c,d,e,f)$ of the Model $4$ brane tiling $\mathcal{T}_4$ by the following.

Step 0: Even though it was not a sign pattern handled in \cite{lai', lai}, we can apply a reverse of the coordinate changes discussed in Section \ref{sec:Mod4}, and consider the $6$-tuple $(A,B,C,D,E,F) = (a-1,b+1,c-1,d,e,f)$.  Note that the entries of this $6$-tuple alternate in sign and we have $A+B+C+D+E+F=0$.  Compare Figures \ref{fig:B2-3} and \ref{fig:B2-all-rotations} for an example of this coordinate change and the role of forced edges along sides $b$, $c$, and $f$.

Step 1: Superimpose the modified contour $\mathcal{C}_4'(A,B,C,D,E,F)$ onto $\mathcal{T}_4$ starting from the {\bf barycenter of the hexagon labeled $5$} and follow six straight lines as in Figures \ref{fig:B3-3}, \ref{fig:B4-1}, or \ref{fig:B5-1} which because of this specific sign pattern will result in the modified contour consisting of two overlapping triangles.  This will lead to two regions of vertices, one with vertices counted once and an inner region with vertices counted twice (including on the border of the two regions).

Step 2: For any side of positive (resp. negative) length, we decrease the multiplicity by $1$ of all black (resp. white) vertices along that side, as well as all edges along that side.
We define $\mathcal{G}_4(a,b,c,d,e,f)$ to be the subgraph induced on the remaining vertices (and recording which vertices have multiplicity one and which have multiplicity two).

Using the fact that the modified contour $\mathcal{C}'_4(A,B,C,D,E,F)$ is built as two overlapping triangles, the innermost triangle of $\mathcal{G}_4(a,b,c,d,e,f)$ resembles a graph in the family for $\varphi\left(A(0,0,n+2)\right)$ except that it may have vertices of multiplicity two on its boundary.  In fact, one of its three boundaries will again have all of its vertices of multiplicity one, and it follows from the construction that this boundary will have an even number of such vertices.   We define the {\bf Taut Condition} analogously to the above for $\varphi\left(A(0,0,n+2)\right)$ so that the multiplicity one vertices along this special boundary are always connected to the multiplicity one vertices on the side nearer to it and in a non-crossing connectivity pattern.
\end{definition}

See the top rows of Figures \ref{fig:A003}, \ref{fig:A004}, and \ref{fig:A005} for the constructions for $A(0,0,3)$, $A(0,0,4)$, and $A(0,0,5)$, respectively.

\begin{remark}
For the self-intersecting contours considered in this section, i.e. with the sign pattern $(+,-,+,-,+,-)$ where $a$, $b$, and $c$ are nonzero, the corresponding modified contours are especially easy to describe since any two sides meet at an acute ($60$ degree) angle as opposed to an obtuse ($120$ degree) angle, which occurs when a contour has two sides in a row of the same sign.
\end{remark}

\begin{remark}The taut condition in Definition \ref{def:self-int} differs from that in \cite[Sec. 6]{KW-tripart} or \cite{KP} since it is not sufficient to forbid connections between two vertices on the same side of the triangle and we have vertices of both multiplicity one and two on the boundary between the dimer and double-dimer regions.  In particular, a path originating from a vertex of multiplicity one on the boundary of the double-dimer region may terminate in the interior of the dimer region rather than another multiplicity one vertex on the boundary between the two regions.  See Figure \ref{fig:forbidden-B3}.  In theory, this would allow one of the patterns in the bottom row of this figure.  But, by looking at small examples, we must forbid a path crossing from left to right (or vice-versa) so that our enumeration agrees with our formulas.
\end{remark}

Following Theorem 4.1 of \cite{KP}, assigning weights to the taut configurations of $\mathcal{G}_4(n+1,-n-1,n+1,-n,n,-n)$ using methods analogous to that in Definition \ref{def:CM4} and Theorem \ref{thm:formula4}, except that closed cycles of size $\geq 4$ in double-dimer configurations come with a coefficient of $2$, yields the Laurent expansions for $\varphi{(A(0,0,n+2)} = z_{0,0,n+1}^{(4)}$.

\begin{figure}
\includegraphics[width=2.7in]{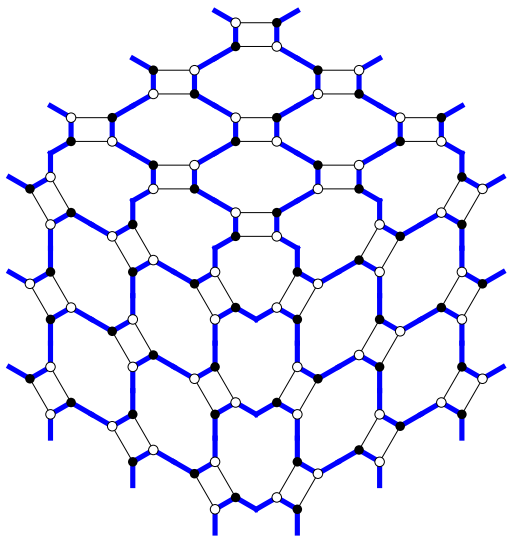}
\caption{The initial configuration $M_0$ on the infinite graph $G_\infty$ as illustrated in \cite[Fig. 9]{KP}.  Taut double-dimer configurations agree with $M_0$ outside of the center.}
\label{fig:KP9}
\end{figure}

\begin{figure}
\includegraphics[width=5in]{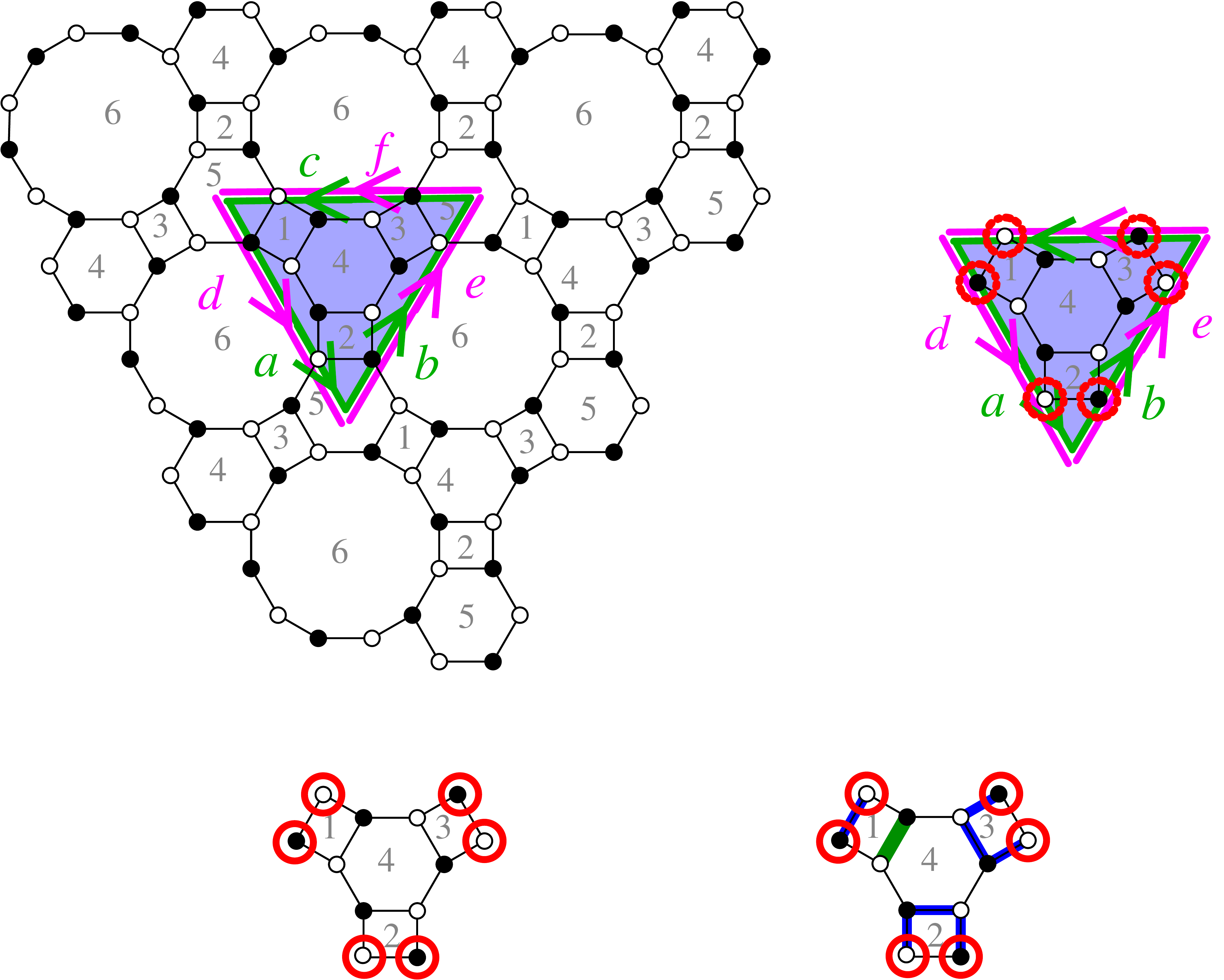}
\caption{Step-by-step construction of the double-dimer configuration from the self-intersecting contour $\mathcal{C}_4(2,-2,2,-1,1,-1)$ corresponding to $z_{0,0,2}^{(4)}$.  We use the modified contour $\mathcal{C}'_4(1,-1,1,-1,1,-1)$. Circled vertices in the bottom row are multiplicity $1$ while the vertices on the hexagonal face $4$ have multiplicity $2$.  The bottom right figure illustrates one of the possible taut double-dimer configurations.  The thickened green edge here, and in the sequel, is a doubled edge.}
\label{fig:A003}
\end{figure}

\begin{figure}
\includegraphics[width=6in]{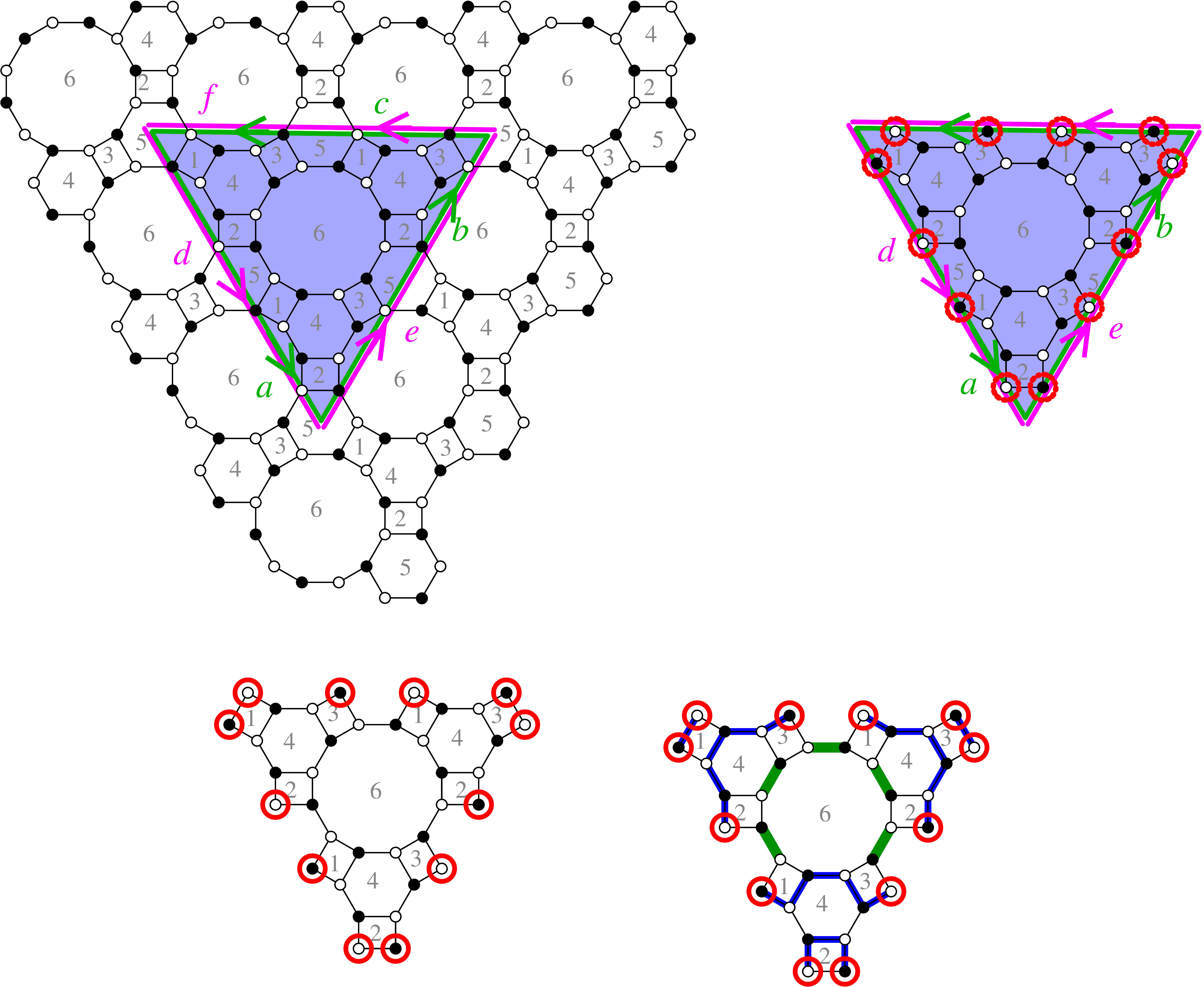}
\caption{Construction for $z_{0,0,3}^{(4)}$ using the self-intersecting contour $\mathcal{C}_4(3,-3,3,-2,2,-2)$ and the modified contour $\mathcal{C}'_4(2,-2,2,-2,2,-2)$.  A taut double-dimer configuration is shown in the bottom right.}
\label{fig:A004}
\end{figure}

\begin{figure}
\includegraphics[width=6in]{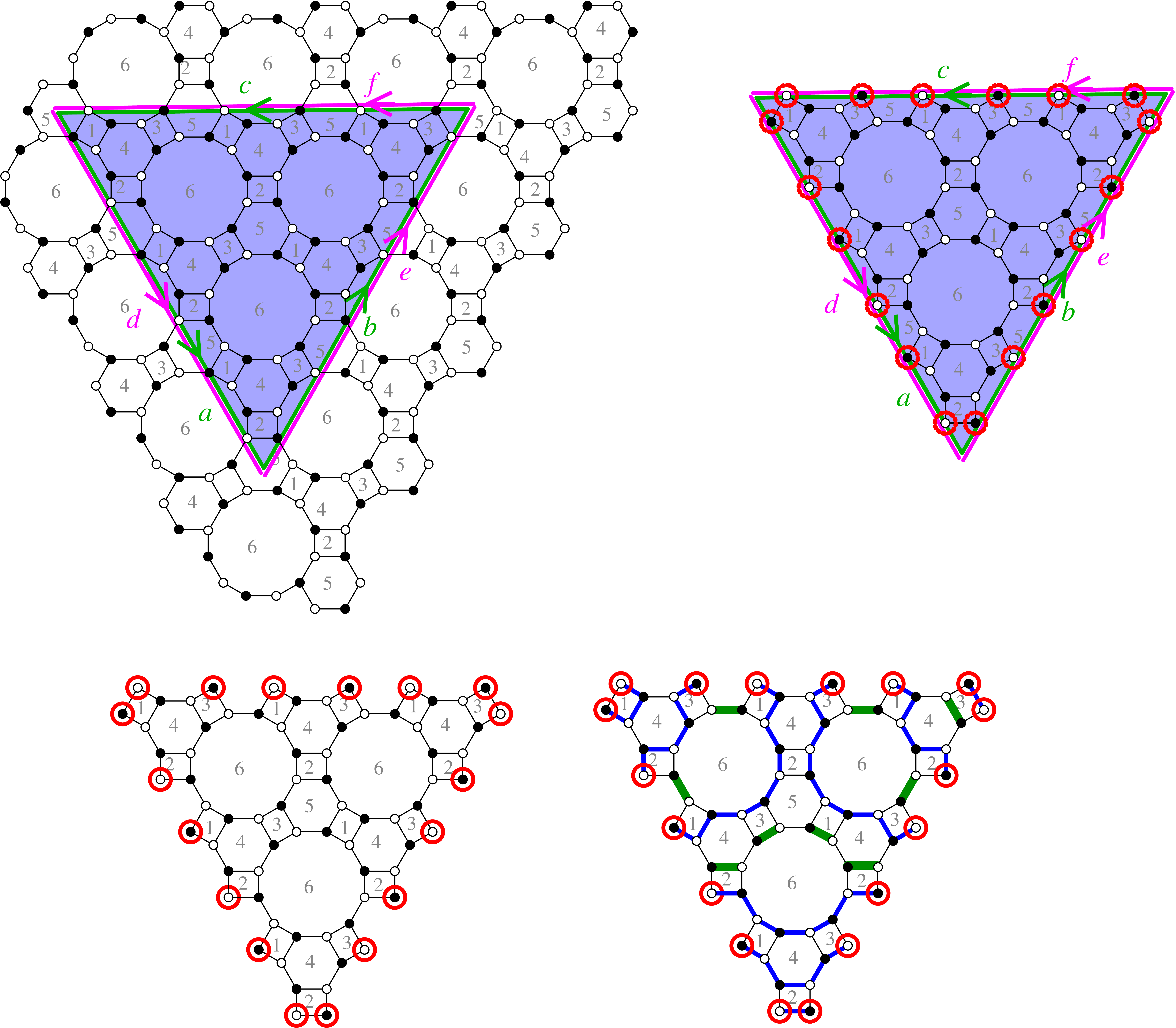}
\caption{Construction for $z_{0,0,4}^{(4)}$ using the self-intersecting contour $\mathcal{C}_4(4,-4,4,-3,3,-3)$ and the modified contour $\mathcal{C}'_4(3,-3,3,-3,3,-3)$.  A taut double-dimer configuration is shown in the bottom right.}
\label{fig:A005}
\end{figure}

\begin{conjecture} \label{conj:self-int}  The above rule for constructing a subgraph (with vertices of multiplicity one and two) for self-intersecting contours, as inspired by the one-dimensional family $\varphi\left(A(0,0,n+2)\right) = z_{0,0,n+1}^{(4)}$ studied in \cite{KP}, can be extended to all other self-intersecting contours, i.e. the three-dimensional space of toric cluster variables {\small $$\{z_{i,j,k}^{(4)}: \left(-k < i,j, i+j < k-1 \mathrm{~for~}k > 0\right) \mathrm{~~and~~} \left( k-1 < i,j, i+j < -k \mathrm{~for~}k \leq 0 \right) \},$$} 

\vspace{-1em} \noindent using mixed configurations where vertices are multiplicity $1$ except for a simply-connected region of vertices of multiplicity $2$ in the interior.  Further, there is an analogue of the taut condition for crossing the inner region, i.e. the double-dimer region.  Lastly, this rule degenerates into the usual rule for dealing with contours (using dimers only) when the self-intersection on the contour disappears, i.e. when $i$, $j$ or $i+j$ falls outside of $[-k,k-1]$ or $[k-1,k]$.
\end{conjecture}

As a first test case for this conjecture, we consider the half-integer coordinates 
$$\{A(i+\frac{1}{2},j+\frac{1}{2}, k) \}~~\cup~~  \{A(i+\frac{1}{2},j, k+\frac{1}{2})\} ~~\cup~~  \{A(i,j+\frac{1}{2}, k+\frac{1}{2})\}$$ for $i,j,k\in \mathbb{Z}$ that are proven to be Laurent polynomials but lack a combinatorial interpretation in \cite{KP}.

We already illustrated the subgraph for $\varphi\left(A(\frac{1}{2},0,\frac{3}{2})\right)$, i.e. $z_{-1,0,1}^{(4)}$, in Figure \ref{fig:Model4initial} when we drew the subgraph for $C_4^{(4)} = (1,0,0,0,1,-1)$.  Letting $x_1=\dots=x_6$, the number of perfect matchings in this subgraph is $B_1=3$.  We next illustrate the contours for $\varphi\left(A(\frac{1}{2},\frac{1}{2},2)\right)$, $\varphi\left(A(\frac{1}{2},0,\frac{5}{2})\right)$, and $\varphi\left(A(0,\frac{1}{2},\frac{5}{2})\right)$, which correspond to 
$z_{-1,1,2}^{(4)}$, $z_{1,0,2}^{(4)}$, and $z_{0,-1,2}^{(4)}$, respectively, in Figure \ref{fig:B2-3}.  In Figure \ref{fig:B2-all-rotations}, we illustrate the core subgraphs obtained directly by applying Definition \ref{def:self-int} and the modified contours instead.  Notice that, in both cases, the three resulting subgraphs are $120$ degree rotations of one another, and that the modified contours exhibit this same rotational symmetry, a feature we do not see as directly from the unmodified contours.  Letting $x_1=\dots=x_6$, all three of the Laurent polynomials obtained by counting perfect matchings of these subgraphs collapse to $B_2=14$.

Applying Definition \ref{def:self-int} again, we build subgraphs (and determine the multiplicities of vertices) for $\varphi(A(0,\frac{1}{2},\frac{7}{2}))$, $\varphi(A(0,\frac{1}{2},\frac{9}{2}))$, and $\varphi(A(0,\frac{1}{2},\frac{11}{2}))$, which correspond to the toric variables $z_{0,1,3}^{(4)}$, $z_{0,-1,4}^{(4)}$, and $z_{0,1,5}^{(4)}$, respectively.  See Figures \ref{fig:B3-3}, \ref{fig:B4-1}, and \ref{fig:B5-1}.  These toric variables become $B_3 = 3 \times 14^2$, $B_4 = 14^4$, and $B_5 = 3 \times 14^6$, respectively, when $x_1=\dots=x_6=1$.  However unlike the $B_1$ and $B_2$ cases (where we got ordinary subgraphs), or the $A_n$ cases (where the entire interior of the core subgraph consisted of vertices of multiplicity 2), the resulting core subgraphs for $B_n$ (for $n\geq 3$) will contain both vertices of multiplicity one and of multiplicity two in its interior. Using Figure \ref{fig:B3-3}, we have verified that the number of mixed configurations (consisting of taut double-dimer configurations on the innermost region and dimer configurations outside of that) on $\mathcal{G}_4(4,-4,3,-1,1,-2)$, which corresponds to the modified contour $\mathcal{C}'_4(3,-3,2,-1,1,-2)$, counting cycles appropriately, is $B_3$ = 588 as desired.  See Figure \ref{fig:forbidden-B3} for examples of configurations on this graph disallowed by the taut condition.  We conjecture the associated Laurent polynomials agree as well and that this pattern continues for the entire $B_n$ sequence.

\begin{figure}
\includegraphics[width=6in]{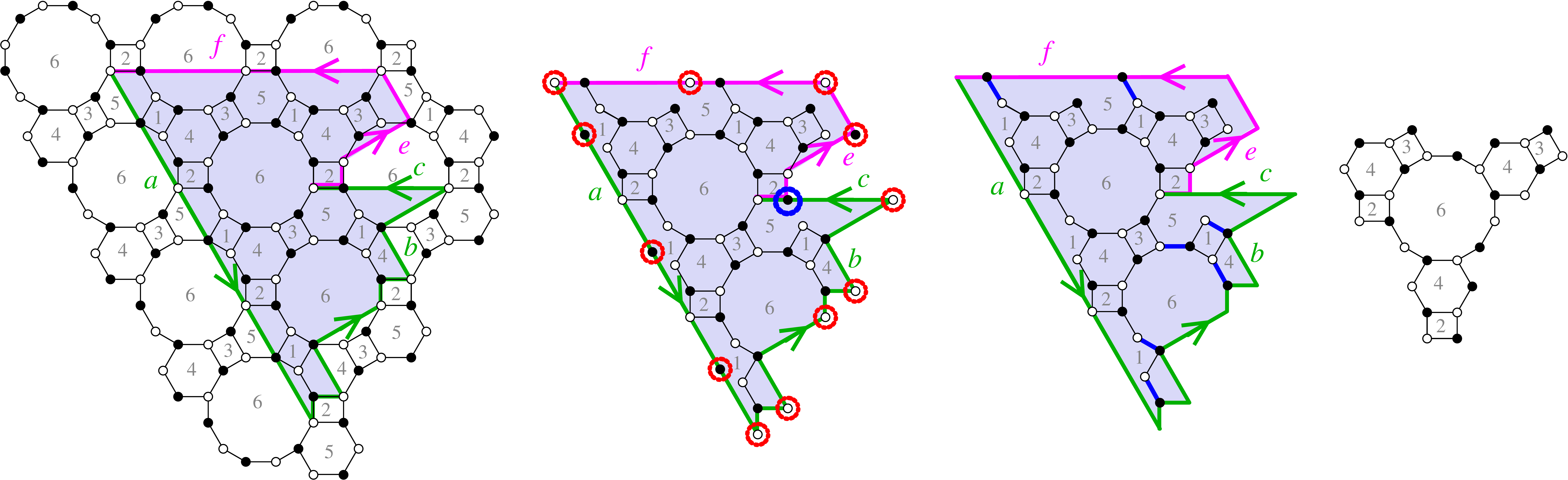}

\includegraphics[width=6in]{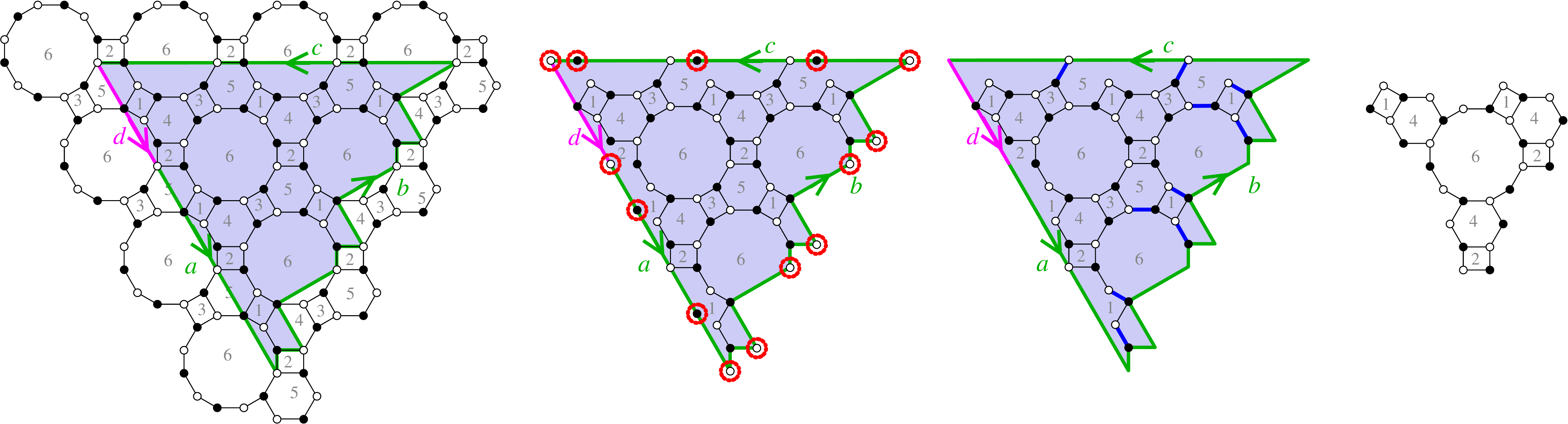}

\includegraphics[width=6in]{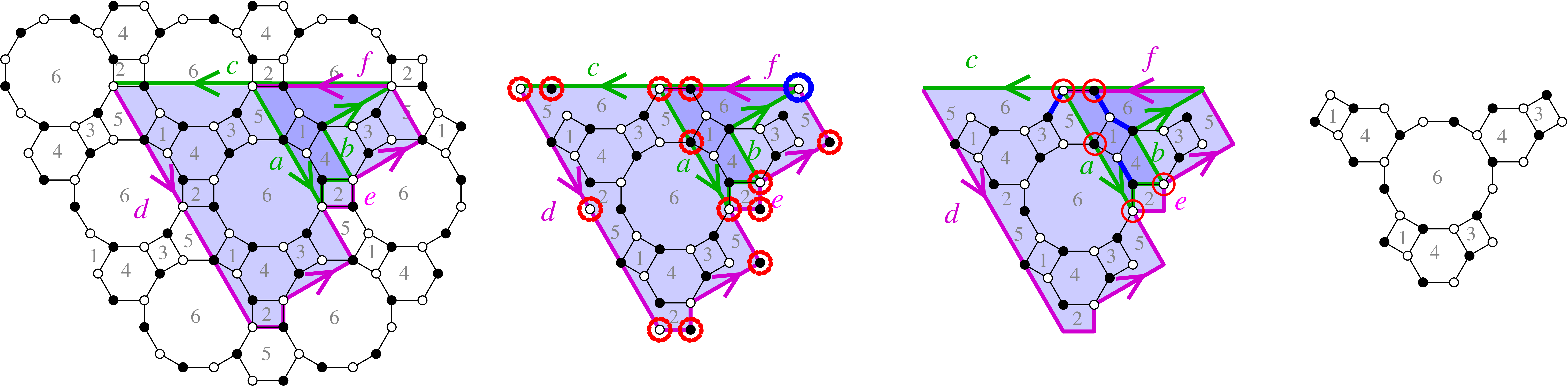}
\caption{Construction of the subgraphs coming from the contours  
$\mathcal{C}_4(3,-2,1,0,1-2)$ corresponding to $z_{-1,1,2}^{(4)}$ (top),
$\mathcal{C}_4(2,-3,3,-1,0,0)$ corresponding to $z_{1,0,2}^{(4)}$ (middle), and
$\mathcal{C}_4(1,-1,2,-2,2,-1)$ corresponding to $z_{0,-1,2}^{(4)}$  (bottom).  The case of $z_{0,-1,2}^{(4)}$ yields a self-intersecting contour, but we follow a rubric for deleting vertices and including forced edges to obtain a rotation of the above subgraphs.}
\label{fig:B2-3}
\end{figure}

\begin{figure}
\includegraphics[width=5in]{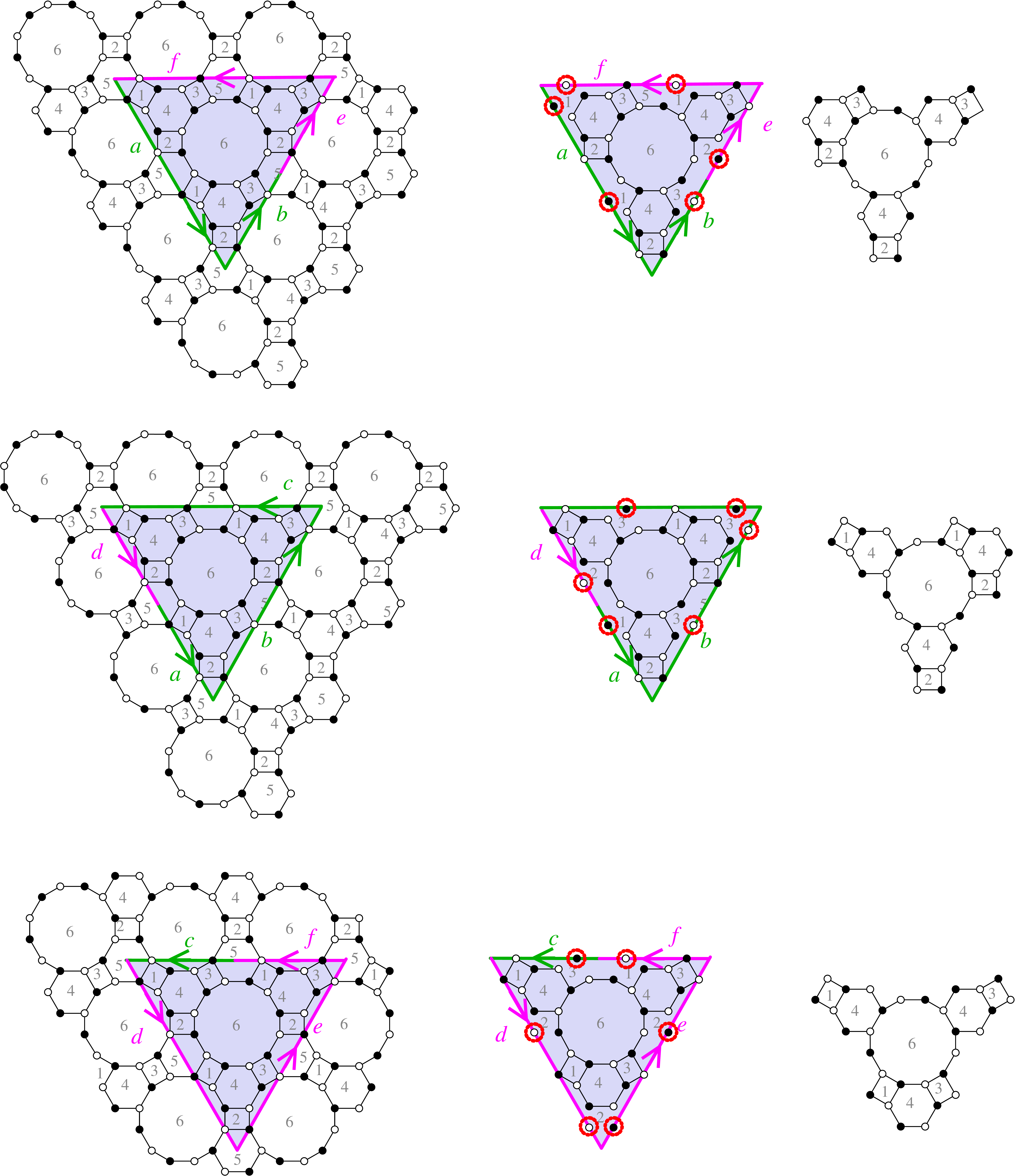}
\caption{Construction of the subgraphs using the modified contours 
$\mathcal{C}'_4(2,-1,0,0,1-2)$ (top),
$\mathcal{C}'_4(1,-2,2,-1,0,0)$ (middle), and
$\mathcal{C}'_4(0,0,1,-2,2,-1)$ (bottom).  The resulting subgraphs agree with those in Figure \ref{fig:B2-3} but no forced edges are required, even for $z_{0,-1,2}^{(4)}$.}
\label{fig:B2-all-rotations}
\end{figure}

\begin{figure}
\includegraphics[width=6in]{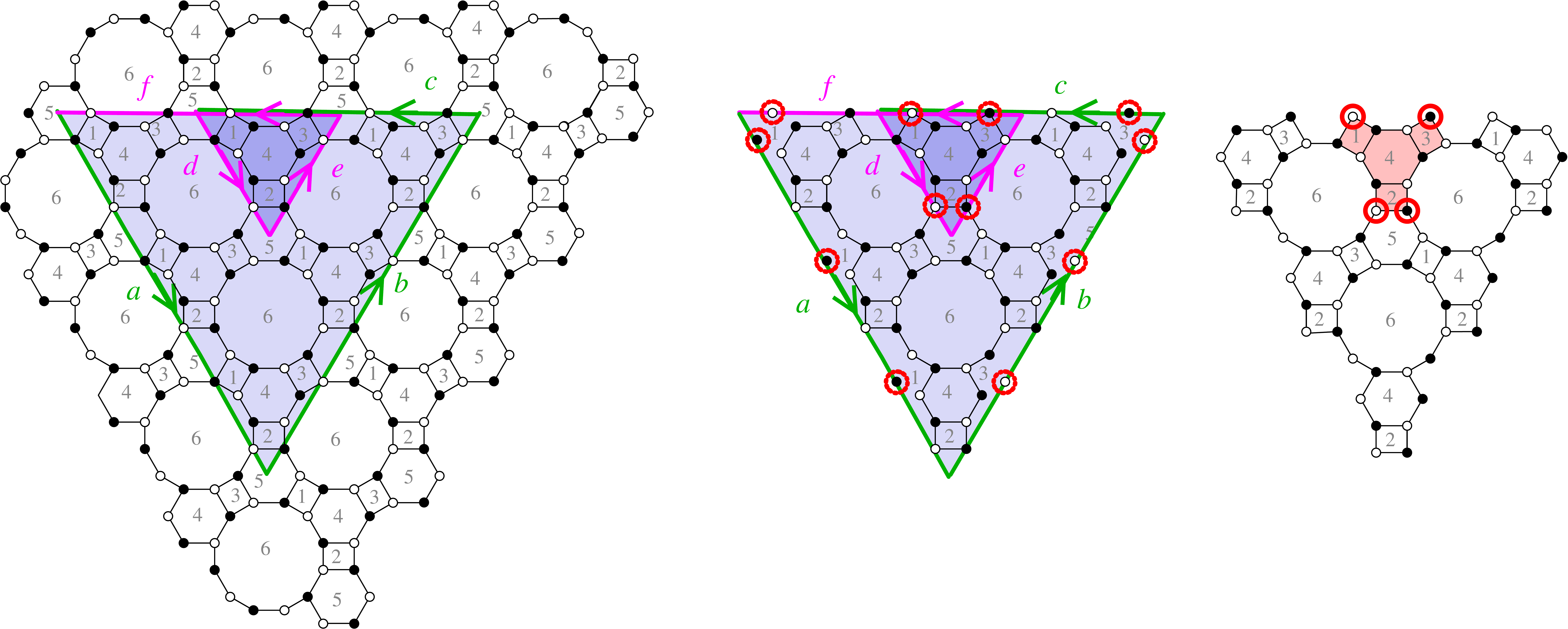}
\caption{Construction for $z_{0,1,3}^{(4)}$ using the self-intersecting contour $\mathcal{C}_4(4,-4,3,-1,1,-2)$ and the modified contour $\mathcal{C}'_4(3,-3,2,-1,1,-2)$.  Circled vertices in the third picture have multiplicity $1$ while the vertices in the interior of the shaded region have multiplicity $2$.}
\label{fig:B3-3}
\end{figure}

\begin{figure}
\includegraphics[width=6in]{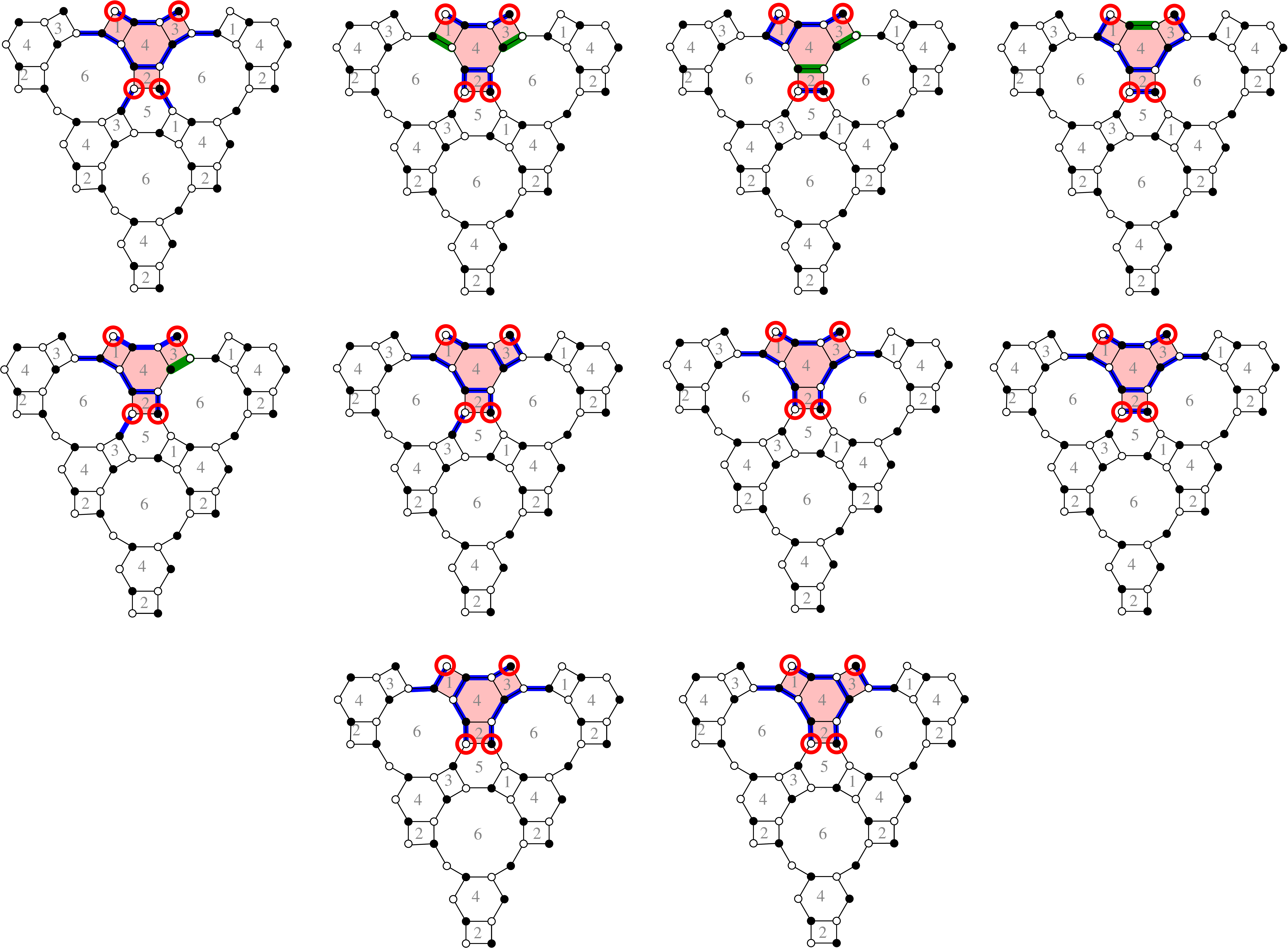}
\caption{Examples of forbidden configurations for the case of $z_{0,1,3}^{(4)}$.  We arrange these representatives into groups according to their connectivity outside the double-dimer region.  Each of these double-dimer configurations can be completed to a mixed configuration in multiple ways by choosing a perfect matching of the remaining vertices.  Most of these configurations (except for the two in the bottom row) are disallowed because the two multiplicity one vertices at the top of the graph are connected to one another by a path.}
\label{fig:forbidden-B3}
\end{figure}

\begin{figure}
\includegraphics[width=6in]{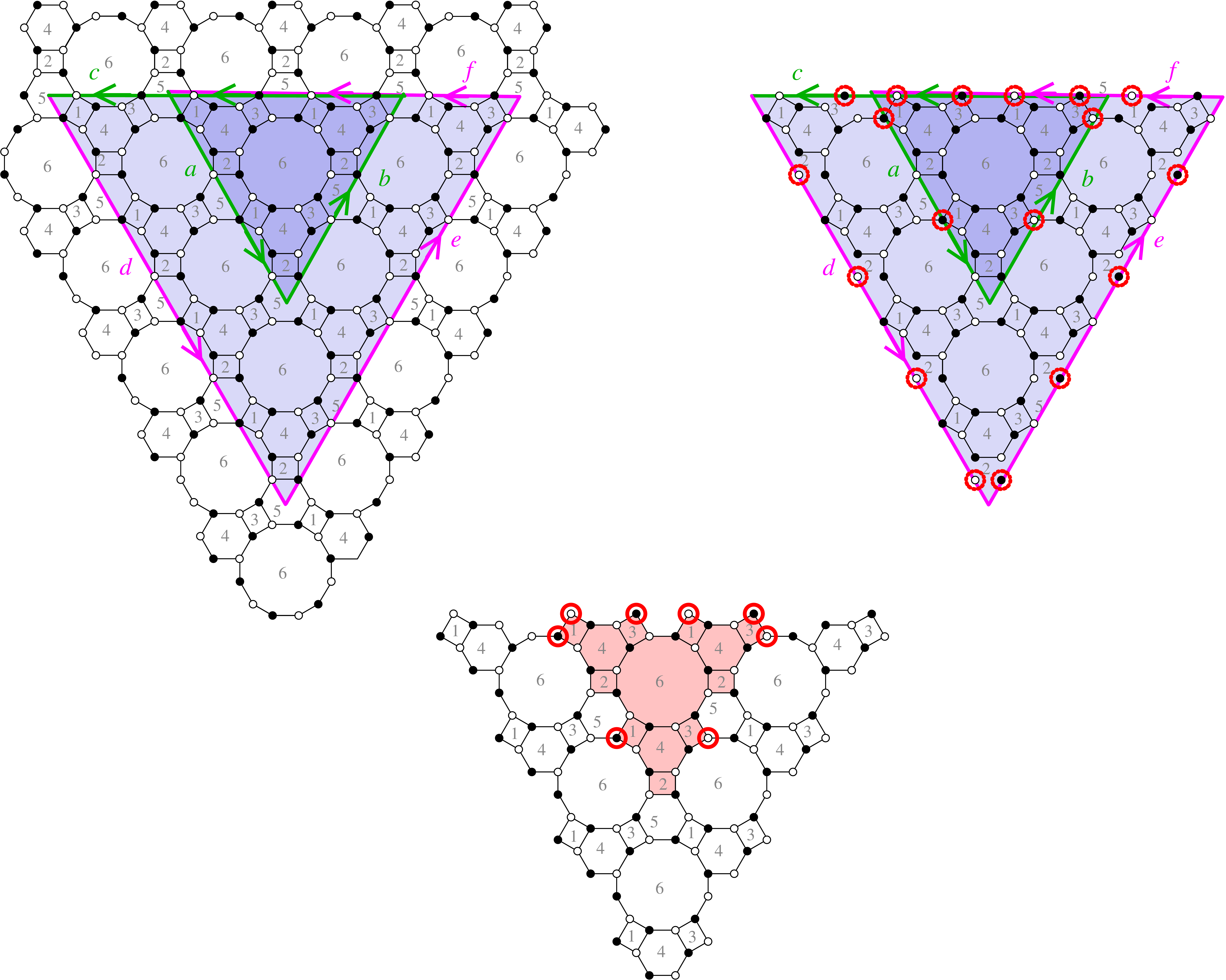}
\caption{Construction for $z_{0,-1,4}^{(4)}$ using the self-intersecting contour $\mathcal{C}_4(3,-3,4,-4,4,-3)$ and the modified contour $\mathcal{C}'_4(2,-2,3,-4,4,-3)$.  Circled vertices in the third picture have multiplicity $1$ while the vertices in the interior of the shaded region have multiplicity $2$.}
\label{fig:B4-1}
\end{figure}

\begin{figure}
\includegraphics[width=6in]{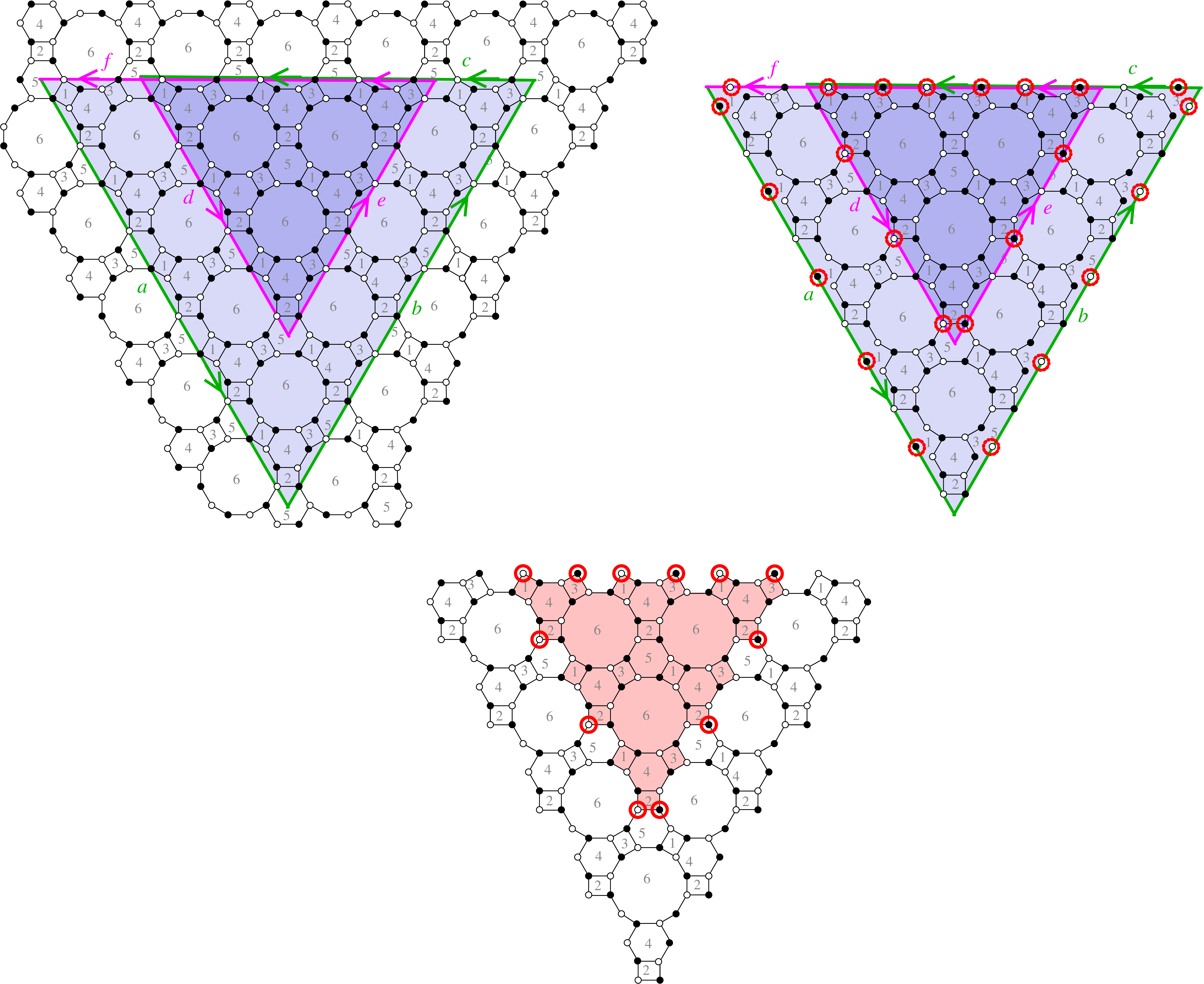}
\caption{Construction for $z_{0,1,5}^{(4)}$ using the self-intersecting contour $\mathcal{C}_4(6,-6,5,-3,3,-4)$ and the modified contour $\mathcal{C}'_4(5,-5,4,-3,3,-4)$.  Circled vertices in the third picture have multiplicity $1$ while the vertices in the interior of the shaded region have multiplicity $2$.}
\label{fig:B5-1}
\end{figure}

\section{Open Questions}

\label{sec:open}

\begin{problem}
As alluded to in the previous section, we wish to better understand the combinatorial interpretations associated to self-intersecting contours where the dimer model is insufficient.  The double-dimer interpretation of Kenyon and Pemantle \cite{KP} provides one such infinite family, the $A_n$ (i.e. $\varphi\left(A(0,0,n+2)\right)$ sequence, in the case of Model 4.  We wish to flesh out this interpretation, as depicted in Conjecture \ref{conj:self-int}, and illuminate how this one-dimensional family relates to the larger set of cluster variables reachable by any toric mutation sequence.  Furthermore, a better understanding of self-intersecting contours in the Model 4 case would provide an additional approach to Problem 9.1 of \cite{LaiMus}, which focused on the Model 1 case, complementing previous examples worked out by the second author and Speyer.
\end{problem}

\begin{problem}
Furthermore, in \cite[Sec. 5 and 6]{KP}, limit shapes of solutions of the hexahedron recurrence are presented.  As detailed above, this (essentially) one-parameter family of double-dimer regions naturally embed inside the larger three-dimensional family of contours from toric cluster variables for Model 4.  Hence, it is an interesting problem to understand limit shapes when there is more than one possible dimension of growth.  Additionally, how do the limit shapes for Model 4 compare with those for Models 1, 2, or 3?
\end{problem}

\begin{problem}
There are numerous additional lattices and contours where the numbers of tilings (equivalently perfect matchings of their dual graphs) are counted by integers that factor into a small number of perfect powers.  Are related cluster algebraic interpretations possible for Generalized Fortresses, as in \cite[Sec. 3]{LaiNewAspects}, and Needle regions, as defined in \cite[Sec. 5]{LaiNewDungeon}?
\end{problem}

\begin{problem}
In \cite{trim}, the first author conjectured a simple product formula for a certain weighted matching numbers for the $A^{(i)}$ and $F^{(i)}$-graphs (see Conjecture 7.2). Despite providing a weighted enumeration via cluster algebras and Theorem \ref{thm:z3} in our current work, it does \emph{not} look like our result for Model 3 implies this conjecture. It would be interesting to find a different cluster algebra model that could yield the weight assignment as desired in the conjecture.
\end{problem}

\section*{Acknowledgments}
\par{
We are grateful to Richard Eager, Sebastian Franco, Rick Kenyon, and David Speyer for numerous helpful discussions.  We thank the referees for their helpful comments.  We also used the cluster algebras package \cite{Sage2} in Sage \cite{Sage} for numerous computations.  T. L. was supported by Simons Foundation Collaboration Grant \# 585923. G.M. was supported by NSF grants DMS-$1148634$ and DMS-$1362980$.
}

\bibliographystyle{alphaurl}

\newcommand{\etalchar}[1]{$^{#1}$}

\end{document}